\documentclass{amsart}
\usepackage{amsmath,amssymb,amsthm}
\usepackage[utf8]{inputenc}
\usepackage[hyphens]{url}
\usepackage[pagebackref=true]{hyperref}
\usepackage{tikz-cd}
\usepackage{tikz}
\usepgflibrary{arrows}
\usetikzlibrary{shapes.arrows,trees,backgrounds,decorations.markings,shadows}
\usepackage{enumitem}

\setlist[itemize,description]{leftmargin=*}
\theoremstyle{plain}
    \newtheorem{thm}{Theorem}[section]
    \newtheorem{prop}[thm]  {Proposition}
    \newtheorem{lem}[thm]   {Lemma}
    \newtheorem{cor}[thm]   {Corollary}
    \newtheorem{athm}{Theorem}
    
    \newtheorem{alem}[athm]{Lemma}
    \newtheorem{aprop}[athm]{Proposition}
\theoremstyle{definition}
    \newtheorem{defn}[thm]  {Definition}
    \newtheorem{nota}[thm]  {Notation}
    \newtheorem{ex}[thm]    {Example}
    
\newcommand{\bfA}{\mathbf{A}} 
\newcommand{\cA}{\mathcal{A}} 
\newcommand{\C}{\mathbb{C}} 
\newcommand{\fc}{\mathfrak{c}} 
\newcommand{\fe}{\mathfrak{e}} 
\newcommand{\bF}{\mathbb{F}} 
\newcommand{\cJ}{\mathcal{J}} 
\newcommand{\fM}{\mathfrak{M}}
\newcommand{\N}{\mathbb{N}} 
\newcommand{\bQ}{\mathbb{Q}} 
\newcommand{\R}{\mathbb{R}} 
\newcommand{\fR}{\mathfrak{R}} 
\newcommand{\fr}{\mathfrak{r}} 
\newcommand{\fS}{\mathfrak{S}} 
\newcommand{\fu}{\mathfrak{u}} 
\newcommand{\cX}{\mathcal{X}} 
\newcommand{\Z}{\mathbb{Z}} 

\newcommand{\vht}{\mathrm{ht}} 
\newcommand{\tr}{\mathbf{t}} 
\newcommand{\fP}{\mathfrak{P}} 
\usepackage{mathbbol}
\DeclareSymbolFontAlphabet{\mathbb}{AMSb}
\DeclareSymbolFontAlphabet{\mathbbl}{bbold}
\newcommand{\one}{\mathbbl{1}} 
\newcommand{\bfone}{\mathbf{1}} 
\newcommand{\tfS}{\tilde{\fS}} 

\newcommand{\geo}{^{\mathrm{geo}}} 
\newcommand{\conj}{\mathrm{conj}} 
\newcommand{\cG}{\mathcal{G}} 
\newcommand{\cK}{\mathcal{K}} 
\newcommand{\cZ}{\mathcal{Z}} 
\newcommand{\Q}{\mathcal{Q}} 
\newcommand{\hQ}{\hat{\Q}} 
\newcommand{\hfR}{\hat{\fR}} 
\newcommand{\PMQ}{\mathbf{PMQ}} 
\newcommand{\PMR}{\mathbf{PMR}} 
\newcommand{\PMon}{\mathbf{PMon}} 
\newcommand{\Mon}{\mathbf{Mon}} 
\newcommand{\MQ}{\mathbf{MQ}} 
\newcommand{\Grp}{\mathbf{Grp}} 
\newcommand{\Mod}{\mathbf{Mod}} 
\newcommand{\PMQGrp}{\mathbf{PMQGrp}} 
\newcommand{\Qnd}{\mathbf{Qnd}} 
\newcommand{\Set}{\mathbf{Set}} 
\newcommand{\XA}{\mathbf{X}\bfA} 
\newcommand{\XSet}{\mathbf{XSet}} 
\newcommand{\XTop}{\mathbf{XTop}} 
\newcommand{\XMod}{\mathbf{XMod}} 
\newcommand{\CH}{\mathrm{Ch}} 
\newcommand{\Alg}{\mathbf{Alg}} 
\newcommand{\FQ}{\bF\bQ} 
\newcommand{\braiding}{\mathfrak{br}} 
\newcommand{\borel}{/\! /} 

\newcommand{\ab}{\mathbf{ab}} 
\newcommand{\Br}{\mathfrak{Br}} 
\newcommand{\gen}{f} 
\newcommand{\fg}{\mathfrak{g}} 
\newcommand{\ufg}{\underline{\fg}} 
\newcommand{\tfg}{\tilde{\fg}} 
\newcommand{\utfg}{\underline{\tfg}} 
\newcommand{\Hur}{\mathrm{Hur}} 
\newcommand{\hur}{\mathrm{hur}} 
\newcommand{\SP}{\mathrm{S}\mathrm{P}} 
\newcommand{\CmP}{\C\setminus P} 
\newcommand{\totmon}{\omega} 
\newcommand{\fQ}{\mathfrak{Q}} 

\newcommand{\us}{\underline{s}} 
\newcommand{\ut}{\underline{t}} 
\newcommand{\ua}{\underline{a}} 
\newcommand{\Arr}{\mathrm{Arr}} 
\newcommand{\NAdm}{\mathrm{NAdm}} 
\newcommand{\ndeg}{\mathrm{ndeg}} 
\newcommand{\hor}{\mathrm{hor}} 
\newcommand{\ver}{\mathrm{ver}} 
\newcommand{\del}{\partial} 
\newcommand{\ub}{\underline{b}} 
\newcommand{\mDelta}{\mathring{\Delta}} 

\newcommand{\pa}[1]{\left(#1\right)}

\newcommand{\abs}[1]{\left|#1\right|}
\newcommand{\set}[1]{\left\{#1\right\}}
\newcommand{\sca}[1]{[\! [#1]\! ]}

\renewcommand{\phi}{\varphi}
\renewcommand{\epsilon}{\varepsilon}

\DeclareMathOperator{\Imm}{Im}
\DeclareMathOperator{\Ker}{ker}
\DeclareMathOperator{\Id}{Id}

\DeclareMathOperator{\Aut}{Aut} 
\DeclareMathOperator{\Ext}{Ext}

\title{Partially multiplicative quandles and simplicial Hurwitz spaces}
\author{Andrea Bianchi}
\thanks{
This work was partially supported by the \emph{Deutsche
  Forschungsgemeinschaft} (DFG, German Research Foundation) under Germany’s
Excellence Strategy (\texttt{EXC-2047/1}, \texttt{390685813}), by the
\emph{European Research Council} under the \emph{European Union's Horizon 2020
research and innovation programme} (grant agreements No. \texttt{716424} and \texttt{772960}), and by the
\emph{Danish National Research Foundation} through the \emph{Copenhagen Centre for
Geometry and Topology} (\texttt{DNRF151}).
}
\email{anbi@math.ku.dk}
\address{Department of Mathematical Sciences, University of Copenhagen \newline
Universitetsparken 5, Copenhagen, 2100, Denmark}  
\date{\today}

\keywords{Quandle, partial monoid, Hurwitz space, bar construction, free group, symmetric group.}
\subjclass[2020]{
08A05, 
08A35, 
18M15, 
20B05, 
20M05, 
55R80. 
}

\begin{document}
\begin{abstract}
We introduce partially multiplicative quandles (PMQ),
a generalisation of both partial monoids and quandles. We set up the basic theory of PMQs, focusing on the properties of free PMQs and complete PMQs. For a PMQ $\mathcal{Q}$ with completion $\hat{\mathcal{Q}}$, we introduce the category of $\hat{\mathcal{Q}}$-crossed topological spaces, and define the Hurwitz space $\mathrm{Hur}^{\Delta}(\mathcal{Q})$: it is a $\hat{\mathcal{Q}}$-crossed space, and it parametrises $\mathcal{Q}$-branched coverings of the plane. The definition recovers classical Hurwitz spaces when $\mathcal{Q}$ is a discrete group $G$.
Finally, we analyse the class of PMQs $\mathfrak{S}_d^{\mathrm{geo}}$ arising from the symmetric groups $\fS_d$, and we compute their enveloping groups and their PMQ completions.
\end{abstract}

\maketitle

\section{Introduction}
The first goal of the article
is to introduce the notion of \emph{partially multiplicative quandle}
(PMQ), which generalises both notions of quandle and of partial monoid:
\begin{itemize}
 \item a \emph{quandle} is a set with a binary operation behaving like  conjugation in a group; quandles were introduced in \cite{Joyce};
 \item a \emph{partial monoid} is a set with a partially defined binary operation behaving like the product in an associative monoid;
 \item a PMQ is a set $\Q$ with both structures simultaneously, satisfying some compatibility conditions (see Definition \ref{defn:PMQ} for more details).
\end{itemize}
We use PMQs to construct Hurwitz spaces in a categorical way. Classical Hurwitz spaces \cite{Hurwitz, Fulton, FriedVolklein, RomagnyWewers, EVW:homstabhur, ETW:shufflealgebras} are defined by fixing an integer $k\ge0$ and a group $G$: an element of $\hur_k(G)$, which we refer to as a \emph{configuration}, is the datum $(P,\phi)$ of a set $P\subset (0,1)^2\subset\C$ of $k$ points in the open unit square, together with a \emph{$G$-valued monodromy} $\phi$, i.e. a group homomorphism $\phi\colon \pi_1(\CmP)\to G$. The topology on $\hur_k(G)$ is usually defined leveraging the topology of configuration spaces, in such a way that the assignment $(P,\phi)\mapsto P$ gives a covering map
$\hur_k(G)\to C_k((0,1)^2)$, with target the configuration space of $k$ unordered points in $(0,1)^2$. The classical Fox-Neuwirth-Fuchs cell stratification on $C_k((0,1)^2)$ \cite{fn, Fuchs} can then be lifted to a cell stratification on $\hur_k(G)$.

In this article we take an opposite point of view: we define the Hurwitz spaces $\hur_k(G)$, and more generally a Hurwitz space $\Hur^{\Delta}(\Q)$ of configurations $(P,\psi)$ with monodromy in a PMQ $\Q$, by first mimicking the cell stratification. Note the absence of the parameter $k$ in the notation ``$\Hur^{\Delta}(\Q)$'':
using the partial product on the PMQ $\Q$, we can make two or more points of $P$ approach and collide in a controlled way, so that the cardinality of the support of a configuration is no longer locally constant on $\Hur^{\Delta}(\Q)$.

The space $\Hur^{\Delta}(\Q)$ is obtained as difference between the geometric realisations of a certain bisimplicial complex $\Arr(\Q)$ and a subcomplex $\NAdm(\Q)$. It turns out that $\Hur^{\Delta}(\Q)$ is a dense open subspace of $|\Arr(\Q)|$, and vice versa $|\Arr(\Q)|$ gives a natural ``bordification'' of $\Hur^{\Delta}(\Q)$; if $\Q$ is normed and finite (Definition \ref{defn:norm}), then $|\Arr(\Q)|$ is a disjoint union of compactifications of the components of $\Hur^{\Delta}(\Q)$. The decoration ``$\Delta$'' reminds us of the simplicial flavour of the construction, and we will also refer to $\Hur^{\Delta}(\Q)$ as a ``simplicial'' Hurwitz space, to distinguish it from a ``classical'' Hurwitz space of the form $\hur_k(G)$ as above, defined as a certain covering space of $C_k((0,1)^2)$.

The connection between the two notions of Hurwitz spaces is apparent if we consider the PMQ $\Q=G\sqcup\set{\one}$ obtained from a discrete group $G$ by adding a disjoint unit $\one$, in which conjugation is induced by group conjugation in $G$, and the partial product is trivial: in this case we have a homeomorphism $\Hur^{\Delta}(\Q)\cong\coprod_{k\ge0}\hur_k(G)$.\footnote{We add a disjoint unit in order to treat the neutral element of $G$ as a non-unit element of $\Q$ and make the homeomorphism hold; note indeed that a configuration in $\hur_k(G)$ is allowed to have some trivial local monodromies.}

Our motivating examples of PMQs are the family $\fS_d\geo$, for $d\ge1$: the PMQ $\fS_d\geo$ is obtained from the symmetric group $\fS_d$ by keeping the group conjugation, and by restricting the group product to certain pairs of permutations, satisfying a ``geodesic'' requirement. In future work \cite{Bianchi:Hur4}, we will establish a connection between the Hurwitz spaces associated with these PMQs, and the moduli spaces of Riemann surfaces with boundary.

\subsection{Statement of results}
We highlight some of the results of the paper.
The paper begins with the basic theory of PMQs, which is in many respects parallel to classical group theory. To introduced the following, which is Proposition \ref{prop:hQ},
we mention that each PMQ $\Q$ admits a \emph{completion} $\hQ$, which is the initial PMQ receiving a map of PMQs from $\Q$ and having a product defined for all pairs of its elements.
\begin{aprop}
\label{prop:main1}
The canonical map $\Q\to\hQ$ from a PMQ to its completion is injective.
\end{aprop}
This is analogous to the classical result stating that a partial monoid $M$ injects into its monoid completion $\hat M$. To appreciate the non-triviality of the result, we note that, instead, the canonical map from $\Q$ to its enveloping group $\cG(\Q)$ need not be injective: here by \emph{enveloping group of $\Q$} we mean the initial group receiving a map of PMQs from $\Q$.

An important notion, analogous to the one given in \cite[Definition 9.1]{Joyce} in the context of quandles,
is the one of PMQ-group pair: roughly speaking, it is a pair $(\Q,G)$ of a PMQ $\Q$ and a group $G$ that are interrelated with each other (see Definition \ref{defn:PMQgrouppair}). Our main results regarding PMQ-group pairs are concerned with the example
of $(\FQ^k_l,\bF^k)$, for $0\le l\le k$: here $\bF^k$ is the free group on $k$ generators $\gen_1,\dots,\gen_k$, and $\FQ^k_l\subset \bF^k$ is the sub-PMQ containing the neutral element and the conjugacy classes of the first $l$ generators $\gen_1,\dots,\gen_l$ of $\bF^k$. 
The following is an informal rephrasing of Theorem \ref{thm:FQfreePMQ}.
\begin{athm}
 \label{thm:main2}
 The PMQ-group pair $(\FQ^k_l,\bF^k)$ is a free object in the category of PMQ-group pairs, on $l$ generators of PMQ-type and $k-l$ generators of group-type.
\end{athm}
To state the next result, let $0\le r\le l\le k$ and consider the Artin action of the braid group $\Br_r$ on the subgroup $\bF^r\subset\bF^k$ spanned by the first $r$ generators $\gen_1,\dots,\gen_r$ of $\bF^k$;
note that the product $\gen_1\dots\gen_r\in\bF^k$ is fixed by this action. Let $\Br_r$ act on the entire $\bF^k$ by acting trivially on the last $k-r$ generators.
The following is a rephrasing of Proposition \ref{prop:standardmove}.
\begin{aprop}
\label{prop:main3}
The Artin action of $\Br_r$ on $\bF^k$ is transitive on the set of decompositions of the element
$\gen_1\dots\gen_r\in\bF^k$ as a product of $r$ elements in $\FQ^k_l$.
\end{aprop}
These last two results will play an important role in future work \cite{Bianchi:Hur2}, where a geometric construction of Hurwitz spaces with monodromies in a PMQ-group pair is developed. This motivates our interest in PMQ-group pairs.

In analogy with the construction of the group ring $R[G]$ of a group $G$, we introduce the PMQ-ring $R[\Q]$ of a PMQ $\Q$ with coefficients in a commutative ring $R$ (Definition \ref{defn:PMQring}). This ring will play an important role in the computation of the stable homology of Hurwitz spaces in future work \cite{Bianchi:Hur3}. The following is Theorem \ref{thm:RQquadratic}, translating properties of a PMQ into properties of the PMQ-ring.
\begin{athm}
 \label{thm:main4}
 Let $\Q$ be a PMQ with the following three properties: it is maximally decomposable (Definition \ref{defn:maxdecomposable}), coconnected (Definition \ref{defn:coconnected}) and pairwise determined (Definition \ref{defn:pairwisedetermined}). Then $R[\Q]$ is a quadratic $R$-algebra.
\end{athm}
The following is Lemma \ref{lem:cominv}: its proof is quite elementary, but the statement is a little surprising, considering that $R[\Q]$ is in general not a commutative ring.
\begin{alem}
 \label{lem:main5}
Let the enveloping group $\cG(\Q)$ act on $R[\Q]$ by conjugation. Then the subring $\cA(\Q)\subseteq R[\Q]$ of conjugation-invariant elements is a commutative ring.
\end{alem}
After the basics of the theory of PMQs are settled, we turn to Hurwitz spaces.
In order to make the definition of simplicial Hurwitz spaces more conceptual, we introduce a categorical framework.
For a complete PMQ $\hQ$ and a category $\bfA$ we introduce the category $\XA(\hQ)$ of $\hQ$-crossed objects in $\bfA$; this notion generalises the classical notion of $G$-crossed objects, for a group $G$.
The following proposition summarizes the discussion of Subsection \ref{subsec:hQcrossed}, leading in particular to Definition \ref{defn:Xbraiding}.
\begin{aprop}
\label{prop:main6}
Let $\bfA$ be a braided closed monoidal category, and let $\hQ$ be a complete PMQ. Then the category $\XA(\hQ)$ of $\hQ$-crossed objects in $\bfA$ is endowed with a braided monoidal structure. 
\end{aprop}
For an \emph{augmented} PMQ $\Q$ (Definition \ref{defn:augmentedPMQ}) with completion $\hQ$, the machinery of $\hQ$-crossed objects is used to define simplicial Hurwitz spaces with monodromies in a PMQ $\Q$. We first introduce a bisimplicial $\hQ$-crossed set $\Arr(\Q)$, containing a sub-bisimplicial $\hQ$-crossed set $\NAdm(\Q)\subseteq\Arr(\Q)$; the $\hQ$-crossed space $\Hur^{\Delta}(\Q)$ is then defined as the difference between the geometric realisations of $\Arr(\Q)$ and $\NAdm(\Q)$.
The following is Theorem \ref{thm:Hurpi0}, classifying components of $\Hur^{\Delta}(\Q)$.
\begin{athm}
 \label{thm:main7}
Let $\Q$ be an augmented PMQ with completion $\hQ$. Then the unique map of $\hQ$-crossed spaces $\totmon\colon\Hur^{\Delta}(\Q)\to\hQ$ induces a bijection on connected components.
\end{athm}

As an application of simplicial Hurwitz spaces, we consider the case in which $G$ is a group and $\Q:=G\sqcup\set{\one}$, considered as a PMQ with trivial product, unit $\one$ and conjugation extending group conjugation on $G$.
The following is Theorem \ref{thm:HurwitzEVW}.
\begin{athm}
 \label{thm:main8}
 Let $\Q=G\sqcup\set{\one}$ as above; then
 $\Hur^{\Delta}(\Q)\cong\coprod_{k\ge0}\hur_k(G)$.
\end{athm}
In particular, when $\Q=G\sqcup\set{\one}$ as in the previous theorem and $G$ is finite, then $|\Arr(G)|$ splits as a disjoint union of compact spaces $|\Arr(G)_k|$, each containing $\hur_k(G)$ as an open dense subspace: vice versa, this provides $\hur_k(G)$ with a natural compactification $|\Arr(G)_k|$. Moreover, combining Theorems \ref{thm:main7} and \ref{thm:main8}, we obtain a bijection between $\pi_0(\coprod_{k\ge0}\hur_k(G))$ and the completion of the PMQ $G\sqcup\set{\one}$.

Besides the difference space $\Hur^{\Delta}(\Q)=|\Arr(\Q)|\setminus|\NAdm(\Q)|$, we can study the couple of spaces $(|\Arr(\Q)|,|\NAdm(\Q)|)$. In the case in which $\Q$ is \emph{Poincar\'e} (see Definition \ref{defn:Poincare}), the relative homology of this pair agrees with the cohomology of $\Hur^{\Delta}(\Q)$: this justifies our interest in the relative homology of the couple $(|\Arr(\Q)|,|\NAdm(\Q)|)$. The following is Theorem \ref{thm:redchain}.
\begin{athm}
 \label{thm:main9}
 Let $R$ be a commutative ring and $\Q$ be an augmented PMQ. Then the cellular chain complex $\CH_*(|\Arr(\Q)|,|\NAdm(\Q)|;R)$ is isomorphic to the reduced total chain complex associated with the double bar construction $B_{\bullet,\bullet}(R[\Q],R,\epsilon_{\Q})$.
\end{athm}

As already mentioned, our main motivation to introduce PMQs and to generalise Hurwitz spaces comes from the family of PMQs denoted $\fS_d\geo$, for $d\ge1$. The following theorem combines Lemma \ref{lem:cGfSdgeo} and Proposition \ref{prop:hatfSdgeo}.
\begin{athm}
 \label{thm:main10}
Let $d\ge2$. The enveloping group of $\fS_d\geo$ coincides with the index 2 subgroup of $\Z\times\fS_d$ of pairs $(r,\sigma)$ such that $r$ and $\sigma$ have the same parity.

The PMQ completion $\widehat{\fS_d\geo}$ of $\fS_d\geo$ is the set of all sequences
\[
(\sigma;\fP_1,\dots,\fP_\ell;r_1,\dots,r_\ell),
\]
consisting of a permutation $\sigma\in\fS_d$, an unordered partition $\fP_1,\dots,\fP_\ell$ of the set $\set{1,\dots,d}$, and a system of weights $r_1,\dots,r_\ell\ge0$ on the pieces of the partition, satisfying the following properties:
\begin{enumerate}
\item $\sigma$ preserves each piece $\fP_j$ of the partition;
\item $r_j+1-|\fP_j|$ is greater or equal to the number of cycles of the restricted permutation $\sigma|_{\fP_j}\in\fS_{\fP_j}$, for all $1\le j\le \ell$;
\item $r_j$ has the same parity as the restricted permutation $\sigma|_{\fP_j}\in\fS_{\fP_j}$, for all $1\le j\le\ell$.
\end{enumerate}
\end{athm}
Lastly, we highlight the following, which is Proposition \ref{prop:fSgeoKoszul}.
\begin{aprop}
 \label{prop:main11}
 Let $R$ be a commutative ring and $d\ge1$. The quadratic $R$-algebra $R[\fS_d\geo]$ is Koszul.
\end{aprop}

\subsection{Outline of the article}
In Section \ref{sec:pmqbasic} we introduce partially multiplicative quandles (PMQs) and the notions of PMQ-group pair, completion of a PMQ, and enveloping group of a PMQ. We prove Proposition \ref{prop:main1}.

In Section \ref{sec:freegroups} we study the PMQs $\FQ^k_l$ and their relation to the free groups $\bF^k$;
we prove Theorem \ref{thm:main2} and Proposition \ref{prop:main3}.

In Section \ref{sec:normedPMQ} we introduce the notion of \emph{norm} on a PMQ, and study several combinatorial properties
that a PMQ can enjoy. We also introduce the PMQ-ring $R[\Q]$ associated with a PMQ $\Q$. We study how properties of the PMQ $\Q$ are reflected into properties of the ring $R[\Q]$, and in particular we prove Theorem \ref{thm:main4} and Lemma \ref{lem:main5}.

In Section \ref{sec:barconstructions} we describe, for a braided monoidal category
$\bfA$, a procedure taking as input a morphism $f\colon A\to B$ of commutative algebras in $\bfA$,
and giving as output a bisimplicial object in $\bfA$, denoted $B_{\bullet,\bullet}(A,B,f)$ and called the
\emph{double bar construction}. The material of this section is largely standard and is given in detail only to make the article self-contained.

In Section \ref{sec:simplhur} we introduce, for a category $\bfA$ and a complete PMQ $\hQ$, the category of $\hQ$-crossed objects in $\bfA$, and prove Proposition \ref{prop:main6}. Given an \emph{augmented} PMQ $\Q$ with completion $\hQ$, we define the bisimplicial $\hQ$-crossed set $\Arr(\Q)$, its sub-bisimplicial $\hQ$-crossed set $\NAdm(\Q)$, and the simplicial Hurwitz spaces $\Hur^{\Delta}(\Q)$, which is a $\hQ$-crossed space.
We prove Theorems \ref{thm:main7}, \ref{thm:main8} and \ref{thm:main9}.

In Section \ref{sec:fSgeo} we give a detailed analysis of the family of PMQs $\fS_d\geo$. We prove Theorem \ref{thm:main10} and Proposition \ref{prop:main11}. The results of this section are partially contained in \cite[Subsection 8.1.3]{BianchiPhD}.

Finally, in the Appendix \ref{sec:racks}, we discuss how the theory changes if we relax the notion of PMQ to the notion of \emph{partially multiplicative rack}:
in particular we explain the failure of Theorem \ref{thm:FQfreePMQ} in the context of partially multiplicative racks, thus motivating our focus
on the more restrictive notion of PMQ.

\subsection{Motivation}
This is the first article in a series about Hurwitz spaces.
The scope we have in mind, for this and the following articles, is to define and study generalised Hurwitz spaces $\Hur(\cX;\Q)$ of finite configurations of points in a subspace $\cX\subset\C$ endowed with monodromies in a PMQ $\Q$.

In this article we achieve a first, \emph{simplicial} definition of a generalised Hurwitz space, the space $\Hur^{\Delta}(\Q)$; this space comes with a stratification by open cells
which reminds of the Fox-Neuwirth-Fuchs stratification of the classical configuration spaces $C_k(\C)$ of the plane; we remark that a similar cell stratification
has been used in \cite{ETW:shufflealgebras} in the study of classical Hurwitz spaces.

The simplicial definition of Hurwitz spaces is, in a certain sense, not
\emph{coordinate-free}, in the sense that it uses dramatically the two standard, Euclidean coordinates of the plane $\C$.
One of the main achievements of the second article of the series \cite{Bianchi:Hur2} will be a coordinate-free definition of
the generalised Hurwitz spaces, allowing for more flexible manipulations. The results of this article will give the algebraic input for the second article.

The third article in the series \cite{Bianchi:Hur3} will use the algebraic input of the first article, together with the topological input of the second, to study Hurwitz spaces as topological monoids, and compute their deloopings.

Finally, the fourth article in the series \cite{Bianchi:Hur4} will apply the entire machinery of generalised Hurwitz spaces to the study of moduli spaces $\fM_{g,n}$ of Riemann surfaces of genus $g$ with $n\ge1$ ordered and parametrised boundary curves: the PMQs $\fS_d\geo$ will play a prominent role and motivate the very definition of PMQs, but we prefer to set up the theory more generally to allow other applications and to make the exposition more conceptual.

The attempt to generalise Hurwitz spaces to PMQs should also be seen as an attempt to unify two classical notions in topology:
classical Hurwitz spaces, and configuration spaces with summable labels. Both notions deal with \emph{decorated} configurations of points in a background
space $\cX$. The \emph{decoration} of a configuration $P\subset\cX$ is:
\begin{itemize}
 \item for classical Hurwitz spaces $\hur_k(G)$, a monodromy with values in a group $G$ and defined on certain loops in $\CmP$ (here $\cX$ is taken to be $\C$ or a subspace of it, e.g. the square $(0,1)^2$);
 \item a labeling with values in an abelian group $G$ (or more generally in a partial abelian monoid $M$) and defined on the points of $P$,
 for configuration spaces with summable labels (here $\cX$ can be any topological space).
\end{itemize}
Neither of these two classical notions is more general than the other; in particular there are two aspects under which the two notions can be compared:
\begin{itemize}
 \item a commutativity property is required for the labels of a configuration space with summable
 labels (i.e. the group or partial monoid must be abelian); differently, classical Hurwitz spaces can take monodromies in a \emph{non-abelian} group (or, more generally, a non-abelian quandle or even in a rack);
 \item collisions between points of a configuration are not allowed in the classical setting of Hurwitz
 spaces; differently, collisions between points of a configuration with summable labels is allowed whenever one can replace the old labels by their sum.
\end{itemize}
The ``intersection'' of these two classical notions only contains configuration spaces of points in $\C$ with labels in a set $S$.

\subsection{A brief history of Hurwitz spaces}
The notion of Hurwitz spaces goes back to Clebsch \cite{Clebsch} and Hurwitz \cite{Hurwitz}. For
a fixed $k\geq 0$ and a discrete group $G$, the classical Hurwitz space $\hur_k(G)$
contains configurations of the form $(P,\phi)$, where
\begin{itemize}
 \item $P=\set{z_1,\dots,z_k}$ is a collection of $k$ distinct points in $\C$;
 \item $\phi\colon\pi_1(\CmP)\to G$ is a group homomorphism.
\end{itemize}
A homotopy theoretic characterisation, which can be found in \cite[Subsection 1.3]{EVW:homstabhur}
and \cite[Section 4]{ORW:Hurwitz}, is the following. There is a natural
action of the braid group $\Br_k$ on the set $G^k$: the standard generator $\mathfrak{b}_i$,
for $1\leq i\leq k-1$, sends the $k$-tuple $(g_1,\dots,g_k)$ to the $k$-tuple
$ (\,g_1\,,\,\dots\,,\,g_{i-1}\,,\,g_{i+1}\,,\, g_{i+1}^{-1}g_ig_{i+1}\,,\,g_{i+2}\,,\,\dots,g_k\,)$.
The classical Hurwitz space $\hur_k(G)$ is then homotopy equivalent to the homotopy quotient $G^k\borel\Br_k$.
This homotopic definition of Hurwitz spaces admits a straightforward extension to the case in which the group $G$ is replaced by a \emph{rack}. Recall
that a rack is a set $\fR$ with a binary operation $\fR\times \fR\to \fR$,
$(a,b)\mapsto a^b$, satisfying the relation $(a^b)^c=(a^c)^{b^c}$ for all $a,b,c\in \fR$; then $\Br_k$ acts on $\fR^k$, by letting
$\mathfrak{b}_i\in \Br_k$ send $(g_1,\dots,g_k)$
to $(g_1,\dots g_{i-1}, g_{i+1}, g_i^{g_{i+1}},g_{i+2},\dots,g_k)$.
The most familiar example of rack is a group $G$, and the reader
will note that, according to this second description, the homotopy type of $\hur_k(G)$ only depends on the underlying structure of rack that a group $G$ has, and not, for instance, on the product of $G$.

The notion of \emph{quandle} is slightly more restrictive than the notion of rack:
a rack $\fR$ is a quandle if $a^a=a$ for all $a\in \fR$. See Definition \ref{defn:quandle} for more details,
\cite{FennRourke} for a classical account on the history of these notions, and the recent preprints
\cite{DehnQuandles} and \cite{FreeRacks} for an updated account.

Classical Hurwitz spaces have been shown to admit a structure of algebraic variety, and have been employed
to study the geometry of the moduli space of curves $\fM_g$ in different characteristics (see \cite{RomagnyWewers}
for an account on the history of applications of Hurwitz spaces in algebraic geometry).
More recently, Hurwitz spaces have been employed
as a topological tool to obtain results about the Cohen-Lenstra heuristics and Malle's conjecture over function fields
\cite{EVW:homstabhur, ETW:shufflealgebras, ORW:Hurwitz}.

Finally, a version of Hurwitz spaces with coefficients in a \emph{space} has been introduced
in \cite{EVW:homstabhurII}, and further investigated in \cite{PalmerTillmann} and \cite{PalmerTillmannbis}.

\subsection{A brief history of configuration spaces with summable labels}
Another classical notion is that of configuration space with summable labels $C(\cX;M)$,
depending on a topological space $\cX$ and on a partial abelian monoid $M$.

This notion was originally considered by Dold and Thom \cite{DT56} in the special case in which $M$ is an abelian group, under the
familiar name of \emph{symmetric product}; later McCord \cite{McCord69} considered the case of an abelian monoid, Kallel \cite{Kallel01} the case of a partial abelian monoid,
and Salvatore \cite{Salvatore01} the case of a partial $E_n$-algebra, in the assumption that $\cX$ is a framed $n$-manifold.
Similar, classical constructions occur in \cite{Segal73} and \cite{McDuff75}.

Configuration spaces with labels in a \emph{space} (without a partial abelian structure)
also appear in the literature \cite{May, Snaith, Boedigheimer87}.

\subsection{Acknowledgments}
This series of articles is a generalisation
and a further development of my PhD thesis \cite{BianchiPhD}. I am grateful to
my PhD supervisor
Carl-Friedrich B\"odigheimer,
Bastiaan Cnossen,
Florian Kranhold,
Luigi Pagano,
Oscar Randal-Williams,
Lukas Woike,
Nathalie Wahl
and the anonymous referees of this article and of \cite{Bianchi:Hur4}
for helpful comments, mathematical explanations and reference suggestions
related to this article.

\tableofcontents
\part{Algebraic theory of partially multiplicative quandles}
\section{Partially multiplicative quandles and their relation to groups}
\label{sec:pmqbasic}
We usually denote by $G$ a discrete group, with neutral element $\one=\one_G$.
\subsection{Basic definitions and first examples}
\label{subsec:basicsPMQ}
\begin{defn}
 \label{defn:quandle}
A \emph{quandle (with unit)} is a set $\Q$ with a marked element $\one\in\Q$, called \emph{unit},
and a binary operation $\Q\times \Q\to \Q$, denoted $(a,b)\mapsto a^b$
and called \emph{conjugation}, such that:
\begin{enumerate}
 \item for all $a\in\Q$ the map $(-)^a\colon \Q\to\Q$ is bijective;
 \item for all $a\in\Q$ we have $\one^a=\one$ and $a^{\one}=a$;
 \item for all $a\in\Q$ we have $a^a=a$;
 \item for all $a,b,c\in \Q$ we have $(a^b)^c=(a^c)^{(b^c)}$.
\end{enumerate}
We denote by $(-)^{a^{-1}}\colon\Q\to\Q$ the inverse map of $(-)^a\colon\Q\to\Q$.

The \emph{conjugacy class} of an element $a\in\Q$, denoted by $\conj(a)$, is the smallest subset $S\subset\Q$
which contains $a$ and is closed under the operations $(-)^b$ and $(-)^{b^{-1}}$ for all $b\in\Q$.
We denote by $\conj(\Q)$ the set of conjugacy classes of $\Q$.

A quandle $\Q$ is \emph{abelian},
if $(-)^c$ is the identity of $\Q$ for all $c\in \Q$.

A morphism of quandles is a morphism of the underlying sets that preserves unit and conjugation.
Quandles form a category $\Qnd$.
\end{defn}
Note that for all $c$ in a quandle $\Q$, the map $(-)^c\colon \Q\to \Q$ is automatically
an automorphism of $\Q$ as a quandle.

The usual definition of ``quandle'' in the literature differs from Definition \ref{defn:quandle}
in that no unit $\one\in\Q$ is required, and condition (2) is dropped.
Note however that if $\Q$ is a quandle without unit, then the set $\Q\sqcup\set{\one}$ can be given a unique
structure of quandle, such that $\one$ is the unit and the inclusion $\Q\subset\Q\sqcup\set{\one}$
preserves conjugation. Throughout the article we will use the word ``quandle'' in the sense of Definition \ref{defn:quandle}.
If instead we drop condition (3) from Definition \ref{defn:quandle}, we obtain
the classical definition of \emph{rack} (with unit).
In the Appendix \ref{sec:racks}
we will briefly discuss the possibility to extend the results of this article to the generality
of racks, and we will describe what difficulties arise.

\begin{ex}
 \label{ex:QinGconjinvquandle}
Let $G$ be a group with unit $\one$, and let $\one\in\Q\subseteq G$ be a conjugation
invariant subset. Then $\Q$ is a quandle by setting $a^b:=b^{-1}ab$, for all $a,b\in \Q$.
\end{ex}

Note that an abelian quandle only contains the information of its underlying pointed set: more precisely, there is a fully faithful functor
$\Set_*\to\Qnd$ with essential image given by abelian quandles.

\begin{defn}
 \label{defn:partialmonoid}
 A \emph{partial monoid} $M$ is a set $M$ with a marked element $\one\in M$, called \emph{unit}, a subset $D\subseteq M\times M$
 and a map $D\to M$, denoted $(a,b)\mapsto ab$ and called \emph{partial product}. We say that \emph{the product $ab$ is defined}
 if $(a,b)\in D$. The following properties must hold:
 \begin{enumerate}
  \item for all $a\in M$ both $\one a$ and $a\one$ are defined and equal to $a$;
  \item for all $a,b,c\in M$, each of the following conditions is satisfied if and only if the other is satisfied:
  \begin{itemize}
   \item $ab$ is defined and $(ab)c$ is defined;
   \item $bc$ is defined and $a(bc)$ is defined.
  \end{itemize}
  Moreover, whenever both conditions are satisfied, we further have $(ab)c=a(bc)$.
 \end{enumerate} 
 A partial monoid $M$ is \emph{abelian} if
 for all $a,b\in M$ either of the following holds:
 \begin{itemize}
  \item both products $ab$ and $ba$ are not defined;
  \item both products $ab$ and $ba$ are defined, and $ab=ba$. 
 \end{itemize}
 A partial monoid $M$ has \emph{trivial} product if for all $a,b\in M$
 the product $ab$ is defined if and only if at least one between $a$ and $b$ is equal to $\one$.
 
 A morphism of partial monoids $M\to M'$ is a map of the underlying sets sending $\one_{M}\mapsto \one_{M'}$
 and the following holds: whenever $a\mapsto a'$,
 $b\mapsto b'$ and $ab$ is defined in $M$, then $a'b'$ is defined in $M'$ and $ab\mapsto a'b'$.
\end{defn}
We can amalgamate Definitions \ref{defn:quandle} and \ref{defn:partialmonoid} into the following one.
\begin{defn}
 \label{defn:PMQ}
 A partially multiplicative quandle (PMQ) is a set $\Q$ with a marked element $\one\in\Q$, called \emph{unit},
 such that $\Q$ is both a quandle and a partial monoid, the unit is $\one$ in both cases and
 for all $a,b,c\in \Q$ the following equalities hold:
 \begin{enumerate}
  \item $ab$ is defined if and only if $b(a^b)$ is defined, and whenever both $ab$ and $b(a^b)$ are defined we have $ab=b(a^b)$; we usually write $ba^b$ for $b(a^b)$;
  \item $a^{(bc)}=(a^b)^c$, whenever the product $bc$ is defined;
  \item $ab$ is defined if and only if $(a^c)(b^c)$ is defined, and whenever both $ab$ and $(a^c)(b^c)$ are defined
  we have $(ab)^c=(a^c)(b^c)$.
 \end{enumerate}
 A PMQ $\Q$ is abelian if the underlying quandle is abelian: condition (1) implies
 that also the underlying partial monoid of $\Q$ is abelian.
 A PMQ has \emph{trivial} product 
 if the underlying partial monoid has trivial product.
 
 A morphism of PMQs is a map of sets that is both a morphism of quandles and of
 partial monoids; the category of PMQs is denoted $\PMQ$.
\end{defn}

\begin{ex}
\label{ex:group}
 Let $G$ be a group; then $G$ is a PMQ by setting $a^b=b^{-1}ab$ for all $a,b\in G$ and by using the group product, defined on the entire $G\times G$; the unit is $\one\in G$.
 This construction defines a forgetful functor $\Grp\to\PMQ$ from the category of groups to the category of PMQs.
\end{ex}

\begin{ex}
\label{ex:quandle}
Let $\Q$ be a quandle; then $\Q$ is a PMQ with trivial product. This construction
 defines a functor $\Qnd\to\PMQ$, which is the left adjoint to the forgetful functor $\PMQ\to\Qnd$ forgetting the partial product; in particular this construction makes every pointed set $S$ into an abelian PMQ with trivial product.
\end{ex}

\begin{ex}
\label{ex:PAM}
Let $M$ be a partial \emph{abelian} monoid; then $M$ is an abelian PMQ by setting $a^b=a$ for all $a,b\in M$; this construction gives an equivalence between the category of abelian PMQs and the category of partial abelian monoids.
\end{ex}

The following definition gives a method to obtain PMQs as subsets of groups.
\begin{defn}
 \label{defn:fullsubPMQ}
 Let $G$ be a group and let $\one\in S\subseteq G$ be a conjugation invariant subset of $G$ satisfying the following
 property: for all $1\le i<j\le r$, if $a_1,\dots,a_r$ are elements in $S$ and if the product $a_1\dots a_r$ lies in $S$,
 then also the product $a_i\dots a_j$ lies in $S$ (this property is to ensure condition (2) from Definition \ref{defn:partialmonoid} after we define the partial product on $S$).
 Then $S$ \emph{inherits from $G$} a structure of PMQ as follows:
 \begin{itemize}
  \item the unit is $\one$;
  \item the conjugation is defined as in $G$; 
  \item given two elements $a,b\in S$, if their product $ab\in G$ lies in $S$, then we declare $ab$ to be
  also their product in $S$ as PMQ; otherwise the product $ab$ is not defined.
 \end{itemize}
\end{defn}
Note that if $S\subseteq G$ inherits from $G$ the structure of PMQ, then the inclusion
$S\hookrightarrow G$ is a map of PMQs, where $G$ is a PMQ as in Example \ref{ex:group}.

\subsection{Enveloping group of a PMQ}
Conversely as in Example \ref{ex:group}, we can construct a group from a PMQ as follows.
\begin{defn}
 \label{defn:cGQ}
 Let $\Q$ be a PMQ. We define its \emph{enveloping group} $\cG(\Q)$ as the group with the following presentation.
 
 \begin{tabular}{ll}
 \textbf{Generators} & For all $a\in \Q$ there is a generator $[a]$.\\
 \textbf{Relations} & $\bullet$ $[b]^{-1}[a][b]=[a^b]$ for all $a,b\in \Q$.\\
  & $\bullet$ $[a][b]=[ab]$ for all $a,b\in \Q$ such that $ab$ is defined in $\Q$.
\end{tabular}

The assignment $\Q\mapsto \cG(\Q)$ gives the left adjoint of the forgetful functor $\Grp\to\PMQ$ from Example \ref{ex:group}.
We denote by $\eta=\eta_{\Q}\colon \Q\to \cG(\Q)$ the unit of the adjunction: it is the map of PMQs defined by $a\mapsto[a]$
for all $a\in \Q$.
\end{defn}
In general $\eta$ is not injective, as we see in the following (compare also with \cite[Section 6]{Joyce}).
\begin{defn}
 \label{defn:GltimesS}
 Let $G$ be a group acting on right on a set $S$. We define a PMQ $G\ltimes S$. The underlying set is $\set{\one}\sqcup G\sqcup S$; the neutral element is $\one$; the partial product $ab$ is only defined if $a$ or $b$ is $\one$ (in which case its value is determined by $\one$ being the unit), or if both $a,b\in G$; the quandle structure is given as follows:
 \begin{itemize}
  \item  $a^s=a$ for all $a\in G\ltimes S$ and $s\in S$;
  \item $h^g=g^{-1}hg$ for all $g,h\in G$;
  \item $s^g=s\cdot g$ for all $g\in G$ and $s\in S$;
  \item $\one^a=\one$ and $a^\one=a$ for all $a\in G\ltimes S$.
 \end{itemize}
\end{defn}
Note that for $g\in G$ and $s\in S$ the following equalities hold in $\cG(G\ltimes S)$:
\begin{itemize}
 \item $[g]^{-1}[s][g]=[s\cdot g]$
 \item $[s]^{-1}[g][s]=[g]$.
\end{itemize}
Putting them together one obtains the equality $[s]=[s\cdot g]\in\cG(G\ltimes S)$. Hence the map
$\eta\colon G\ltimes S\to\cG(G\ltimes S)$ identifies the elements $s$ and $s\cdot g$ of $S$,
and is not injective unless $G$ acts trivially on $S$.

\subsection{Adjoint action and PMQ-group pairs}
\label{subsec:adjointaction}
\begin{nota}
 \label{nota:AutPMQ}
For a PMQ $\Q$ we denote by $\Aut_{\PMQ}(\Q)$ the group of automorphisms of $\Q$ as a PMQ;
we use the classical convention that automorphisms, as functions in general, act on left.
We denote by $\Aut_{\PMQ}(\Q)^{op}$ the opposite group, whose elements are still those functions of sets $\Q\to\Q$ that are automorphisms of PMQs (and such functions can be evaluated on elements of $\Q$), but whose composition is reversed. 
\end{nota}
By definition of PMQ there is a map of PMQs $\Q\to \Aut_{\PMQ}(\Q)^{op}$ given by $a\mapsto (-)^a$. This map gives rise to a homomorphism of groups $\rho\colon\cG(\Q)\to\Aut_{\PMQ}(\Q)^{op}$, i.e. to a right action of $\cG(\Q)$ on $\Q$: we call this the \emph{adjoint
action}. Note that the map $\eta\colon \Q\to\cG(\Q)$ (see Definition \ref{defn:cGQ})
is $\cG(\Q)$-equivariant if we consider the right action of $\cG(\Q)$ on itself by conjugation.
\begin{nota}
 \label{nota:cK}
 We denote by $\cK(\Q)\subseteq\cG(\Q)$ the kernel $\Ker(\rho)$.
\end{nota}
\begin{lem}
 \label{lem:cKcentercG}
The subgroup $\cK(\Q)$ is contained in the centre of $\cG(\Q)$.
\end{lem}
\begin{proof}
 Let $g\in \cG(\Q)$ be an element with $\rho(g)=\Id_\Q$. Then conjugation by $g$ fixes the image
 of $\eta\colon \Q\to\cG(\Q)$, which contains all generators of $G$: hence conjugation by $g$ fixes $\cG(\Q)$,
 i.e., $g$ is central in $\cG(\Q)$.
\end{proof}
In general equality does not hold, as we see in the following example. Let $G$ be a group acting on a set $S$.
Using Definitions \ref{defn:cGQ} and \ref{defn:GltimesS}, it is immediate to see that $\cG(G\ltimes S)$
is isomorphic to $G\times \bigoplus_{S/G}\Z$, i.e., the direct product of $G$ and the free abelian
group on the orbits of the action of $G$ on $S$.
If we assume that $G$ has non-trivial centre $\cZ(G)$ and acts faithfully on $S$, then $\cK(G\ltimes S)$ is the subgroup $\bigoplus_{S/G}\Z$, whereas the centre $\cZ(\cG(G\ltimes S))$
is the strictly larger subgroup $\cZ(G)\times \bigoplus_{S/G}\Z$.

Consider now the special case of a \emph{finite} PMQ $\Q$; then the group $\Aut_{\PMQ}(\Q)$ is also finite, and therefore we have a central extension
of groups with finite cokernel
\[
 \begin{tikzcd}
  1 \ar[r] & \cK(\Q) \ar[r] & \cG(\Q) \ar[r] & \Imm(\rho) \ar[r] & 1.
 \end{tikzcd}
\]

Since $\cG$ is a functor, it transforms automorphisms of PMQs into automorphisms of groups, hence $\cG$
gives a map of groups $\Aut_{\PMQ}(\Q)\to\Aut_{\Grp}(\cG(\Q))$. We obtain the following lemma.
\begin{lem}
 \label{lem:mainsequence}
There is a natural sequence of maps of PMQs:
\[
 \begin{tikzcd}
  \Q\ar[r,"\eta"] & \cG(\Q)\ar[r,"\rho"] & \Aut_{\PMQ}(\Q)^{op}\ar[r,"\cG"] &\Aut_{\Grp}(\cG(\Q))^{op}.
 \end{tikzcd}
\]
\end{lem}

The following definition generalises the situation of Lemma \ref{lem:mainsequence}. Compare
also with \cite[Definition 9.1]{Joyce}.
\begin{defn}
 \label{defn:PMQgrouppair}
 A \emph{PMQ-group pair} consists of a PMQ $\Q$, a group $G$, a map of PMQs $\fe\colon\Q\to G$ and
 an right action $\fr\colon G\to \Aut_{\PMQ}(\Q)^{op}$ of $G$ on $\Q$, such that the composition
 $\fr\circ\fe\colon\Q\to\Aut_{\PMQ}(\Q)^{op}$ is equal to the map $\rho\circ\eta$,
 and such that the map $\fe$ is $G$-equivariant if $G$ acts on $\Q$ by $\fr$ and on itself by right group conjugation.

 We usually denote by $(\Q,G,\fe,\fr)$ a PMQ-group pair, or just by $(\Q,G)$, leaving the maps $\fe$ and $\fr$ implicit. A map of PMQ-group pairs $(\Q,G,\fe,\fr)\to(\Q',G',\fe',\fr')$ is given by a pair $(\Psi,\Phi)$,
 where $\Psi\colon \Q\to \Q'$ is a map of PMQs and $\Phi\colon G\to G'$ is a map of groups, such that the following diagrams of PMQs commute, the second for all $g\in G$:
\[
 \begin{tikzcd}
  \Q\ar[r,"\fe"]\ar[d,"\Psi"'] & G \ar[d,"\Phi"]& & \Q\ar[r,"\fr(g)"]\ar[d,"\Psi"'] & \Q \ar[d,"\Psi"]\\
  \Q'\ar[r,"\fe'"] & G',& & \Q'\ar[r,"\fr'(\Phi(g))"] & \Q'.
 \end{tikzcd}
\]
We obtain a category $\PMQGrp$ of PMQ-group pairs.
\end{defn}
By Lemma \ref{lem:mainsequence}, $(\Q,\cG(\Q),\eta,\rho)$ is a PMQ-group pair for all PMQ $\Q$; moreover, since $\rho$ factors through the quotient by $\cK(\Q)$ (see Notation \ref{nota:cK}), we also have that
$(\Q,\cG(\Q)/\cK(\Q))$ is naturally a PMQ-group pair. Finally, if $S$ and $G$ are as in Definition
\ref{defn:fullsubPMQ}, then $(S,G)$ is a PMQ-group pair in a natural way.

\begin{defn}
 \label{defn:prodPMQgrouppairs}
 Let $(\Q,G,\fe,\fr)$ and $(\Q',G',\fe',\fr')$ be PMQ-group pairs.
 We define the product $(\Q,G)\times(\Q',G')$ as the PMQ-group pair $(\Q\times\Q',G\times G',\fe\times\fe',\fr\times\fr')$,
 where:
 \begin{itemize}
  \item the conjugation on $\Q\times\Q'$ is defined component-wise, and the product
  $(a,a')(b,b')$ is defined if and only if the products $ab\in\Q$ and $a'b'\in\Q'$ are defined: in this case
  $(a,a')(b,b')=(ab,a'b')$;
  \item $G\times G'$ is given the product group structure, and $\fe\times\fe'\colon(a,a')\mapsto(\fe(a),\fe'(a'))$;
  \item for $(g,g')\in G\times G'$ we
  set $\fr\times\fr'\colon (g,g')\mapsto \Big((a,a')\mapsto (\fr(g)(a),\fr(g')(a'))\Big)$.
 \end{itemize}
The product $(\Q,G)\times(\Q',G')$ is the categorical product of $(\Q,G)$ and $(\Q',G')$ in $\PMQGrp$,
and its projections are denoted
$p_{(\Q,G)}\colon(\Q,G)\times(\Q',G')\to(\Q,G)$ and
$p_{(\Q',G')}\colon(\Q,G)\times(\Q',G')\to(\Q',G')$.
\end{defn}
\begin{nota}
 \label{nota:oneone}
 We will denote by $(\one,\one)$ the unique PMQ-group pair $(\Q,G)$ for which both $\Q$ and $G$
 consist of the only element $\one$.
\end{nota}

\subsection{Complete PMQs}
\label{subsec:completePMQ}
\begin{defn}
  \label{defn:PMQcomplete}
 A PMQ $\Q$ is \emph{complete} if the product
 is defined for all pairs of elements. Complete PMQs are also called \emph{multiplicative quandles} and form a full subcategory $\MQ\subset\PMQ$. 
 A PMQ-group pair $(\Q,G)$ is complete if $\Q$ is complete.
 \end{defn}

\begin{defn}
\label{defn:hQ}
The inclusion functor $\MQ\to \PMQ$ admits a left adjoint: given a PMQ $\Q$, we construct its \emph{completion}
$\hQ$ as follows:
\begin{itemize}
 \item as a monoid, $\hQ$ is freely generated by elements $\hat a$ for $a\in\Q$, with generating relations
 given by $\hat{\one}=\one$, $\hat a\hat b=\widehat{ab}$ whenever the product $ab$ is defined in $\Q$,
 and $\hat a\hat b=\hat b\widehat{a^b}$ for all $a,b\in\Q$;
 \item there is a natural map of partial monoids $\Q\to\hQ$ given by $a\mapsto\hat a$; the map of partial monoids
 $\rho\circ\eta\colon \Q\to\Aut_{\PMQ}(\Q)^{op}$ extends to a map of monoids $\hQ\to\Aut_{\PMQ}(\Q)^{op}$;
 we can compose the latter with the natural maps of groups $\Aut_{\PMQ}(\Q)^{op}\hookrightarrow\Aut_{\PMon}(\Q)^{op}
 \to\Aut_{\Mon}(\hQ)^{op}\hookrightarrow\Aut_{\Set}(\hQ)^{op}$;
 \item the adjoint $\hQ\times\hQ\to\hQ$ of the map $\hQ\to\Aut_{\Set}(\hQ)^{op}$ makes $\hQ$ into a quandle;
 all axioms of PMQ are satisfied.
\end{itemize}
\end{defn}

Our next aim is to prove that the map of PMQs $\Q\to\hQ$ is injective, and the subset $\hQ\setminus\Q\subset\Q$ is an
ideal, in the following sense.
\begin{defn}
  \label{defn:PMQideal}
  Let $\Q$ be a PMQ and let $I\subset\Q$ be a subset. We say that $I$ is an \emph{ideal}
  if the following hold:
  \begin{itemize}
   \item $I$ is conjugation invariant, i.e. for all $a\in I$ and $b\in\Q$ we have $a^b,a^{b^{-1}}\in I$;
   \item $I$ absorbs products when they are defined, i.e. for all $a\in I$ and $b\in\Q$,
   if $ab$ is defined then it lies in $I$, and if $ba$ is defined then it lies in $I$.   
  \end{itemize}
\end{defn}
Note that if $\Psi\colon\Q\to\Q'$ is a map of PMQs and $I'\subset\Q'$ is an ideal, then also $I:=\Psi^{-1}(I')\subset\Q$
is an ideal.
\begin{defn}
 \label{defn:PMQgrouppairJoin}
 Let $(\Q,G,\fe,\fr)$ be a PMQ-group pair: we define a new PMQ-group pair $(\Q\Join G,G,\bar\fe,\bar\fr)$.
 As a set, we put $\Q\Join G=\Q\sqcup G$; for $a\in\Q$ and $g\in G$ we let $\bar a$ and $\bar g$
 be the corresponding elements in $\Q\Join G$.
 
 To define conjugation in $\Q\Join G$, for $a,b\in \Q$ and $g,h\in G$ we set $\bar a^{\bar b}=\overline{a^b}$,
 $\bar a^{\bar g}=\overline{\fr(g)(a)}$, $\bar g^{\bar a}=\overline{\fe(a)^{-1}g\fe(a)}$ and $\bar g^{\bar h}=\overline{h^{-1}gh}$.

 The product of the PMQ $\Q\Join G$ is defined for all pairs of elements:  for $a,b\in \Q$ and $g,h\in G$ we set
 $\bar g\bar h=\overline{gh}$, $\bar a\bar g=\overline{\fe(a)g}$, $\bar g\bar a=\overline{g\fe(a)}$; moreover,
 if $ab$ is defined in $\Q$ we set $\bar a\bar b=\overline{ab}$, otherwise we set $\bar a\bar b=\overline{\fe(a)\fe(b)}$.
 The unit of $\Q\Join G$ is $\bar{\one}\in\Q$.
 
 For $a\in\Q$ and $g,h\in G$, the map $\bar\fe\colon\Q\Join G\to G$ is given by $\bar\fe(\bar a)=\fe(a)$ and $\bar\fe(\bar g)=g$;
 the automorphism $\bar\fr(g)\in\Aut_{\PMQ}(\Q\Join G)^{op}$ sends $\bar a\mapsto \overline{\fr(g)(a)}$ and $\bar h\mapsto \overline{g^{-1}hg}$.

\end{defn}
Let now $\Q$ be any PMQ and fix a PMQ-group pair of the form $(\Q,G)$, e.g. $(\Q,\cG(\Q))$.
Note that the natural inclusion $\Q\subset\Q\Join G$ is a map of PMQs, extending with the identity of $G$ to a map of PMQ-group
pairs $(\Q,G)\to(\Q\Join G,G)$. Note also that $\Q\Join G$ is a complete PMQ; therefore the inclusion $\Q\subset\Q\Join G$ induces
a map of complete PMQs $\Psi\colon\hQ\to\Q\Join G$. Since the composition $\Q\to\hQ\overset{\Psi}{\to}\Q\Join G$ is injective, also $\Q\to\hQ$ is injective, so we can regard $\Q$ as a subset of $\hQ$.

Moreover $G\subset \Q\Join G$ is an ideal, hence $\Psi^{-1}(G)\subset\hQ$ is also an ideal. We claim that $\Psi^{-1}(G)=\hQ\setminus\Q$. Since $\Psi(\Q)=\Q$, we have the inclusion $\Psi^{-1}(G)\subseteq\hQ\setminus\Q$. On the other hand
an element $w\in\hQ\setminus\Q$ can be represented by a word $\hat a_1\dots\hat a_r$ such
that the product $a_1\dots a_r$ is not defined in $\Q$ (otherwise $w$ would lie in $\Q\subset\hQ$). By definition
of the product on $\Q\Join G$ we have $\Psi(w)\in G$. We obtain the following proposition.
\begin{prop}
 \label{prop:hQ}
 Let $\Q$ be a PMQ. Then the natural map of PMQs $\Q\to\hQ$ is injective, and $\hQ\setminus\Q$ is an ideal of $\hQ$.
\end{prop}
To conclude the subsection, note that if $\Q$ is already a complete PMQ, then the natural map of PMQs $\Q\to\hQ$
is an isomorphism. In particular every complete PMQ is in the essential image of the completion functor $\PMQ\to\MQ$, and thus, whenever we want to consider a complete PMQ, we can assume that it has the form $\hQ$ for some PMQ $\Q$. For this reason we shall often abuse notation and denote by $\hQ$ a generic multiplicative quandle, even if no ``underlying'' PMQ $\Q$ is specified, whose completion is $\hQ$.

\begin{nota}
 \label{nota:JQ}
 For a PMQ $\Q$ we denote by $\cJ(\Q)$ the ideal $\hQ\setminus\Q$ of $\hQ$.
 Since the natural map $\Q\to\hQ$ is injective, we will often abuse notation and denote by $a\in\hQ$ 
 the element $\hat a$ corresponding to $a\in\Q$.
\end{nota}

\section{Free groups and associated PMQs}
\label{sec:freegroups}
In this section we study certain PMQs arising as subsets of free groups.
\subsection{Free sub-PMQs}
We fix natural numbers $0\leq l\leq k$ throughout the section.
\begin{nota}
\label{nota:freegroup}
 We denote by $\bF^k$ the free group on $k$ generators $\gen_1,\dots,\gen_k$. The abelianisation
 of $\bF^k$ is identified with $\Z^k$, generated by the classes of the generators $\gen_1,\dots,\gen_k$.
 We denote by $\ab\colon\bF^k\to\Z^k$ the abelianisation map.
\end{nota}
\begin{defn}
 \label{defn:freePMQ}
 Let $0\leq l\leq k$. We denote by $\FQ^k_l\subset\bF^k$ the union of $\set{\one}$ and the conjugacy classes of the generators
 $\gen_1,\dots,\gen_l$. The set $\FQ^k_l$ inherits from $\bF^k$ the structure of PMQ in the sense of Definition \ref{defn:fullsubPMQ}, and we call it the \emph{free sub-PMQ of $\bF^k$ on the first $l$ generators}.
\end{defn}
To check that the conditions from Definition \ref{defn:fullsubPMQ} are satisfied, note that
each non-unit element of the set $\FQ^k_l$ is mapped under the map $\ab$ to a vector in $\Z^k$
with one entry (among the first $l$) equal to 1, and all other entries equal to 0: hence
the product in $\bF^k$ of two or more non-trivial elements in $\FQ^k_l$ does not lie in $\FQ^k_l$,
and thus $\FQ^k_l$ inherits from $\bF^k$ a structure of \emph{trivial} PMQ (see Definition \ref{defn:PMQ}.
It follows that $(\bF^k,\FQ^k_l)$ is a PMQ-group pair (see Definition \ref{defn:PMQgrouppair}).

The following Theorem generalises \cite[Theorem 4.1]{Joyce}.
\begin{thm}
 \label{thm:FQfreePMQ}
 Let $(\Q,G,\fe,\fr)$ be a PMQ-group pair, let $a_1,\dots,a_l\in\Q$ and let $g_{l+1},\dots,g_k\in G$.
 Then there are unique maps $\phi\colon\bF^k\to\cG(\Q)$ of groups and $\psi\colon\FQ^k_l\to\Q$ of PMQs
 such that $\psi\colon\gen_i\mapsto a_i$ for all $1\leq i\leq l$, $\phi\colon \gen_i\mapsto g_i$ for all $l+1\leq i\leq k$
 and $(\psi,\phi)\colon(\FQ^k_l,\bF^k)\to(\Q,G)$ is a map of PMQ-group pairs.
\end{thm}
Before proving Theorem \ref{thm:FQfreePMQ} we introduce some notation.
\begin{nota}
 \label{nota:reducedexpression}
 Let  $w\in\bF^k$. The \emph{reduced expression} of $w$ as a word in
 the letters $\gen_1^{\pm 1},\dots,\gen_k^{\pm 1}$ takes the form $w=\gen_{j_1}^{\epsilon_1}\dots\gen_{j_m}^{\epsilon_m}$,
 where $\epsilon_i=\pm 1$ for all $1\leq i\leq m$, and no two consecutive letters cancel out in $\bF^k$. The number $m$ is also called \emph{word-length norm} of $w$ and denoted by $\abs{w}$.
\end{nota}
\begin{proof}[Proof of Theorem \ref{thm:FQfreePMQ}]
Since $\bF^k$ is a free group on $\gen_1,\dots,\gen_k$, $\phi$ is uniquely determined by the requirements $\phi(\gen_i)=\fe(a_i)$ for $1\leq i\leq l$ and $\phi(\gen_i)=g_i$ for $l+1\leq i\leq k$.
 To show existence of $\psi$, start by setting $\psi(\one)=\one$.
 Let $\fg\neq \one$ be an element in $\FQ^k_l$; then there are unique $w\in\bF^k$ and $1\leq \nu\leq l$ such that:
 \begin{enumerate}
 \item $\fg=w^{-1}\gen_{\nu}w$ in $\bF^k$;
 \item if $w=\gen_{j_1}^{\epsilon_1}\dots\gen_{j_m}^{\epsilon_m}$ is the reduced expression of $w$,
 then $m=0$ or $\gen_{j_1}\neq\gen_{\nu}$.
\end{enumerate}
 For $a\in\Q$ and $g\in G$ denote by $a^g\in\Q$ the image of $a$ under the map $\fr(g)\colon\Q\to\Q$.
 We then set $\psi(\fg)\colon=a_{\nu}^{\phi(w)}$. This defines a map of sets $\psi\colon\FQ^k_l\to\Q$, with $\psi(\gen_i)=a_i$ for all $1\leq i\leq l$. 
 
 If we drop condition (2), the choice of $w$ and $\nu$ fails to be unique only because of the following ambiguity: 
 one can replace $w$ by $\gen_{\nu}^e w$, for some $e\in\mathbb{Z}$. Note however that
 $a_{\nu}=a_{\nu}^{\fe(a_{\nu})}=a_{\nu}^{\fe(a_{\nu})^{-1}}$, because $a_{\nu}=a_{\nu}^{a_{\nu}}=a_{\nu}^{a_{\nu}^{-1}}$: at this point the assumption that $\Q$ is a \emph{quandle},
 and not only a \emph{rack}, is crucial (see Definition \ref{defn:quandle}).
 Therefore $\psi$ is well-defined by the formula given above even if we drop condition (2)
 in the choice of $w$ and $\nu$. By construction $\psi$ is a map of PMQs and $(\psi,\phi)$ is a map of PMQ-group pairs.
 
 Conversely, let $\psi'\colon\FQ^k_l\to\Q$ be a map of PMQs such that $(\psi',\phi)$ is a map of PMQ-group pairs, and such
 that $\psi'\colon\gen_i\mapsto q_i$ for all $1\leq i\leq l$.
 Then $\psi'$ satisfies the formula above for any $\fg\in\FQ^k_l$, and hence $\psi'=\psi$: this shows uniqueness of $\psi$.
\end{proof}
In the proof of Theorem \ref{thm:FQfreePMQ} we see for the first time why it is convenient
to work with quandles instead of racks, see the discussion after Definition \ref{defn:quandle}.
Theorem \ref{thm:FQfreePMQ} motivates the use of the word ``free'' in Definition \ref{defn:freePMQ}.

\subsection{Decompositions of elements in free groups}
In the rest of the section we study the problem of decomposing elements $\fg\in\bF^k$ as products of elements
in $\FQ^k_l$ in different ways. Proposition \ref{prop:standardmove} ensures that if
$\fg$ has a particularly nice form, then there is essentially only one such decomposition.
\begin{defn}
 \label{defn:decomposition}
Let $\fg\in\bF^k$; a decomposition of $\fg$ with respect to $\FQ^k_l$ is a sequence $\ufg=(\fg_1,\dots,\fg_r)$
of elements in $\FQ^k_l$ such that the product $\fg_1\dots \fg_r$ is equal to $\fg$. 
\end{defn}
In general, if an element $\fg\in\bF^k$ admits a decomposition with respect to
$\FQ^k_l$, this decomposition is not unique:
for example, if $\fg$ can be decomposed as $\fg_1\cdot\fg_2$, then it can also be decomposed as
$\fg_2\cdot\fg_1^{\fg_2}$ or $\fg_2^{\fg_1^{-1}}\cdot\fg_1$.

However we note that the number $r$ of factors appearing in any decomposition $\ufg$
of $\fg$ with respect to $\FQ^k_l$ is the same for any decomposition. To see this,
consider again the map $\ab$ from Notation \ref{nota:freegroup}:
then $\ab(\fg)=\ab(\fg_1)+\dots+\ab(\fg_r)$, and each summand $\ab(\fg_i)$ is a vector
with one entry equal to 1 and all other entries equal to 0; hence $r$ only depends on $\fg$ and is equal to the sum of the entries in $\ab(\fg)$.

\begin{defn}
 \label{defn:standardmove}
Let $\Q$ be a quandle; a \emph{standard move} on a sequence of elements
$(a_1,\dots,a_r)$ replaces, for some $1\leq i\leq r-1$,
the pair of consecutive elements $(a_i,a_{i+1})$ with either $\pa{a_{i+1},a_i^{a_{i+1}}}$
or $\pa{a_{i+1}^{a_i^{-1}},a_i}$.
\end{defn}
If one applies, after one other,
two standard moves on the same pair of indices $(i,i+1)$, using once each of the two rules
$(a_i,a_{i+1})\mapsto\pa{a_{i+1},a_i^{a_{i+1}}}$ and $(a_i,a_{i+1})\mapsto\pa{a_{i+1}^{a_i^{-1}},a_i}$, one recovers
the original sequence.
In the case $\Q=\bF^k$, the reader will notice the connection between standard moves and Artin's action
of the braid group $\Br_n$ on the free group $\mathbb{F}^n$:
for $1\leq i\leq n-1$, the standard generator $\mathfrak{b}_i\in\Br_n$ acts on $\mathbb{F}^n$
by mapping the list of generators $(\gen_1,\dots,\gen_k)$ to the list of generators
$(\gen_1,\dots,\gen_{i-1},\gen_{i+1},\gen_i^{\gen_{i+1}},\gen_{i+2},\dots,\gen_k)$, i.e., by applying
a standard move.

\begin{prop}
 \label{prop:standardmove}
 Let $\tfg=\gen_1\dots\gen_r$ for some $1\leq r\leq l$,
 and let $(\tfg_1,\dots,\tfg_r)$ be a decomposition of $\tfg$ with respect to
 $\FQ^k_l$. Then it is possible to pass from the decomposition $(\tfg_1,\dots,\tfg_r)$
 to the decomposition $(\gen_1,\dots,\gen_r)$ by applying 
 a suitable sequence of standard moves.
\end{prop}
We will prove Proposition \ref{prop:standardmove} in the rest of the section.

\subsection{Generalised decompositions and straightforward computations}
\begin{defn}
 \label{defn:generaliseddecomposition}
 A \emph{generalised decomposition} (gd) in $\bF^k$ is a formal, structured iteration of the operations of
 conjugation by an element $\gen_i^{\pm 1}$ and of product,
 using only the elements $\gen_1,\dots,\gen_k$ as elementary inputs and taking the associativity
 of the product into account. More precisely, the set of all
 gds is recursively constructed as follows:
 \begin{itemize}
  \item for all $1\leq i\leq k$ we have a gd $\gen_i$;
  \item if $x$ and $y$ are gds, then also $x\cdot y$ is a gd;
  \item if $x$ is a gd, then for all $1\leq i\leq k$ both $\pa{x}^{\gen_i}$ and $\pa{x}^{\gen_i^{-1}}$
  are gds.
 \end{itemize}
 Associativity of the product is formally taken into account, i.e. for any three gds $x_1,x_2,x_3$ the
 two gds $x_1\cdot\pa{x_2\cdot x_3}$ and $\pa{x_1\cdot x_2}\cdot x_3$ are equivalent, and we write
 both of them as $x_1\cdot x_2\cdot x_3$.
 The \emph{weight} $\|x\|$ of a gd $x$ is defined recursively by:
 \begin{itemize}
  \item $\|\gen_i\|=1$ for all $1\leq i\leq k$;
  \item if $x$ and $y$ are gds, then $\|x\cdot y\|=\|x\|+\|y\|$;
  \item if $x$ is a gd and $1\leq i\leq k$, then $\left\|\pa{x}^{\gen_i}\right\|=\left\|\pa{x}^{\gen_i^{-1}}\right\|=\|x\|+2$.
 \end{itemize}
\end{defn}
\begin{defn}
\label{defn:straightforwardcomputation}
 Each gd $x$ gives rise to an element $\overline{x}\in\bF^k$ by \emph{straightforward computation},
 i.e. by interpreting product and conjugation inside $\bF^k$. We first define recursively
 a \emph{formal computation} associating with every gd $x$ a word in the letters $\gen_1^{\pm 1},\dots,\gen_k^{\pm 1}$:
 \begin{itemize}
  \item the formal computation of the gd $\gen_i$ is one-letter word $(\gen_i)$;
  \item let $x$ be a gd and suppose that the formal computation of $x$ is
  $(\gen_{\nu_1}^{\epsilon_1},\dots,\gen_{\nu_{\lambda}}^{\epsilon_{\lambda}})$;
  then the formal computations of $\pa{x}^{\gen_i}$ and $\pa{x}^{\gen_i^{-1}}$ are, respectively,
  \[
  (\gen_i^{-1}\,,\,\gen_{\nu_1}^{\epsilon_1},\dots,\gen_{\nu_{\lambda}}^{\epsilon_{\lambda}}\,,\,\gen_i)
  \quad\mbox{ and }\quad
    (\gen_i\,,\,\gen_{\nu_1}^{\epsilon_1},\dots,\gen_{\nu_{\lambda}}^{\epsilon_{\lambda}}\,,\,\gen_i^{-1});
  \]
  \item let $x$ and $y$ be gds, then the formal computation of $x\cdot y$
  is the concatenation of the formal computations of $x$ and of $y$.
 \end{itemize}
 If the formal computation of a gd $x$ is
 $(\gen_{\nu_1}^{\epsilon_1},\dots,\gen_{\nu_{\lambda}}^{\epsilon_{\lambda}})$, we set
$\overline{x}=\gen_{\nu_1}^{\epsilon_1}\dots \gen_{\nu_{\lambda}}^{\epsilon_{\lambda}}\in\bF^k$.
 We say that the straightforward computation of the gd $x$ involves \emph{no cancellation}
 if no cancellation between two consecutive occurrences of $\gen_i^{\pm 1}$
 occurs in the product $\gen_{\nu_1}^{\epsilon_1}\dots \gen_{\nu_{\lambda}}^{\epsilon_{\lambda}}$.
 For an element $\fg\in\bF^k$ we say that \emph{$x$ is a gd of $\fg$} if $\fg=\overline{x}$.
\end{defn}
\begin{ex}
 \label{ex:fff}
 For $\fg=\gen_1\gen_2\gen_3\in\bF^4$ the following are gds of $\fg$:
 \begin{itemize}
 \itemsep0.1cm
  \item $\gen_1\cdot\gen_2\cdot\gen_3$, having weight 3;
  \item $\gen_2\cdot\pa{\gen_1}^{\gen_2}\cdot\gen_3$, having weight 5;
  \item $\gen_3\cdot\pa{\gen_2\cdot\pa{\gen_1}^{\gen_2}}^{\gen_3}$, having weight 7;
  \item $\gen_3\cdot\pa{\gen_2}^{\gen_3}\cdot\pa{\pa{\gen_1}^{\gen_2}}^{\gen_3}$, having weight 9;
  \item $\pa{\pa{\gen_3\cdot\pa{\gen_2\cdot\pa{\gen_1}^{\gen_2}}^{\gen_3}}^{\gen_4}}^{\gen_4^{-1}}$, having weight 11.
 \end{itemize}
 \end{ex}
Note that the weight of a gd is the length of its formal computation.
Clearly for any gd $x$ of $\fg\in\bF^k$ we have $\|x\|\geq \abs{\fg}$,
where $\|x\|$ is the weight of $x$ and $\abs{\fg}$ is the word-length norm of
$\fg\in\bF^k$ (see Notation \ref{nota:reducedexpression}).
The equality occurs exactly when the straightforward
computation of $\fg$ from $x$ involves no cancellation.

 \begin{lem}
  \label{lem:tencases}
  Let $x$ be a gd of $\fg\in\bF^k$ and suppose that the straightforward computation of $x$ involves some cancellation. Then $x$ contains a sub-gd that has one of the following forms, where $y_1$ and $y_2$ are gds
  and $1\leq i\leq k$:
  \vspace{0.2cm}
\begin{center}
\begin{tabular}{ll}
(1)   $\pa{y_1}^{\gen_i}\cdot\pa{y_2}^{\gen_i}$;
&  (6)  $\pa{y_1\cdot \pa{y_2}^{\gen_i}}^{\gen_i^{-1}}$;\\[0.2cm]
(2)  $\pa{y_1}^{\gen_i^{-1}}\cdot\pa{y_2}^{\gen_i^{-1}}$;
&  (7)  $\pa{\pa{y_1}^{\gen_i}\cdot y_2}^{\gen_i^{-1}}$;\\[0.2cm]
(3)  $\gen_i\cdot \pa{y_1}^{\gen_i}$ \,or $\pa{\gen_i\cdot y_1}^{\gen_i}$;
&  (8) $\pa{y_1\cdot \pa{y_2}^{\gen_i^{-1}}}^{\gen_i}$;\\[0.2cm]
(4)  $\pa{y_1}^{\gen_1^{-1}}\cdot \gen_i$ \,or $\pa{y_1\cdot \gen_i}^{\gen_1^{-1}}$;
&  (9) $\pa{\pa{y_1}^{\gen_i^{-1}}\cdot y_2}^{\gen_i}$;\\[0.2cm]
(5)  $\pa{\pa{y_1}^{\gen_i}}^{\gen_i^{-1}}$ \, or $\pa{\pa{y_1}^{\gen_i^{-1}}}^{\gen_i}$;
&  (10)  $\pa{\gen_i}^{\gen_i}$\, or $\pa{\gen_i}^{\gen_i^{-1}}$.
\end{tabular}
\end{center}

 \end{lem}
 \begin{proof}
 We start the straightforward computation of $x$ from the \emph{innermost} operations, and we continue until the
 first cancellation occurs.
 \begin{itemize}
  \item Suppose that the first cancellation occurs after a conjugation, taking the form $(-)^{\gen_i}$
  or $(-)^{\gen_i^{-1}}$; i.e. there is a sub-gd
  $y$ in $x$ such that the straightforward computation of $y$ involves no cancellation, but the straightforward
  computation of $\pa{y}^{\gen_i}$ or $\pa{y}^{\gen_i^{-1}}$ involves some cancellation.
  Then we are cancelling
  one instance of $\gen_i$ with one instance of $\gen_i^{-1}$, and one of these two letters is the first or the last letter of
  the formal computation of $y$. Either the gd $y$ is obtained by
  conjugating once a smaller gd (then we are in case (5)), or $y$ is obtained by multiplying
  two smaller gds (then we are in one of cases (3),(4),(6),(7),(8) and (9)), or $y$ is $\gen_i$
  (then we are in case (10)).
  \item Suppose that the first cancellation occurs after a product,  i.e. there are sub-gds
  $y_1$ and $y_2$ in $x$ such that the straightforward computations of $y_1$ and $y_2$ involves no cancellation,
  but the straightforward computation of $y_1\cdot y_2$ involves some cancellation. Then we are cancelling the last letter
  of the formal computation of $y_1$ with the first letter of the formal computation of $y_2$;
  up to reducing the size of $y_1$ and $y_2$ and using
  the associativity of the product built in Definition \ref{defn:generaliseddecomposition},
  we can assume that neither $y_1$ nor $y_2$ is itself obtained as a product of two smaller gds.
  If $y_1$ or $y_2$ is equal to $\gen_i$, we are in one of cases (3) and (4); if both $y_1$ and $y_2$ are obtained
  by conjugating a smaller gd, we are in one of cases (1) and (2).
 \end{itemize}
 \end{proof}
 
 \begin{nota}
 \label{nota:decompositiontogd}
 Let $\underline{\fg}=(\fg_1,\dots,\fg_{\lambda})$ be a decomposition of an element $\fg\in\bF^k$ with respect
 to $\FQ^k_l$ (see Definition \ref{defn:standardmove}). For each $1\leq i\leq \lambda$ there is an
 element $w_i\in\bF^k$ and a generator $\gen_{\nu_i}$, such that $\fg_i=\gen_{\nu_i}^{w_i}$ and
 either $w_i=\one$ or the first letter appearing in the reduced
 expression of $w_i$ is not $\gen_{\nu_i}^{\pm 1}$ (see Notation \ref{nota:reducedexpression} and compare with
 $w$ and $\gen_{\nu}$ in the proof of Theorem \ref{thm:FQfreePMQ}).
 We associate with $(\fg_1,\dots,\fg_{\lambda})$ the following gd of $\fg$:
  \[
  \pa{\gen_{\nu_1}}^{w_1}\cdot \pa{\gen_{\nu_2}}^{w_2} \cdot\dots\cdot\pa{\gen_{\nu_{\lambda}}}^{w_{\lambda}},
  \]
 where $(-)^{w_i}$ is the iteration of $\abs{w_i}$ conjugations.
 We thus consider each decomposition $\underline{\fg}$ of $\fg$ with respect to $\FQ^k_l$ also as a gd of $\fg$.
 \end{nota}
Using Notation \ref{nota:decompositiontogd}, we have $\left\|\underline{\fg}\right\|=\lambda+2{\abs{w_1}}+\dots+2\abs{w_{\lambda}}$.
 \begin{defn}
  \label{defn:gdtodecomposition}
  Given a gd $x$ of some element $\fg\in\bF^k$, we can find a decomposition $\underline{\fg}=(\fg_1,\dots,\fg_{\lambda})$ of
  $\fg$ with respect to $\FQ^k_k$ as follows:
  \begin{itemize}
   \item we first make a list $\pa{\gen_{\nu_1},\dots,\gen_{\nu_{\lambda}}}$ of all sub-gds of $x$ of the elementary
   form $\gen_{\nu}$, reading $x$ from left to right;
   \item we change the previous list as follows: for all $1\leq i\leq \lambda$ we
   apply to the $i$\textsuperscript{th} element $\gen_{\nu_i}$, in the natural order, all conjugations $(-)^{\gen_j^{\pm 1}}$
   which in $x$ conjugate a sub-gd containing $\gen_{\nu_i}$.
  \end{itemize} 
  We say that $\underline{\fg}$ is the decomposition of $\fg$ with respect to $\FQ^k_k$ \emph{associated with $x$}.
 \end{defn}
Consider again the element $\fg$ from Example \ref{ex:fff} and the given list of gds: the corresponding
decomposition of $\fg$ with respect to $\FQ^4_4$ are, respectively:
\begin{itemize}
\itemsep0.2cm
 \item $\pa{\gen_1,\gen_2,\gen_3}$;
 \item $\pa{\gen_2,\gen_1^{\gen_2},\gen_3}$;
 \item $\pa{\gen_3,\gen_2^{\gen_3},\gen_1^{\gen_2\gen_3}}$ for the last three gds.
\end{itemize}
Note that if we start from a decomposition $\underline{\fg}$ of an element $\fg\in\bF^k$ with respect to
$\FQ^k_l$, consider $\underline{\fg}$ as a gd of $\fg$ according to Notation \ref{nota:decompositiontogd}, and then take
again the associated decomposition with respect to $\FQ^k_k$ in the sense of Definition \ref{defn:gdtodecomposition}, we recover precisely $\underline{\fg}$: here we also use the inclusion $\FQ_l^k\subset\FQ_k^k$.

\subsection{Proof of Proposition \ref{prop:standardmove}}
The decomposition $(\tfg_1,\dots,\tfg_r)$ of $\tfg\in\bF^k$ can be seen as a gd $x_0$ of $\tfg$ as in Notation \ref{nota:decompositiontogd}.
Suppose that $x_0$ contains a sub-gd that has one of the forms (1)-(10) listed in Lemma \ref{lem:tencases}.
Then we can obtain a new gd $x_1$ of $\tfg$ by replacing the given sub-gd respectively by:

\begin{center}
\begin{tabular}{ll}
  (1) $\pa{y_1\cdot y_2}^{\gen_i}$;
&  (6) $\pa{y_1}^{\gen_i^{-1}}\cdot y_2$;\\[0.1cm]
  (2) $\pa{y_1\cdot y_2}^{\gen_i^{-1}}$;
&  (7) $y_1\cdot \pa{y_2}^{\gen_i^{-1}}$;\\[0.1cm]
  (3) $y_1\cdot \gen_i$;
&  (8) $\pa{y_1}^{\gen_i}\cdot y_2$;\\[0.1cm]
  (4) $\gen_i\cdot y_1$;
&  (9) $y_1\cdot \pa{y_2}^{\gen_i}$;\\[0.1cm]
  (5) $y_1$;
&  (10) $\gen_i$.
\end{tabular}
\end{center}

Note that $\|x_1\|<\|x_0\|$ in all cases.
We iterate such replacements until it is no longer possible, obtaining a sequence of gds $x_0,x_1,x_2,\dots$; since the weight drops at each replacement, we will reach a gd $x_n$ of $\tfg$ containing no sub-gd of the forms (1)-(10).
By Lemma \ref{lem:tencases} the straightforward computation of $x_n$ yields $\tfg=\gen_1\dots\gen_r$ without
cancellations. Since the reduced expression of $\tfg=\gen_1,\dots\gen_r\in\bF^k$ does not contain letters of the form $\gen_i^{-1}$,
no conjugation can occur in $x_n$, and we conclude that $x_n$ is just $\gen_1\cdot\gen_2\cdot\dots\cdot\gen_r$.

For $1\leq \nu\leq n$, let $\utfg_{\nu}=\pa{\tfg_{\nu,1},\dots,\tfg_{\nu,r}}$ be the decomposition of $\tfg$ with respect
to $\FQ^k_k$ associated with the gd $x_{\nu}$ of $\tfg$ (see Definition \ref{defn:gdtodecomposition}); then for all $0\leq \nu\leq n-1$
the following holds:
\begin{itemize}
 \item if passing from $x_{\nu}$ to $x_{\nu+1}$ we have used a replacement of type (1)-(10) which is \emph{not} of type (3) or (4), then
 $\utfg_{\nu}=\utfg_{\nu+1}$;
 \item if passing from $x_{\nu}$ to $x_{\nu+1}$ we have used a replacement of type (3) or (4), then $\utfg_{\nu+1}$ is obtained from $\utfg_{\nu}$ by a standard move.
\end{itemize}
It now suffices to note that $\utfg_0=(\tfg_1,\dots,\tfg_r)$ and $\utfg_n=(\gen_1,\dots,\gen_r)$; a posteriori
we also note that all decompositions $\utfg_\nu$ of $\tfg$ are actually with respect to $\FQ^k_l\subset\FQ^k_k$.

\section{Tameness properties for PMQs and the PMQ-ring}
\label{sec:normedPMQ}
In this section we introduce the notion of
norm for a PMQ, and discuss several properties that a PMQ may enjoy, such as being
\emph{augmented} and \emph{locally finite}, and, in the normed case,
being \emph{maximally decomposable}, \emph{coconnected}, \emph{pairwise determined}, and \emph{Koszul}.
We also define the PMQ-ring $R[\Q]$ of a PMQ $\Q$ with coefficients
in a commutative ring $R$, and study its basic properties.

\subsection{Normed PMQs and normed groups}
\label{subsec:normPMQgroup}
\begin{defn}
\label{defn:norm}
 A \emph{norm} on a PMQ $\Q$ is a map of PMQs $N\colon\Q\to\N$ satisfying $N^{-1}(0)=\set{\one}\subset\Q$; here the abelian monoid $\N$
 is given the abelian PMQ structure from Definition \ref{defn:PMQ}.
 A PMQ $\Q$ is \emph{normed} if it is endowed with a norm.
\end{defn}
In most cases the norm $N$ that we consider on a PMQ $\Q$ is evident from the context and left implicit, and we will only say that $\Q$ is normed. Note however that ``being normed'' is not a \emph{property} that a PMQ may or may not satisfy, but it is an additional \emph{structure}. The following are examples of PMQs with or without norms.
\begin{ex}
The natural numbers $\N$ form a normed PMQ, with unit $\one_{\N}=0$ and norm the identity. 
\end{ex}
\begin{ex}
If $\Q$ is a PMQ with norm $N\colon\Q\to\N$, we can extend $N$ to a map of PMQs $N\colon\hQ\to\N$, using that $\N$ is a complete PMQ and the universal property of the completion $\hQ$ of $\Q$. The map $N\colon\hQ\to\N$ turns out
to be a norm on $\hQ$: if $w\in\hQ$, we can represent $w$ as a product $a_1\dots a_r$ with $a_i\in\Q$; then $N(w)=N(a_1)+\dots+N(a_r)$,
so if $N(w)=0$ we must have $N(a_i)=0$ for all $1\leq i\leq r$, i.e. $a_i=\one$, and hence $w=\one\in\hQ$.
\end{ex}
Similarly as in the previous example, if $\Q$ is a PMQ with norm $N$,
there is an induced map between the enveloping groups $\cG(N)\colon\cG(\Q)\to \cG(\N)=\Z$.
\begin{ex}
\label{ex:normedset}
 The same PMQ $\Q$ may have different norms. For instance, if $\one\in S$ is a pointed set, then $S$ can be considered as a trivial abelian PMQ as in Example \ref{ex:quandle}.
 Any map of sets $N\colon S\to\N$ with $N^{-1}(0)=\set{\one}$ is a norm.
\end{ex}
\begin{ex}
 A PMQ $\Q$ may not admit any norm. For instance, if $G$ is a non-trivial
 group, then $G$ can be considered as a PMQ as in Example \ref{ex:group}.
 Then there exists no norm $N$ on $G$, as for all $\one\neq g\in G$ we would have $N(g)+N(g^{-1})=N(\one)=0$,
 but both $N(g)$ and $N(g^{-1})$ should be strictly positive integers.
\end{ex}

Definition \ref{defn:norm} is inspired by the following, standard definition.
\begin{defn}
 \label{defn:normgroup}
 A \emph{norm} on a group $G$ is a function of sets $N\colon G\to\N$ satisfying the following properties:
 \begin{itemize}
  \item $N(g)+N(h)\geq N(gh)$ for all $g,h\in G$;
  \item $N(g)=0$ if and only if $g=\one$.
 \end{itemize}
 A norm is \emph{conjugation invariant} if moreover $N(g)=N(h^{-1}gh)$ for all $g,h\in G$.
\end{defn}
The link between the notions of normed group and normed PMQ is given by the following definition, associating a normed PMQ with a normed group.
\begin{defn}
 \label{defn:Ggeo}
Let $G$ be a group with a conjugation invariant norm $N$. We define a PMQ $G\geo=G\geo_N$, called the \emph{geodesic PMQ} associated with $G$.
The underlying quandle of $G\geo$ is the underlying quandle of $G$; the partial product of $G\geo$ is only defined on pairs $(a,b)$ of elements of $G$ such that $N(ab)=N(a)+N(b)$, and coincides with the product $ab$ in $G$. The unit is the unit of $G$.
The norm $N\colon G\geo\to\N$ is defined, as a map of sets, by the norm $N\colon G\to \N$.
\end{defn}
The triangular inequality and conjugation invariance for $N$ ensure that all conditions in Definitions \ref{defn:partialmonoid} and \ref{defn:PMQ} are satisfied.
The fact that $N\colon G\geo\to\N$ is a map of PMQs, and in particular of partial monoids,
follows from the fact that the products $(a,b)\mapsto ab$ allowed in $G\geo$ are precisely
those for which $N$ is additive.

One can consider Definition \ref{defn:Ggeo} as a particular instance of Definition \ref{defn:fullsubPMQ}:
we can indeed consider the group $G\times \Z$, and let $S\subset G\times\Z$ be the subset containing
all pairs of the form $(g,N(g))$. Then $S$ contains the unit $(\one,0)$ of $G\times \Z$ and
inherits from $G\times\Z$ a structure of PMQ, which is isomorphic to $G\geo$.
Note also that $(G\geo,G)$ is naturally a PMQ-group pair (see Definition \ref{defn:PMQgrouppair}),
by considering $\Id_G$ as a map (of PMQs) $G\geo\to G$,
and by letting $G$ act on $G\geo$ by right conjugation.

\begin{defn}
\label{defn:epsilongeo}
 The map of PMQs $\Id_G\colon G\geo\to G$ gives rise to a map of groups
 $\epsilon\geo=\cG(\Id_G)\colon\cG(G\geo)\to G$.
\end{defn}
In Section \ref{sec:fSgeo} we will study in detail the PMQ $\fS_d\geo$ arising from the symmetric group $\fS_d$, endowed with the word length norm with respect to transpositions.

\subsection{Augmented and locally finite PMQs}
A necessary condition, for a PMQ to admit norms, is that it is \emph{augmented}.
\begin{defn}
 \label{defn:augmentedPMQ}
Recall Definition \ref{defn:PMQideal}. A PMQ $\Q$ is \emph{augmented} if
$\Q\setminus\set{\one}$ is an ideal of $\Q$. For an augmented PMQ $\Q$ we denote $\Q_+ =\Q\setminus\set{\one}$. If $\Q$ and $\Q'$ are augmented
PMQs, a map of PMQs $\Psi\colon\Q\to\Q'$ is \emph{augmented} if $\Psi(\Q_+)\subset\Q'_+$.
\end{defn}

\begin{ex}
 \label{ex:one0}
Let $\set{\one,0}$ be the abelian monoid with unit $\one$, such that $0\cdot 0=0$, and regard
$\set{\one,0}$ as an abelian (complete) PMQ. Then $\set{\one,0}$ is augmented, as $\set{0}$ is an ideal.
In fact, for a generic PMQ $\Q$, the following are equivalent:
\begin{itemize}
 \item $\Q$ is augmented;
 \item the map of sets $\Q\to\set{\one,0}$, given by
$a\mapsto 0$ for $a\in\Q\setminus\set{\one}$ and $\one\mapsto\one$, is a map of PMQs.
\end{itemize}
This explains the use of the word \emph{augmented} in Definition \ref{defn:augmentedPMQ}:
we think of the map $\Q\to\set{\one,0}$ as being an augmentation.
\end{ex}

\begin{ex}
 Every PMQ with trivial product is augmented. More generally, 
 a normed PMQ $\Q$ is augmented: the set $\Q\setminus\set{\one}$ contains all elements
 of strictly positive norm and is thus an ideal.
\end{ex}

\begin{defn}
\label{defn:locfin}
 A PMQ $\Q$ is \emph{locally finite} if for every $a\in\Q$ there are finitely many sequences $(a_1,\dots,a_r)$
 of elements of $\Q\setminus\set{\one}$ with $a=a_1\dots a_r$.
\end{defn}

\begin{ex}
 \label{ex:hQnotlocfin}
We give an example of a complete and normed PMQ which is not locally finite.
 Let $\Q=\bF^2$ be the free group on two elements, and consider it as a PMQ with \emph{trivial} multiplication.
 Define $N\colon\Q\to\N$ by declaring $N(a)=1$ for all $a\neq \one\in\Q$. Since every element $a\in\Q$ can be factored only as $a\one$ or $\one a$, we have that $\Q$ is locally finite.
 
 The completion $\hQ$ is however not locally finite: for instance, if $\gen_1,\gen_2$ denote the generators of
 $\bF^2$, then the element $w=\hat{\gen_1}\hat{\gen_2}\in\hQ$ can be rewritten in infinitely many ways as a product of elements of norm 1:
 \[
 \hat{\gen_1}\hat{\gen_2}=\hat{\gen_2}\hat{\gen_1}^{\hat\gen_2}=
 \hat{\gen_1}^{\hat\gen_2}\hat{\gen_2}^{\hat{\gen_1}^{\hat\gen_2}}=
 \hat{\gen_2}^{\hat{\gen_1}^{\hat\gen_2}} \pa{\hat{\gen_1}^{\hat\gen_2}}^{\hat{\gen_2}^{\hat{\gen_1}^{\hat\gen_2}}}=\dots.
 \]
\end{ex}
\begin{ex}
 Let $G$ be a finite, non-trivial group and consider $G$ as a PMQ with full product.
 Then for all $a\in G$ there are finitely many pairs $(b,c)\in G\times G$ such that $a=bc$;
 nevertheless $G$ is not locally finite, since every element $a$ can be written in infinitely
 many ways as a product $a_1\dots a_r$, for $r$ arbitrarily large.

Similarly, one can show that a locally finite PMQ must be augmented: if $\Q$ is not augmented, there exist $a,b\in\Q\setminus\set{\one}$ with $ab=\one$, and thus $\one$ can be decomposed as $ab=abab=ababab=\dots$, in infinitely many ways, so that $\Q$ is not locally finite.
\end{ex}
\begin{ex}
 Let $\Q$ be a finite, normed PMQ with norm $N$: then $\Q$ is locally finite. Indeed
 for all $n\ge0$ there are finitely many sequences $(a_1,\dots,a_r)$
 of elements of $\Q_+$ with $N(a_1)+\dots +N(a_r)\le n$; a fortiori each element
 $a\in\Q$ admits finitely many decompositions as product of elements $\neq \one$.
\end{ex}

\subsection{Coconnected PMQs}
If $\Q$ is a PMQ with norm $N$ and $a\in\Q$, we may attempt to decompose $a$ in $\Q$ as a product
$a_1\dots a_r$ of elements of norm 1.
\begin{nota}
 \label{nota:Qh}
 For a normed PMQ $\Q$ and $r\geq 0$ we denote by $\Q_r\subset\Q$ the subset of elements of $\Q$
 of norm $r$.
\end{nota}
\begin{defn}
 \label{defn:maxdecomposable}
Recall Definition \ref{defn:decomposition}, and let $a$ be an element of a normed PMQ $\Q$.
A decomposition of $a$ with respect to $\Q_1$ is a (possibly empty) sequence $(a_1,\dots,a_r)$
of elements of $\Q_1$, where $r=N(a)$, such that the product
$a_1\dots a_r$ is defined in $\Q$ and is equal to $a$.
We say that $\Q$ is \emph{maximally decomposable} if every element $a\in\Q$
admits a decomposition with respect to $\Q_1$.
\end{defn}
Not every normed PMQ is maximally decomposable: for instance the normed
PMQ $S$ from Example \ref{ex:normedset} is maximally
decomposable if and only if all elements in $S\setminus\set{\one}$
have norm 1.
\begin{defn}
\label{defn:irreducible}
 Let $\Q$ be an augmented PMQ. An element $a\in\Q_+$ is \emph{irreducible} if it cannot
 be written as a product $bc$ in $\Q$ with $b,c\in\Q_+$.
 For a generic element $a\in\Q$ we set $h(a)\in\N\cup\set{\infty}$ to be the supremum
 of $r\ge0$ such that $a$ admits a decomposition $a=a_1\dots a_r$ with $a_1,\dots,a_r\in\Q_+$;
 by convention $h(\one)=0$ and $h(a)=1$ if $a$ is irreducible.
 If $h(a)$ is finite for all $a\in\Q$, we say that the function $h\colon\Q\to\N$ is the \emph{intrinsic pseudonorm}
 of $\Q$.
\end{defn}
A condition ensuring that $h(a)$ is finite for all $a\in\Q$ is that $\Q$ is locally finite.
Note that for all $a,b\in\Q$ we have $h(ab)\ge h(a)+h(b)$, and $h(a)=0$ if and only if $a=\one$.
Note also that if $h(a)$ is finite and $a=a_1\dots a_{h(a)}$ is a decomposition
witnessing the value $h(a)$, then the elements $a_1,\dots,a_{h(a)}$ must be irreducible.

We can now characterise those augmented PMQs which admit a norm
$N$ for which they are maximally decomposable. These are precisely those augmented
PMQs satisfying all of the following properties:
\begin{itemize}
 \item for all $a\in\Q$, $h(a)$ is finite;
 \item for all $a\in\Q$ and all decompositions $a=a_1\dots a_r$ of $a$ into irreducibles,
 we have $r=h(a)$.
\end{itemize}
\begin{defn}
 \label{defn:intrinsicnorm}
If $\Q$ satisfies the above properties, then the intrinsic pseudonorm $h$
from Definition \ref{defn:irreducible} is in fact a norm,
making $\Q$ into a maximally decomposable PMQ, with $\Q_1$ being the set of irreducible elements;
moreover $h$ is the unique such norm.
We will say that $h$ is the \emph{intrinsic norm} on the augmented PMQ $\Q$.
\end{defn}
If $\Q$ does not satisfy the above properties, then there is no norm $N\colon\Q\to\N$ making
$\Q$ into a maximally decomposable PMQ.

We have discussed existence of maximal decompositions, let us now turn to uniqueness.
Recall Definition \ref{defn:standardmove}, and note that if $(a_1,\dots,a_r)$
is a decomposition of $a\in\Q$ with respect to $\Q_1$ and if $r\geq 2$, then we can apply
a standard move to it and obtain a possibly different decomposition, e.g. $(a_2,a_1^{a_2},a_3,\dots,a_r)$.
\begin{defn}
\label{defn:coconnected}
A normed PMQ $\Q$ is \emph{coconnected} if it is maximally decomposable
and if, for all $a\in\Q$ and for every pair of decompositions $(a_1,\dots,a_r)$ and $(a'_1,\dots,a'_r)$ of $a$ with respect
to $\Q_1$, there is a sequence of standard moves connecting the first decomposition to the second.
\end{defn}
\begin{ex}
\label{ex:Ncoconnected}
The abelian monoid $\N$,
regarded as a normed, abelian PMQ, is coconnected. More generally, for $n\ge0$ the subset $\Q=\set{0,\dots,n}\subset\N$ can be regarded as an abelian PMQ by virtue of Definition \ref{defn:fullsubPMQ}, with norm
given by the inclusion in $\N$: then $\Q$ is coconnected.
\end{ex}

\begin{ex}
\label{ex:Ncup1p}
An example of a PMQ which is maximally decomposable but not coconnected is
$\Q=\N\cup\set{1'}=\set{0,1,1',2,3,\dots}$, regarded as an abelian, complete PMQ as follows:
switching to additive notation, we use
the usual sum for pairs of elements of $\N\subset\Q$ and we
set $1'+0=0+1'=1'$, $1'+1'=2$ and $n+1'=1'+n=n+1$ for all $n\in\N\setminus\set{0}$.
The norm $N\colon\Q\to\N$ is the identity on $\N\subset\Q$, and $N(1')=1$.

Note that $2\in\Q$ has four decompositions with respect to $\Q_1=\set{1,1'}$, namely
$(1,1)$, $(1',1')$, $(1,1')$ and $(1',1)$; the last two are connected by a standard move, whose
effect is just swapping the two entries; however there is, for instance, no sequence of standard moves
connecting any two of the first three decompositions.
\end{ex}

We conclude the subsection by giving a convenient description of the completion of a coconnected PMQ $\Q$.
Let $\Q_{\le1}=\set{\one}\cup\Q_1\subset\Q$ be the subset
containing all elements of norm $\le1$. Then $\Q_{\le1}$ can be considered as a PMQ with trivial multiplication.
The inclusion $\Q_{\le1}\subset\Q$ is a map of PMQs, inducing a map between the completions
$\widehat{\Q_{\le1}}\to\hQ$.

\begin{lem}
\label{lem:coconnected}
If $\Q$ is coconnected, the above map $\widehat{\Q_{\le1}}\to\hQ$ is an isomorphism.
\end{lem}
\begin{proof}
 The monoid $\hQ$ is generated by the elements of $\Q$, which in turn can be obtained as products of elements
 in $\Q_{\le1}$, since $\Q$ is maximally decomposable. This implies that the map
 $\widehat{\Q_{\le1}}\to\hQ$ is surjective.
 
 To show injectivity, let $(a_1,\dots,a_r)$ and $(a'_1,\dots,a'_r)$ be sequences of elements
 of $\Q_1$ such that the products $a_1\dots a_r\in\widehat{\Q_{\le1}}$ and $a'_1\dots a'_r\in\widehat{\Q_{\le1}}$
 are sent to the same element of $\hQ$ of norm $r$. The equality $a_1\dots a_r=a'_1\dots a'_r$ in $\hQ$
implies the existence of $n\ge1$ and of sequences $(b_{i,1},\dots,b_{i,r_i})$
of elements of $\Q$, for $1\le i\le n$, such that:
 \begin{itemize}
  \item  $(b_{1,1},\dots,b_{1,r_1})=(a_1,\dots,a_r)$ and $(b_{n,1},\dots,b_{n,r_n})=(a'_1,\dots,a'_r)$;
  \item for all $1\le i\le n-1$, the sequence $(b_{i+1,1},\dots,b_{i+1,r_{i+1}})$ can be obtained
  from $(b_{i,1},\dots,b_{i,r_i})$ by one of the following moves (or their inverses):
  \begin{enumerate}
   \item replace two consecutive entries $(b_{i,j},b_{i,j+1})$ by their product $b_{i,j}b_{i,j+1}$, provided that this
   product exists in $\Q$;
   \item replace two consecutive entries $(b_{i,j},b_{i,j+1})$ by the two consecutive entries
   $(b_{i,j+1},b_{i,j}^{b_{i,j+1}})$.
  \end{enumerate}
 The hypothesis that $\Q$ is coconnected implies that, without loss of generality, we can assume that no move of type
 (1) takes place, at the cost of increasing $n$ and inserting more standard moves, i.e. those of type (2).
 If no move of type (1) occurs, all elements $b_{i,j}$ belong to $\Q_1$, and the entire procedure witnesses that also the products $a_1\dots a_r\in\widehat{\Q_{\le1}}$ and $a'_1\dots a'_r\in\widehat{\Q_{\le1}}$ are equal.
\end{itemize}
\end{proof}
Lemma \ref{lem:coconnected} can be interpreted as follows: all 
coconnected PMQs can be obtained as sub-PMQs of completions of PMQs with trivial product. Moreover, the completion of a coconnected PMQ is again coconnected.

\subsection{Pairwise determined PMQs}
Classically, a partial abelian monoid $M$ is \emph{pairwise determined} if for every
$r\geq 3$ and every sequence $(a_1,\dots,a_r)$ of elements of $M$, the sum $a_1+\dots+a_r$ is defined in $M$ if and only if for all $1\leq i<j\leq r$ the sum $a_i+a_j$ is defined. Note that if $M$, considered as an abelian PMQ,
is normed and maximally decomposable,
then it is equivalent to require the previous dichotomy for all $r\geq 3$
and only for all sequences $(a_1,\dots,a_r)$ of elements of $M$ of norm 1. We generalise this notion to PMQs.
\begin{defn}
 \label{defn:pairwisedetermined}
 Let $\Q$ be a normed and
 maximally decomposable PMQ. We say that $\Q$ is \emph{pairwise determined} if the following holds:
 for every $r\geq 3$ and every sequence $(a_1,\dots,a_r)$ of elements of $\Q_1$,
 either the product $a_1\dots a_r$ is defined in $\Q$, or there is a sequence
 of standard moves connecting $(a_1,\dots,a_r)$ to a new sequence $(a'_1,\dots,a'_r)$
 such that the product $a'_1a'_2$ is not defined in $\Q$.
\end{defn}
An example of a PMQ which is pairwise determined but not coconnected is any abelian monoid which is not coconnected, see for instance Example \ref{ex:Ncup1p}.

\begin{ex}
Let $n\ge1$ and let
$\Q=\set{0,\dots,n}\subset\N$, regarded as an abelian PMQ by virtue of Definition \ref{defn:fullsubPMQ},
with norm given by the inclusion into $\N$.
Note that every element 
admits a unique maximal decomposition as sum of 1's, hence $\Q$ is coconnected;
on the other hand, if $n\ge2$, the sequence $(1,1,\dots,1)$ of length $n+1$ is fixed by any standard move,
is not summable, yet its first two elements are summable: hence for $n\ge2$ the PMQ $\Q$ is not pairwise determined.
\end{ex}

\subsection{PMQ-ring of a PMQ}
For a group $G$ and a commutative ring $R$, there is a classical notion of group ring $R[G]$.
We generalise this notion to PMQs.
\begin{defn}
 \label{defn:PMQring}
 For a PMQ $\Q$ and a commutative ring $R$ we denote by $R[\Q]$ the \emph{PMQ-ring}
 of $\Q$ with coefficients in $R$. It is an associative $R$-algebra; as an $R$-module it is free with standard basis given by elements $\sca a$ for $a\in\Q$; for $a,b\in\Q$ we set $\sca a\sca b=\sca{ab}\in R[\Q]$ if $ab$ is defined in $\Q$, and we set $\sca a\sca b=0$ otherwise.
\end{defn}
The definition of $R[\Q]$ only depends on the underlying partial monoid of the PMQ $\Q$. If $\Q$ is augmented (see Definition \ref{defn:augmentedPMQ}), then $R[\Q]$ is an augmented
algebra: the augmentation $\epsilon\colon R[\Q]\to R$ sends $\sca{\one}\mapsto 1$ and $\sca a\mapsto 0$
for all $a\in\Q\setminus\set{\one}$.

If $\Q$ is endowed with a norm $N$, then $R[\Q]$ is a graded algebra: for $\nu\geq 0$
the degree-$\nu$ part of $R[\Q]$, denoted $R[\Q]_\nu$, is spanned by the elements $[a]$ with $a\in\Q_\nu$.

If $\Q$ is maximally decomposable, then $R[\Q]$ is generated in degree 1:
every basis element $\sca a$ can be written as a product
$\sca{a_1}\dots\sca{a_r}$ of elements of degree 1, where $(a_1,\dots,a_r)$ is a decomposition of $a$ with respect to $\Q_1$.
\begin{defn}
 \label{defn:quadraticRalgebra}
 Let $R$ be a commutative ring; a graded $R$-algebra $A$ is \emph{quadratic}
 if it is generated in degree 1 and related in degree 2. More precisely, the degree 1 part $A_1\subset A$ is a free $R$-module, and $A$ is the quotient of the free tensor algebra $T_\bullet A_1=\bigoplus_{i\ge0} (A_1)^{\otimes i}$ by a two-sided ideal
 $I\subset T_\bullet A_1$ generated by elements in $(A_1)^{\otimes 2}$.
 Here tensor products are taken over $R$.
\end{defn}
\begin{thm}
 \label{thm:RQquadratic}
 Let $\Q$ be a maximally decomposable, coconnected and pairwise determined PMQ;
 then $R[\Q]$ is a quadratic $R$-algebra.
\end{thm}
\begin{proof}
Let $R\left<\Q_1\right>$ denote the free associative $R$-algebra on the set $\Q_1$,
which is a graded $R$-algebra by putting the generators in degree 1. There is a natural
map of graded $R$-algebras $\fu\colon R\left<\Q_1\right>\to R[\Q]$ given by $a\mapsto\sca a$
for $a\in\Q_1$. Since $\Q$ is maximally decomposable, the map $\fu$ is surjective;
note also that $\fu$ is bijective in degrees 0 and 1.
We want to show that the kernel of $\fu$ is generated, as a two-sided ideal
of $R\left<\Q_1\right>$, by elements of degree 2. Let $\ker(\fu)_\nu$ denote the degree $\nu$
part of $\ker(\fu)$. Then $\ker(\fu)_\nu$ is generated as $R$-module by the following elements:
\begin{enumerate}
 \item monomials $\left<a_1,\dots,a_h\right>$, such that the product $a_1\dots a_h$ is not defined in $\Q$;
 \item differences of monomials $\left<a_1,\dots,a_h\right>-\left<a'_1,\dots,a'_h\right>$,
 such that the products $a_1\dots a_h$ and $a'_1\dots a'_h$ are defined and equal in $\Q$.
\end{enumerate}
Let $I\subset R\left<\Q_1\right>$ be the two-sided ideal generated by $\ker(\fu)_2$: we want to prove
that the inclusion $I\subseteq\ker(\fu)$ is an equality.
First, let $(a_1,\dots,a_r)$ and $(a'_1,\dots a'_r)$ be two sequences of elements of $\Q_1$ that differ by
a standard move, swapping the elements in positions $j$ and $j+1$ (in particular $a_i=a'_i$ for all $i\neq j,j+1$).
Then $\left<a_j,a_{j+1}\right>-\left<a'_j,a'_{j+1}\right>\in\ker(\fu)_2$, in both of the following cases:
\begin{itemize}
 \item if $a_ja_{j+1}$ is defined, then $\left<a_j,a_{j+1}\right>-\left<a'_j,a'_{j+1}\right>$ is an element of type (2);
 \item if $a_ja_{j+1}$ is not defined, then $\left<a_j,a_{j+1}\right>-\left<a'_j,a'_{j+1}\right>$  is a difference of elements of type (1).
\end{itemize}
It follows that the difference $\left<a_1,\dots,a_r\right>-\left<a'_1,\dots, a_r'\right>$ is equal to
\[
 \left<a_1,\dots,a_{j-1}\right> \big( \left<a_j,a_{j+1}\right>-\left<a'_j,a'_{j+1}\right> \big)
 \left< a_{j+2},\dots a_r\right> \in I.
\]
Using that $\Q$ is coconnected, we can express every element of $\ker(\fu)_r$ of type (2)
as a linear combination (in fact a sum) of elements of the form $\left<a_1,\dots,a_r\right>-\left<a'_1,\dots, a'_r\right>$ with $(a_1,\dots,a_r)$ and $(a'_1,\dots a'_r)$ differing for a single standard move;
this shows that all elements of type (2) lie in $I$.

By the previous argument and the hypothesis that $\Q$ is pairwise determined,
every element $\left<a_1,\dots,a_r\right>$ of type (1)
can be written as a sum of an element in $I$ and another element of the form
$\left<a'_1,\dots,a'_r\right>$, where we can assume that the product $a'_1a'_2$ is not defined
in $\Q$. Then $\left<a'_1,a'_2\right>\in\ker(\fu)_2$, and thus 
\[
\left<a'_1,\dots,a'_r\right>= \left<a'_1,a'_2\right>\left<a'_3,\dots,a'_r\right>\in I.
\]
\end{proof}

\subsection{Invariants of the adjoint action}
Recall from Subsection \ref{subsec:adjointaction} the adjoint action of $\cG(\Q)$ on $\Q$: it induces an action
of $\cG(\Q)$ on $R[\Q]$ by ring automorphisms.
\begin{defn}
\label{defn:caQ}
 We define $\cA(\Q)$ as the sub-$R$-algebra $R[\Q]^{\cG(\Q)}\subseteq R[\Q]$ of invariant elements under the adjoint action.
\end{defn}
Note that at least the copy of $R\subset R[\Q]$ spanned by the element $\sca{\one}$ is contained in $\cA(\Q)$;
therefore $\cA(\Q)$ is a unital $R$-algebra.
\begin{defn}
\label{defn:conj}
Let $\Q$ be a PMQ, and recall from Definition \ref{defn:quandle} the notion of conjugacy class.
If $S\in\conj(\Q)$ is a finite conjugacy class, we denote
\[
 \sca{S}:=\sum_{a\in S}\sca{a}\in R[\Q].
\]
\end{defn}
Note that the elements $\sca{S}$, for $S$ ranging among finite conjugacy classes of $\Q$, exhibit $\cA(\Q)$ as a free $R$-module.
This follows from the observation that a generic element $x\in R[\Q]$ takes the form $x=\sum_{a\in\Q}\lambda_a\sca{a}$, where almost all coefficients $\lambda_a\in R$ are zero. The element $x$ belongs to $\cA(\Q)$ if and only if $\lambda_a=\lambda_b$ whenever $\conj(a)=\conj(b)$; hence for every conjugacy class $S$ there is $\lambda_S\in R$ such that
$x=\sum_{S\in\conj(\Q)}\sum_{a\in S}\lambda_S\sca a$. Since $x$ is a finite linear combination of elements $\sca{a}$,
$\lambda_S$ vanishes whenever $S$ is infinite, and moreover almost all coefficients $\lambda_S$ vanish.

\begin{lem}
\label{lem:cominv}
The ring $\cA(\Q)$ is contained in the centre of $R[\Q]$; in particular $\cA(\Q)$ is a commutative ring.
\end{lem}
\begin{proof}
It suffices to prove that for all finite conjugacy class $S\in\conj(\Q)$ and all $b\in\Q$ the equality $\sca{S}\sca{b}=\sca{b}\sca{S}$ holds in $R[\Q]$. The first product is equal to $\sum_{a\in S}\sca{a}\sca{b}$; by Definition \ref{defn:PMQring}
and the axioms of PMQ, we have $\sca{a}\sca{b}=\sca{b}\sca{a^b}$ for all $a,b\in \Q$, so we may rewrite
the first product as $\sum_{a\in S}\sca{b}\sca{a^b}$. It suffices now to remember that $(-)^b\colon\Q\to\Q$ is a bijection and restricts to a bijection $(-)^b\colon S\to S$, hence the latter formula equals the second product.
\end{proof}

\subsection{Koszul PMQs}
Let $A$
be a graded, associative $R$-algebra, and assume that $A_i$ is a finitely generated, free $R$-module
for all $i\ge0$. The algebra $A$ is \emph{connected}
if $A_0=R$; in this case $A$ admits a canonical augmentation
$A\to R$, making in particular $R$ into a left and right $A$-module concentrated in degree 0.
The groups $\Ext_{A}^*(R,R)$ inherit a grading from $A$, so that for $j\geq0$ there is
a decomposition $\Ext^j_{A}(R,R)=\bigoplus_{\nu\geq0}\Ext_{A}^{j,\nu}(R,R)$. A connected $R$-algebra $A$ is \emph{Koszul} if for all $j\ge0$, $\Ext_{A}^j(R,R)$ is a finitely generated, free $R$-module concentrated in degree $j$, i.e. $\Ext_{A}^j(R,R)=\Ext_{A}^{j,j}(R,R)$. Recall that if $A$ is Koszul, then it is also quadratic.

\begin{defn}
 \label{defn:PMQkoszul}
 A normed, locally finite PMQ $\Q$ is \emph{Koszul} (over $R$) if $R[\Q]$, as a graded, connected $R$-algebra, is Koszul.
\end{defn}
By Theorem \ref{thm:RQquadratic}, if $\Q$ is a normed, locally finite, coconnected, pairwise determined and Koszul PMQ, then $R[\Q]$ is a (quadratic)
Koszul $R$-algebra and therefore $\Ext_{R[\Q]}^*(R,R)$
is isomorphic, as a graded $R$-algebra, to the \emph{dual} quadratic algebra of $R[\Q]$. More precisely,
on the one hand
$R[\Q]$ is isomorphic to the free associative $R$-algebra with following generators and relations:

\begin{tabular}{cl}
 \textbf{Generators} & {For all $a\in\Q_1$ there is a generator $\sca{a}$ in degree 1.}\\[0.1cm]
 \textbf{Relations}  & For all $(a,b)\in\Q_1\times\Q_1$ there is a relation $\sca{a}\sca{b}=\sca{b}\sca{a^b}$;\\
 &if $ab$ is not defined in $\Q$, there is also a relation $\sca{a}\sca{b}=0$ .
\end{tabular}

On the other hand $\Ext_{R[\Q]}^{*}(R,R)$ is isomorphic to the free associative $R$-algebra with the following generators and relations:

\begin{tabular}{cl}
 \textbf{Generators} &  For all $a\in\Q_1$ there is a generator $\sca{a}'$ in degree 1.\\[0.1cm]
 \textbf{Relations} & For all $c\in\Q_2$ there is a relation $\sum \sca{a}'\sca{b}'=0$: here the sum\\
 &is extended
 over all pairs $(a,b)\in\Q_1\times\Q_1$ satisfying $ab=c$.
\end{tabular}

Note that the sum $\sum \sca{a}'\sca{b}'=0$ is finite because we assume $\Q$ locally finite.
\begin{ex}
Let $\Q=\N=\set{0,1,2,\dots}$ with the identity norm; then $\Q$ is a Koszul PMQ, as $R[\N]\cong R[x]$
is a Koszul algebra.
\end{ex}
\begin{ex}
\label{ex:1ton}
Let $n\ge1$ and let $\Q=\set{0,1,\dots,n}\subset\N$, regarded as a PMQ by virtue of Definition
\ref{defn:fullsubPMQ}, with norm given by the inclusion in $\N$.
Then $R[\Q]\cong R[x]/x^{n+1}$ is not a quadratic algebra unless $n=1$;
for $n=1$ we have that $R[x]/x^2$ is a Koszul algebra.
Therefore $\Q$ is Koszul if and only if $n=1$.
\end{ex}

\begin{ex}
\label{ex:segre}
Let $\Q=\set{\one,a,a',b,b',c}$ be the abelian PMQ in which the only non-trivial, defined multiplications
are $ab=a'b'=c$. Define a norm $N\colon\Q\to\N$ by setting $N(a)=N(a')=N(b)=N(b')=1$ and $N(c)=2$.

Then the completion $\hQ$ of $\Q$ is the free abelian monoid with generators $a,a',b,b'$, modulo the relation $ab=a'b'$, again considered as an abelian PMQ, and the norm $N$ extends to $\hQ$. Note that $\hQ$ is not coconnected, as we still have $c=ab=a'b'$.

We have that $R[\hQ]\cong R[x,y,x',y']/(xy-x'y')$ is isomorphic to the Segre subalgebra
$R[s_1t_1,s_1t_2,s_2t_1,s_2t_2]$ of the polynomial ring $R[s_1,s_2,t_1,t_2]$ in 4 variables, which is known
to be Koszul. Hence $\hQ$ is Koszul. 
\end{ex}

\part{Simplicial Hurwitz spaces}
\section{Double bar constructions in braided monoidal categories}
\label{sec:barconstructions}
In this section we collect some general facts about algebra objects in braided monoidal categories.
The main goal is to define, for a pair of commutative algebras $(A,B,f)$ in a braided monoidal category
$\bfA$, the double bar construction $B_{\bullet,\bullet}(A,B,f)$, which is a bisimplicial object
in $\bfA$. The material of this section is standard and is included for the sake of completeness.
\begin{nota}
For a category $\bfA$ we denote by $s\bfA$ the category of simplicial objects in $\bfA$ (i.e. functors $\Delta^{op}\to\bfA$); similarly $ss\bfA$ denotes the category of bisimplicial objects in $\bfA$.
\end{nota}
\subsection{Algebras in monoidal categories}
In this subsection we denote by $\bfA$ a monoidal category, with monoidal product $-\otimes-$ and unit object $\bf1$. We shall neglect all issues related to associators and unitors.
\begin{defn}
\label{defn:AlgbfA}
We denote by $\Alg(\bfA)$ the category of algebras (or unital monoid objects) in $\bfA$. An algebra $A$ is endowed with a multiplication $\mu_A\colon A\otimes A\to A$ and a unit $\eta_A\colon\bfone\to \bfA$,
and the following identities are required:
\begin{itemize}
 \item $\mu_A\circ(\mu_A\otimes\Id_A)=\mu_A\circ(\Id_A\otimes\mu_A)\colon A\otimes A\otimes A\to A$;
 \item $\mu_A\circ(\eta_A\otimes\Id_A)=\mu_A\circ(\Id_A\otimes\eta_A)=\Id_A\colon A\to A$.
\end{itemize}
A morphism $f\colon A\to B$ is an algebra morphism if it satisfies the following identities:
\begin{itemize}
 \item $\mu_B\circ(f\otimes f)=f\circ\mu_B\colon A\otimes A\to B$;
 \item $f\circ\eta_A=\eta_B\colon\bfone\to B$.
\end{itemize}
An algebra pair in $\bfA$ is the datum $(A,B,f)$ of two algebras $A,B\in\Alg(\bfA)$ and a morphism
of algebras $f\colon A\to B$.
Algebra pairs form a category $\Alg(\bfA)^{[0,1]}$, which is the arrow category of $\Alg(\bfA)$.

Let $A\in\Alg(\bfA)$.
A left $A$-module $B$ is endowed with a multiplication $\mu_{A,B}\colon A\otimes B\to B$, and the following identities are required:
 \begin{itemize}
  \item $\mu_{A,B}\circ(\mu_A\otimes\Id_B)=\mu_{A,B}\circ(\Id_A\otimes\mu_{A,B})\colon A\otimes A\otimes B\to B$;
  \item $\mu_{A,B}\circ(\eta_A\otimes \Id_B)=\Id_B\colon B\to B$.
 \end{itemize}
The notion of right $A$-module is defined in an analogous way.
An $A$-bimodule $B$ is endowed with both left and right $A$-module structures, satisfying the following:
\begin{itemize}
 \item $\mu_{B,A}\circ(\mu_{A,B}\otimes\Id_A)=\mu_{A,B}\circ(\Id_A\otimes\mu_{B,A})\colon A\otimes B\otimes A\to B$.
\end{itemize}
\end{defn}

\begin{nota}
 \label{nota:algebrapairbimodule}
If $(A,B,f)\in\Alg(\bfA)^{[0,1]}$ is an algebra pair, we consider $B$ as an $A$-bimodule
with structure maps $\mu_{A,B}:=\mu_B\circ(f\otimes\Id_B)$ and $\mu_{B,A}:=\mu_B\circ(\Id_B\otimes f)$.
\end{nota}

\subsection{Algebras in braided monoidal categories}
Assume now that $\bfA$ is braided monoidal; for two objects $A,B\in\bfA$ we denote by $\braiding_{A,B}\colon A\otimes B\to B\otimes A$ the braiding.
The following construction makes $\Alg(\bfA)$ into a monoidal category.
\begin{defn}
 \label{defn:tensoralgebra}
 Let $\bfA$ be a braided monoidal category with tensor product $-\otimes-$,
 unit object $\bfone_{\bfA}$ and braiding $\braiding_{-,-}$, and let
 $A,B\in\Alg(\bfA)$. We endow $A\otimes B$ with a structure of algebra in $\bfA$, by letting the product $\mu_{A\otimes B}\colon (A\otimes B)\otimes (A\otimes B)\to A\otimes B$ be the composition
 \[
 \begin{tikzcd}[column sep=6em]
  A\otimes B\otimes A\otimes B \ar[r,"\Id_A\otimes\braiding_{B,A}\otimes\Id_B"] &
  A\otimes A\otimes B\otimes B \ar[r,"\mu_A\otimes\mu_A"] & A\otimes B.
  \end{tikzcd}
 \]
Similarly, the unit $\eta_{A\otimes B}\colon \bfone_{\bfA}\to A\otimes B$
is defined as $\bfone_{\bfA}\cong \bfone_{\bfA}\otimes \bfone_{\bfA}\overset{\eta_A\otimes\eta_B}\to A\otimes B$.
\end{defn}
If $f\colon A\to A'$ and $g\colon B\to B'$ are algebra morphisms, then also
$f\otimes g\colon A\otimes B\to A'\otimes B'$ is an algebra morphism.
Thus the previous definition makes $\Alg(\bfA)$ into a monoidal category.
As a consequence, if $(A,B,f)$ and $(A',B',f')$ are algebra pairs in $\bfA$, then $f\otimes f'\colon A\otimes A'\to B\otimes B'$ is a morphism of algebras, hence $(A\otimes A',B\otimes B',f\otimes f')$ is also an algebra pair.
In fact we obtain a monoidal structure on the category $\Alg(\bfA)^{[0,1]}$:
it is actually a general fact that the arrow category of a monoidal category (in this case $\Alg(\bfA)$) is also a monoidal
category.

Assuming that $\bfA$ is braided monoidal allows us also to define the notion of 
commutative algebra (and similarly of commutative algebra pair).
\begin{defn}
\label{defn:commutative}
Let $\bfA$ be a braided monoidal category. An algebra
$A\in\Alg(\bfA)$ is \emph{commutative}
if the following diagram commutes:
  \[
   \begin{tikzcd}
    A\otimes A\ar[r,"\braiding_{A,A}"]\ar[dr,"\mu_A"'] & A\otimes A\ar[d,"\mu_A"]\\
    &A.
   \end{tikzcd}
  \]
Similarly, $(A,B,f)\in\Alg(\bfA)^{[0,1]}$ is \emph{commutative}
if both $A$ and $B$ are commutative.
\end{defn}

\subsection{Single bar construction}
We will define the bar construction in the setting of a monoidal category $\bfA$ and an algebra
pair $(A,B,f)$ in $\bfA$; the most general construction would use a left $A$-module
and a right $A$-module which are possibly distinct, but we will not need this level of generality
in this article.
\begin{defn}
\label{defn:barconstruction}
 Let $\bfA$ be a monoidal category and let $(A,B,f)$ be an algebra
 pair in $\bfA$. Recall Notation \ref{nota:algebrapairbimodule}.
 We define a simplicial object $B_\bullet (A,B,f) \in s\bfA$:
 \begin{itemize}
  \item for $n\ge0$, the $n$\textsuperscript{th} object $B_n(A,B)$ is given by $B\otimes A^{\otimes n}\otimes B$;
  \item for $n\ge1$ and $1\le i\le n-1$, the $i$\textsuperscript{th} face map $d_i\colon B_n(A,B)\to B_{n-1}(A,B)$ is given by
  $\Id_B\otimes \Id_A^{\otimes i-1}\otimes \mu_A\otimes \Id_A^{\otimes n-i-1}\otimes \Id_B$;
  we also set $d_0=\mu_{B,A}\otimes\Id_A^{\otimes n-1}\otimes \Id_B$ and $d_n=\Id_B\otimes\Id_A^{\otimes n-1}\otimes \mu_{A,B}$.
  \item for $n\ge0$ and $0\le i\le n$, the $i$\textsuperscript{th} degeneracy map $s_i\colon B_n(A,B)\to B_{n+1}(A,B)$ is given by
  $\Id_B\otimes\Id_A^{\otimes i}\otimes \eta_A\otimes \Id_A^{\otimes n-i}\otimes \Id_B$.
 \end{itemize}
\end{defn}
By Definition \ref{defn:tensoralgebra}, the objects
$B_n (A,B,f)=B\otimes A^{\otimes n}\otimes B$ can be naturally regarded as algebras in $\bfA$. However, in order to enhance $B_\bullet (A,B,f)$ to a simplicial object in $\Alg(\bfA)$, we need some additional hypothesis on $(A,B,f)$.

\begin{lem}
\label{lem:sAlg}
 Let $\bfA$ be a braided monoidal category and $(A,B,f)$ be a \emph{commutative} algebra pair; then
 $B_\bullet (A,B,f)$ can be naturally regarded as an object in $s\Alg(\bfA)$.
 Moreover the assignment
 \[
 B_n(f,\Id_B):=\Id_B\otimes f^{\otimes n}\otimes \Id_B \colon B_n(A,B,f)\to B_n(B,B,\Id_B)
 \]
 gives a morphism in $s\Alg(\bfA)$, so that
 we obtain a simplicial algebra pair
 \[
 (B_\bullet (A,B,f),B_\bullet(B,B,\Id_B),B_\bullet(f,\Id_B)).
 \]
\end{lem}
\begin{proof}
 We first have to check that all face maps
 $d_i\colon B_n(A,B)\to B_{n-1}(A,B)$ and degeneracy maps $s_i\colon B_n(A,B)\to B_{n+1}(A,B)$,
 which a priori are only morphisms in $\bfA$, are in fact morphisms in $\Alg(\bfA)$.
 First, note that since $\mu_A$ is associative and is invariant under precomposition with $\braiding_{A,A}$,
 the following diagram commutes:
  \[
  \begin{tikzcd}
   A\otimes A\otimes A\otimes A \ar[d,"\mu_A\otimes\mu_A"]\ar[r,"\mu_{A\otimes A}"] & A\otimes A\ar[d,"\mu_A"]\\
   A\otimes A\ar[r,"\mu_A"] &A.
  \end{tikzcd}
 \]
This, together with the observation $\mu_A\circ\eta_{A\otimes A}=\eta_A$,
implies that $\mu_A\colon A\otimes A\to A$ is a map of algebras in $\bfA$.
Similarly one can check that $\eta_A\colon\bfone\to A$,
$\mu_{A,B}\colon A\otimes B\to B$, $\mu_{B,A}\colon B\otimes A\to B$
are morphisms in $\Alg(\bfA)$, using that $B$ is commutative and that, by the naturality
of the braiding, also the following diagram commutes:
\[
 \begin{tikzcd}
  A\otimes A\ar[r,"\braiding_{A,A}"]\ar[d,"f\otimes f"]&A\otimes A\ar[d,"f\otimes f"]\\
  B\otimes B\ar[r,"\braiding_{B,B}"] & B\otimes B.
 \end{tikzcd}
\]

It follows now directly from definition \ref{defn:barconstruction} that both the face map
$d_i\colon B_nA\to B_{n-1}A$
and the degeneracy map $s_i\colon B_nA\to B_{n+1}A$
are monoidal products (in the monoidal category $\Alg(\bfA)$) of morphisms
in $\Alg(\bfA)$, and hence they are morphisms in $\Alg(\bfA)$ as well, i.e. maps of algebras.
This makes $B_\bullet (A,B)$ into a simplicial algebra in $\bfA$, or equivalently
into an algebra in $s\bfA$; the category $s\bfA$ is braided monoidal,
with levelwise monoidal product and braiding.

The same argument shows that $B_\bullet(B,B,\Id_B)$ is a simplicial algebra in $\bfA$.
To check that $B_\bullet(f,\Id_B)$ is a morphism of simplicial algebras, it suffices
to check that for all $n\ge0$ the morphism $B_n(f,\Id_B)$ is a morphism of algebras: this is evident,
as $B_n(f,\Id_B)$ is a tensor product of morphisms of algebras.
\end{proof}

\subsection{Double bar construction}
Having a simplicial algebra pair in $\bfA$, we can apply again, levelwise, the bar construction from Definition \ref{defn:barconstruction}.

\begin{defn}
 \label{defn:doublebarconstruction}
 Let $(A,B,f)\in\Alg(\bfA)^{[0,1]}$ be a commutative algebra pair in a braided monoidal category.
 Regard $(B_\bullet (A,B,f),B_\bullet(B,B,\Id_B),B_\bullet(f,\Id_B))$
 as an object in $s\Alg(\bfA)^{[0,1]}$. We denote by $B_{\bullet,\bullet}(A,B,f)$ the bisimplicial object
 in $\bfA$ obtained by applying levelwise the bar construction: explicitly, we have
 \[
 B_{p,q}(A,B,f)=B_p\pa{B_q(A,B,f),B_q(B,B,\Id_B),B_q(f,\Id_B)},
 \]
 where we use that $\pa{B_q(A,B,f),B_q(B,B,\Id_B),B_q(f,\Id_B)}\in\Alg(\bfA)^{[0,1]}$.

  For $0\leq i\le p$ and $q\ge0$ with $p\neq 0$, we denote by
  $
  d^{\hor}_i\colon B_{p,q}(A,B,f)\to B_{p-1,q}(A,B,f)
  $
  the
 $i$\textsuperscript{th} \emph{horizontal} face map.
 For $p\ge0$ and $0\le j\le q$ with $q\neq 0$, we denote by $d^{\ver}_j\colon B_{p,q}(A,B,f)\to B_{p,q-1}(A,B,f)$ the
 $j$\textsuperscript{th} \emph{vertical} face map. We use a similar notation $s^{\hor}_i$ and $s^{\ver}_j$ for the \emph{horizontal}
 and \emph{vertical} degeneracy maps.
\end{defn}
Note that, even if $(A,B,f)$ is a commutative algebra pair, in general $B\otimes A^{\otimes n}\otimes B$
and $B^{\otimes n+2}$ are not commutative algebras:
thus, in general, $B_\bullet(A,B,f)$ is not a simplicial object in \emph{commutative} algebra pairs, and
$B_{\bullet,\bullet}(A,B,f)$ is not a bisimplicial object in $\Alg(\bfA)$, so that no further bar construction is available.
For completeness, we remark that if $(A,B,f)\in\Alg(\bfA)$ is a commutative algebra pair
and the following equalities hold (they are for instance automatic if $\bfA$ is symmetric monoidal)
\[
\braiding_{A,A}\circ\braiding_{A,A}=\Id_{A\otimes A};\quad\quad
\braiding_{A,B}\circ\braiding_{B,A}=\Id_{B\otimes B};\quad\quad
\braiding_{B,A}\circ\braiding_{A,B}=\Id_{A\otimes B},
\]
then $(B_\bullet(A,B,f),B_\bullet(B,B,\Id_B),B_\bullet(f,\Id_B))$ is a commutative algebra pair
in $s\bfA$ (which is endowed with the levelwise braided monoidal structure) satisfying the analogous of the three equalities above,
and one can iterate arbitrarily many times the bar construction.
We will be most interested in examples in which the three equalities above do not hold
(in fact, the first equality does not hold), so we will only focus on the double bar construction.

\section{Simplicial Hurwitz spaces}
\label{sec:simplhur}
We fix a PMQ $\Q$ throughout the section, and denote by $\hQ$ its completion. We use this to construct a CW complex $|\Arr(\Q)|$,
obtained as geometric realisation of a
bisimplicial set $\Arr(\Q)$. The definition of $\Arr(\Q)$
is based on \emph{arrays} of elements of $\hQ$ and is an instance
of a double bar construction.

We then assume that $\Q$ is augmented, and define a sub-bisimplicial set
$\NAdm(\Q)$ of $\Arr(\Q)$, containing
the \emph{non-admissible} part of $\Arr(\Q)$.
The \emph{simplicial Hurwitz space} $\Hur^{\Delta}(\Q)$
is then defined as the difference of spaces
$|\Arr(\Q)|\setminus|\NAdm(\Q)|$.

The relative cellular chain complex $\CH_*(|\Arr(\Q)|,|\NAdm(\Q)|)$
has a rather
simple, combinatorial description, and can also be identified with the reduced total chain complex associated with a certain bisimplicial abelian group, arising through a double bar construction.

If $\Q$ is locally finite, then $\Hur^{\Delta}(\Q)$ turns out to be locally compact, and the compactly supported dual cochain complex $\CH^*_{\mathrm{cpt}}(|\Arr(\Q)|,|\NAdm(\Q)|)$ computes the compactly supported cohomology of $\Hur^{\Delta}(\Q)$.

\subsection{\texorpdfstring{$\hQ$}{hQ}-crossed objects}
\label{subsec:hQcrossed}
The following definition extends the classical notion of
$G$-crossed objects in a category (see for instance \cite[Section 4.2]{FreydYetter})
from the specific case of a group $G$ to the more general case of a complete PMQ $\hQ$.
\begin{defn}
\label{defn:hQborel}
We define a (small) category $\hQ\borel\hQ$. Its set of objects is
$\hQ$; for $a,b\in\hQ$, a morphism $a\to b$ is given by an element $c\in\hQ$ such that
$b=a^c$. Composition of morphisms is given by multiplication in $\hQ$.
For a category $\bfA$ we let $\XA(\hQ)$ be the category of functors
$\hQ\borel\hQ\to\bfA$, with natural transformations as morphisms. An object in $\XA(\hQ)$ is called a \emph{$\hQ$-crossed object in $\bfA$}.
\end{defn}
Concretely, an object $X\in\XA(\hQ)$ consists of objects $X(a)$ for all $a\in\hQ$,
together with maps $(-)^b\colon X(a)\to X(a^b)$ for all $a,b\in\hQ$.
In all examples we will consider, $\bfA$ will be a category
with coproducts; in this case we have a forgetful functor
\[
 U\colon \XA(\hQ)\to \bfA,\quad X\mapsto \coprod_{a\in\hQ}X(a);
\]
we will sometimes regard $X\in\XA(\hQ)$ as
the object $\coprod_{a\in\hQ}X(a)\in\bfA$,
which is endowed with a decomposition into subobjects $X(a)$ as well as with a right action
of $\hQ$, i.e. a map of PMQs $\rho\colon\hQ\to \Aut_{\bfA}(X)^{op}$ with the property that $\rho_b$ restricts to a map $X(a)\to X(a^b)$ for all $a,b\in\hQ$.
The object $X(a)$ will be called the part of $X$
of \emph{$\hQ$-grading} equal to $a$. The map $\rho_b$ is also denoted $(-)^b$.

Since $\Aut_{\bfA}(X)^{op}$ is a group, the map $\rho\colon\hQ\to \Aut_{\bfA}(X)^{op}$
extends to a map of groups $\cG(\rho)\colon\cG(\hQ)\to \Aut_{\bfA}(X)^{op}$, so a $\hQ$-crossed object in $\bfA$ is naturally endowed with an action of the group
$\cG(\hQ)$.

The examples of $\hQ$-crossed categories that we will consider are the following:
\begin{itemize}
 \item the category $\XSet(\hQ)$ of $\hQ$-crossed sets;
 \item the category $\XTop(\hQ)$ of $\hQ$-crossed topological spaces;
 \item for a commutative ring $R$, the category $\XMod_R(\hQ)$ of $\hQ$-crossed $R$-modules.
\end{itemize}
In the special case $\hQ=G$, for a group $G$, the category $\XMod_R(\hQ)$ agrees with the category of \emph{Yetter-Drinfeld modules} over the Hopf algebra $R[G]$, considered recently in \cite{ETW:shufflealgebras}. The name ``Yetter-Drinfeld'' is explained as follows: Yetter \cite{Yetter90} associates with any Hopf algebra $H$ a category of ``crossed bimodules'', and Yetter's category agrees with the category of (left) modules over the Drinfeld double of $H$ \cite{Drinfeld86} when $H$ is a finite dimensional Hopf algebra over a field, as proved in \cite{Majid91}.

\begin{ex}
 \label{ex:hQascrossedset}
  The set $\hQ$ has a $\hQ$-crossed set structure, with $a\in\hQ$ in $\hQ$-grading $a$, and where $\hQ$ acts on itself by conjugation. In fact $\hQ$ is the terminal object in $\XSet(\hQ)$: every $\hQ$-crossed set admits a unique map of $\hQ$-crossed sets to $\hQ$.
  
  If $\Q$ is an augmented PMQ with completion $\hQ$, we can regard $\hQ$ as the disjoint union of three $\hQ$-crossed sets, namely
  $\cJ(\Q)=\hQ\setminus\Q$, $\Q_+=\Q\setminus\set{\one}$
   and $\set{\one}$.
\end{ex}

\begin{ex}
\label{ex:RQ}
  The PMQ-rings $R[\Q]$ and $R[\hQ]$
  (see Definition \ref{defn:PMQring})
  are $\hQ$-crossed $R$-modules, by putting $\sca{a}$ in $\hQ$-grading $a\in\hQ$, and by letting $\sca{a}^b=\sca{a^b}$ for $b\in\hQ$.
\end{ex}

 \begin{defn}
 \label{defn:Xproduct}
 Let $\bfA$ be a monoidal category admitting coproducts and for which the monoidal product
 $-\otimes -$ distributes with respect to coproducts (e.g., $\bfA$ is closed monoidal). Let $\bfone_{\bfA}$ be the monoidal unit of $\bfA$.
  We define a monoidal product on $\XA(\hQ)$, called \emph{Day convolution} and denoted also $-\otimes-$.
 For $X,Y\in\XA$, we set
 \[
  \pa{X\otimes Y}(a)=\coprod_{b,c\in\hQ\colon bc=a}X(b)\otimes Y(c).
 \]
 The action of $\hQ$ is diagonal: for $d\in\hQ$ the map $(-)^d\colon(X\otimes Y)(a)\to(X\otimes Y)(a^d)$
 restricts to the map $(-)^d\otimes (-)^d\colon X(b)\otimes Y(c)\to X(b^d)\otimes Y(c^d)$ whenever
 $bc=a$. Note that $X(b^d)\otimes Y(c^d)$ is indeed one of the objects occurring in the coproduct defining $(X\otimes Y)((bc)^d)$, because of the equality $b^dc^d=(bc)^d$.

 We obtain a monoidal structure on $\XA(\hQ)$: associativity of the tensor product follows from the associativity
 of the product of $\hQ$.
 The unit object of the monoidal category $\XA(\hQ)$ is
 $\bfone_{\XA(\hQ)}$, defined by setting:
 \begin{itemize}
  \item $\bfone_{\XA(\hQ)}(\one)=\bfone_{\bfA}$, which is endowed with the trivial $\hQ$-action;
  \item $\bfone_{\bfA}(a)=\emptyset_{\bfA}$, for all $a\in\hQ\setminus\set{\one}$, where
  $\emptyset_{\bfA}$ is the initial object (or empty coproduct) in $\bfA$.
 \end{itemize}
\end{defn}
 Note that by definition we have a chain of natural isomorphisms of objects in $\bfA$
 \[
 \begin{split}
  U(X)\otimes U(Y)& = \pa{\coprod_{b\in\hQ}X(b)}\otimes\pa{\coprod_{c\in\hQ}Y(c)}\cong
  \coprod_{b,c\in\hQ} X(b)\otimes Y(c) \cong\\
  &\cong \coprod_{a\in\hQ}(X\otimes Y)(a)=U(X\otimes Y),
 \end{split}
 \]
 which is natural in $X$ and $Y$, so that $U\colon\XA(\hQ)\to \bfA$ is a strong monoidal functor.

\begin{defn}
 \label{defn:Xbraiding}
Let $\bfA$ be a category as in Definition \ref{defn:Xproduct}, and assume further that $\bfA$
is braided monoidal, with braiding denoted $\braiding_{-,-}$.
We enhance the monoidal structure on $\XA(\hQ)$ to a braided monoidal structure, by
defining a braiding on $\XA(\hQ)$, also denoted $\braiding_{-,-}$.
For $X,Y\in\XA(\hQ)$ and for $a\in\hQ$, the braiding $\braiding_{X,Y}(a)\colon (X\otimes Y)(a)\to (Y\otimes X)(a)$ restricts, for all decompositions $a=bc$ in $\hQ$, to the isomorphism
$X(b)\otimes Y(c)\to Y(c)\otimes X(b^c)$ given by the composition
\[
\begin{tikzcd}[column sep=11ex]
 X(b)\otimes Y(c)\ar[r,"\braiding_{X(b),Y(c)}"]& Y(c)\otimes X(b)\ar[r,"\Id_{Y(c)}\otimes (-)^c"] &
 Y(c)\otimes X(b^c).
\end{tikzcd}
\]
\end{defn}

We check explicitly that the braiding $\braiding_{-,-}$ on $\XA(\hQ)$ satisfies
the braid relation. Let $X,Y,Z\in\XA(\hQ)$; we have to check for all $a\in\hQ$ the equality of the
two following compositions of morphisms $(X\otimes Y\otimes Z)(a)\to (Z\otimes Y\otimes X)(a)$:
\[
\begin{split}
(\braiding_{Y,Z}\otimes\Id_{X})(a)\circ(\Id_{Y}\otimes\braiding_{X,Z})(a)\circ(\braiding_{X,Y}\otimes\Id_{Z})(a);\\
(\Id_{Z}\otimes\braiding_{X,Y})(a)\circ(\braiding_{X,Z}\otimes\Id_{Y})(a)\circ(\Id_X\otimes\braiding_{Y,Z})(a).
\end{split}
\]
For all $b,c,d\in\hQ$ satisfying $bcd=a$,
both compositions
restrict to the same morphism $X(b)\otimes Y(c)\otimes Z(d)\to Z(d)\otimes Y(c^d)\otimes X(b^{cd})$,
namely
\[
 (\Id_{Z(d)}\otimes(-)^d\otimes (-)^{cd}) \circ\braiding_{X(b),Y(c),Z(d)}\colon X(b)\otimes Y(c)\otimes Z(d)\to Z(d)\otimes Y(c^d)\otimes X(b^{cd}),
\]
where $\braiding_{X(b),Y(c),Z(d)}$ denotes either of the following compositions in $\bfA$:
\[
\begin{split}
(\braiding_{Y(c),Z(d)}\otimes\Id_{X(b)})\circ(\Id_{Y(c)}\otimes\braiding_{X(b),Z(d)})\circ(\braiding_{X(b),Y(c)}\otimes\Id_{Z(d)})=\\
(\Id_{Z(d)}\otimes\braiding_{X(b),Y(c)})\circ(\braiding_{X(b),Z(d)}\otimes\Id_{Y(c)})\circ(\Id_{X(b)}\otimes\braiding_{Y(c),Z(d)}).
\end{split}
\]
This uses that the braiding of $\bfA$ satisfies the braid relation, together with the naturality of the braiding
of $\bfA$ and the equality $cd=dc^d$.
We leave to the reader to check that the other axioms of braided monoidal category are satisfied.

\subsection{The bisimplicial set of arrays}
 Recall Example \ref{ex:hQascrossedset} and Definition \ref{defn:commutative}. The monoid structure of $\hQ$ makes $\hQ$ into a commutative algebra in $\XSet(\hQ)$.
\begin{defn}
\label{defn:ArrQ}
 Let $\Q$ be a PMQ and let $\hQ$ denote its completion. We denote by $\Arr(\Q)=B_{\bullet,\bullet}(\hQ,\hQ,\Id_{\hQ})$. It is a bisimplicial $\hQ$-crossed set, and by the levelwise
 forgetful functor $ssU\colon ss\XSet(\hQ)\to ss\Set$ it can be regarded as a bisimplicial set.
\end{defn}

Note that $\Arr(\Q)$ only depends on the completion $\hQ$ of $\Q$. In Subsection \ref{subsec:NAdm}, under the hypothesis that $\Q$ is augmented, we will define a bisimplicial sub-$\hQ$-crossed set $\NAdm(\Q)\subset\Arr(\Q)$: the latter will depend on $\Q$ and not only on $\hQ$, and we will be mainly interested in the couple of bisimplicial $\hQ$-crossed sets $(\Arr(\Q),\NAdm(\Q))$.
Note also that the $\hQ$-crossed set $\Arr_{p,q}(\Q)=B_{p,q}(\hQ,\hQ,\Id_{\hQ})$ is isomorphic to the $\hQ$-crossed set $\hQ^{\otimes(p+2)(q+2)}$, whose underlying set is the cartesian power $\hQ^{(p+2)(q+2)}$:
we thus regard an element of $\Arr_{p,q}(\Q)$ as an \emph{array} of size $(p+2)\times(q+2)$ with entries in $\hQ$. More precisely, a generic array $\ua=(a_{i,j})_{0\leq i\leq p+1,0\leq i\leq q+1}$
of size $(p+2)\times(q+2)$ with entries in $\hQ$ consists of $p+2$ \emph{columns}
$\ua_0,\dots,\ua_{p+1}$, containing each the entries $a_{i,j}$ for a fixed value of $i$;
similarly, the $\hQ$-crossed set $B_{p,q}(\hQ,\hQ,\Id_{\hQ})=(\hQ^{\otimes q+2})^{\otimes p+2}$ can be regarded as a set under the forgetful functor $U$;
the set $U(B_{p,q}(\hQ,\hQ,\Id_{\hQ}))$ is in canonical bijection with the set
$(\hQ^{q+2})^{p+2}$, containing $(p+2)$-tuples of $(q+2)$-tuples of elements of $\hQ$.

We describe now the horizontal and vertical face and degeneracy maps of $\Arr(\Q)$.
\begin{nota}
 \label{nota:column}
 For an array $\ua\in\Arr_{p,q}(\Q)$ and for $0\le i\le p+1$
 we denote by $\ua_i=(a_{i,0},\dots,a_{i,q+1})$ the $i$\textsuperscript{th} column, which is a sequence of $q+2$
 elements in $\hQ$.
 
 For $0\leq j\leq q+2$ and for any sequence $\ub=(b_0,\dots,b_{q+1})$ of $q+2$ elements in $\hQ$
 we denote by $c(\ub)_j$ the product $b_0\dots b_{j-1}\in\hQ$; by convention we set $c(\ub)_0=\one$.
\end{nota}
The following lemma is a direct consequence of the definitions.
\begin{lem}
 \label{lem:facemaps}
Let $\ua\in\Arr_{p,q}(\Q)$ for some $p,q\ge0$:
\begin{itemize}
 \item for $0\le i\le p$, the degeneracy map $s_i^{\hor}\colon\Arr_{p,q}(\Q)\to\Arr_{p+1,q}(\Q)$
 acts on $\ua$ by adjoining a column made of $\one$'s, between the $i$\textsuperscript{th} and $(i+1)$\textsuperscript{st} columns of $\ua$;
 \item similarly, for $0\le j\le q$, the degeneracy map
 $s_j^{\ver}\colon\Arr_{p,q}(\Q)\to\Arr_{p,q+1}(\Q)$ acts on $\ua$ by adjoining a row made of $\one$'s, between the $j$\textsuperscript{th} and
 $(j+1)$\textsuperscript{st} rows of $\ua$;
 \item for $p\ge1$ and $0\le i\le p$, the face map $d_i^{\hor}\colon\Arr_{p,q}(\Q)\to\Arr_{p-1,q}(\Q)$
 acts on $\ua$ as follows: the columns of $d_i^{\hor}(\ua)$ are obtained from those of $\ua$ by
 replacing the $i$\textsuperscript{th} and $(i+1)$\textsuperscript{st} columns of $\ua$ by the following sequence of $q+2$ elements in $\hQ$:
\[
 \pa{a_{i,0}a_{i+1,0}\, ,\, a_{i,1}^{a_{i+1,0}}a_{i+1,1}\,,\,\dots\, ,\, a_{i,j}^{c(\ua_{i+1})_j}a_{i+1,j}\,,\,\dots\,,\, a_{i,q+1}^{c(\ua_{i+1})_{q+1}}a_{i+1,q+1}}. 
\]
\item for $q\ge1$ and $0\le j\le q$, the face map $d_j^{\ver}\colon\Arr_{p,q}(\Q)\to\Arr_{p,q-1}(\Q)$
 acts on $\ua$ as follows: the rows of $d_j^{\ver}(\ua)$ are obtained from those of $\ua$ by replacing the
 $j$\textsuperscript{th} and $(j+1)$\textsuperscript{st} rows by the following sequence of $p+2$ elements in $\hQ$:
\[
 \pa{a_{0,j}a_{0,j+1}\,,\,a_{1,j}a_{1,j+1}\,,\,\dots\,,\,a_{i,j}a_{i,j+1}\,,\,\dots\,,\,a_{p+1,j}a_{p+1,j+1}}.
\]
\end{itemize}
\end{lem}

\subsection{Non-degenerate arrays}
For the rest of the section we assume that $\Q$, and hence also $\hQ$, are augmented. In this subsection we
show that the non-degenerate arrays in $\Arr(\Q)$ form a semi-bisimplicial $\hQ$-crossed set, i.e. they are closed under face maps.
\begin{nota}
An array $\ua\in\Arr_{p,q}(\Q)$ is degenerate if it lies in the image of a degeneracy
map $s_i^{\hor}$ or $s_j^{\ver}$ of the bisimplicial $\hQ$-crossed set $\Arr(\Q)$.
For all $p,q\ge0$ we denote by $\Arr^{\ndeg}_{p,q}(\Q)\subset\Arr_{p,q}(\Q)$
the subset of non-degenerate arrays.
\end{nota}
Visualising elements of $\Arr(\Q)$ as arrays of elements in $\hQ$ is very helpful: for instance, an array $\ua$ is degenerate if and only if it has an inner row or an inner column made of $\one$'s, i.e. if either of the following holds:
\begin{itemize}
 \item there is $1\le i\le p$ such that for all $0\le j\le q+1$ we have $a_{i,j}=\one$;
 \item there is $1\le j\le q$ such that for all $0\le i\le p+1$ we have $a_{i,j}=\one$.
\end{itemize}

\begin{lem}
\label{lem:ndegsemibisimplicial}
The $\hQ$-crossed sets $\Arr_{p,q}^{\ndeg}(\Q)$ assemble into a semi-bisimplicial $\hQ$-crossed set $\Arr^{\ndeg}(\Q)$, whose horizontal
and vertical face maps are the restrictions of those of the bisimplicial $\hQ$-crossed set $\Arr(\Q)$.
\end{lem}
\begin{proof}
Let $\ua\in\Arr_{p,q}(\Q)$ be an array, and let $0\le i\le p+1$ and $0\le j\le q+1$.
We need to check that $d^{\hor}_i(\ua)$ and $d^{\ver}_j(\ua)$ are non-degenerate arrays,
as soon as $\ua$ is non-degenerate. 
Applying Lemma \ref{lem:facemaps} we obtain the following.
\begin{itemize}
 \item Suppose that $\ua':=d^{\hor}_i\ua\in\Arr_{p-1,q}(\Q)$ is degenerate; there are two possibilities.
 \begin{itemize}
  \item The $(i')$\textsuperscript{th} column of $\ua'$ is made of $\one$'s, for some $1\le i'\le p-1$: if $i'\neq i$,
  then the $(i')$\textsuperscript{th} column of $\ua'$ is also one of the inner columns of $\ua$, witnessing that $\ua$
  is degenerate; if $i'=i$, by Lemma \ref{lem:facemaps} the $i$\textsuperscript{th} column of $\ua'$
  contains the elements $a_{i,j'}^{c(\ua_{i+1})_{j'}} a_{i+1,j'}$, for $0\le j'\le q+1$, and if all these elements are equal to $\one$,
  using that $\hQ$ is augmented, both the $i$\textsuperscript{th} and
  the $(i+1)$\textsuperscript{st} columns of $\ua$ witness that $\ua$ is degenerate.
  \item The $(j')$\textsuperscript{th} row of $\ua'$ is made of $\one$'s, for some $1\le j'\le q$: by Lemma \ref{lem:facemaps}
  the $(j')$\textsuperscript{th} row of $\ua'$ contains the elements $a_{i',j'}$, for $0\le i'\le p+1$ with $i'\neq i,i+1$,
  together with the element $a_{i,j'}^{c(\ua_{i+1})_{j'}}a_{i+1,j'}$; again, if all elements in the $(j')$\textsuperscript{th}
  row of $\ua'$ are $\one$'s, then also all elements in the $(j')$\textsuperscript{th} row of $\ua$ are $\one$'s.
 \end{itemize}
 \item Suppose that $\ua':=d^{\ver}_j\ua\in\Arr_{p,q-1}(\Q)$ is degenerate; there are two possibilities.
 \begin{itemize}
  \item The $(j')$\textsuperscript{th} row of $\ua'$ is made of $\one$'s, for some $1\le j'\le q-1$: if $j'\neq j$,
  then the $(j')$\textsuperscript{th} row of $\ua'$ is also one of the inner rows of $\ua$; if $j'=j$, by Lemma \ref{lem:facemaps} the $j$\textsuperscript{th} row of $\ua'$
  contains the elements $a_{i',j}a_{i',j+1}$, for $0\le i'\le p+1$, and if all these elements are equal to $\one$, using that $\hQ$ is augmented, then both the $j$\textsuperscript{th} \emph{and} the $(j+1)$\textsuperscript{st} rows of $\ua$ witness that $\ua$ is degenerate.
  \item The $(i')$\textsuperscript{th} column of $\ua'$ is made of $\one$'s, for some $1\le i'\le q$: by Lemma \ref{lem:facemaps}
  the $(i')$\textsuperscript{th} column of $\ua'$ contains the elements $a_{i',j'}$, for $0\le j'\le q+1$ with $j'\neq j,j+1$,
  together with the element $a_{i',j}a_{i',j+1}$; again, if all elements
  in the $(i')$\textsuperscript{th} column of $\ua'$ are $\one$'s, then also the $(i')$\textsuperscript{th} column of $\ua$ is made of $\one$'s.
 \end{itemize}
\end{itemize}
Thus a horizontal or vertical face of a non-degenerate array is again non-degenerate.
\end{proof}
In particular, the geometric realisation $|\Arr(\Q)|$ of $\Arr(\Q)$ as a bisimplicial set
is \emph{homeomorphic} to the thick geometric realisation $|\!|\Arr^{\ndeg}(\Q)|\!|$ of $\Arr^{\ndeg}(\Q)$ as a semi-bisimplicial set.
This has an advantage when constructing $\CH_*(|\Arr(\Q)|)$, the cellular chain complex of
$|\Arr(\Q)|$: its generators are in bijection with cells of $|\Arr(\Q)|$, i.e. with non-degenerate arrays;
moreover the differential in $\CH_*(|\Arr(\Q)|)$ is given by the usual alternating sum of
vertical and horizontal
face maps, where no term has to be skipped because it corresponds to a degenerate face of a bisimplex.

\subsection{Non-admissible arrays}
\label{subsec:NAdm}
We keep assuming that $\Q$ and $\hQ$ are augmented.
\begin{defn}
 \label{defn:NAdm}
Let $p,q\ge0$; an array $\ua\in\Arr_{p,q}(\Q)$ is \emph{admissible} if both the following conditions hold:
\begin{enumerate}
 \item the first and last rows, as well as the first and last columns, are made of $\one$'s;
 i.e., whenever $i\in\set{0,p+1}$ or $j\in\set{0,q+1}$ we have $a_{i,j}=\one\in\hQ$.
 \item all entries $a_{i,j}$ lie in $\Q\subset\hQ$.
\end{enumerate}
An array is \emph{non-admissible} if it is not admissible.
We denote by $\NAdm_{p,q}(\Q)\subset\Arr_{p,q}(\Q)$
the subset of non-admissible arrays. By definition $\NAdm_{p,q}(\Q)$ arises as the (non-disjoint)
union of two sets $\NAdm^{(1)}_{p,q}(\Q)$ and $\NAdm^{(2)}_{p,q}(\Q)$, containing arrays
for which condition (1), respectively condition (2), fails.
\end{defn}
We observe that $\NAdm_{p,q}(\Q)$ is closed under the action of $\hQ$ by conjugation, hence $\NAdm_{p,q}(\Q)$ can be regarded as a $\hQ$-crossed set. The same remark holds, for $\star=1,2$, for the subset
$\NAdm_{p,q}^{(\star)}(\Q)$.
\begin{lem}
 \label{lem:NAdminArr}
 Let $\star=1,2$; then the $\hQ$-crossed sets $\NAdm_{p,q}^{(\star)}(\Q)$, for varying $p,q\ge0$, assemble into a sub-$\hQ$-crossed bisimplicial set $\NAdm(\Q)^{(\star)}\subset \Arr(\Q)$.
 As a consequence also the sets $\NAdm_{p,q}(\Q)$ assemble into a sub-$\hQ$-crossed
 bisimplicial set $\NAdm(\Q)\subset \Arr(\Q)$.
\end{lem}
\begin{proof}
We prove that all face and degeneracy maps of $\Arr(\Q)$ reflect each of condition (1) and (2)
in Definition \ref{defn:NAdm}: this means, for example, that if
$\ua\in\Arr_{p,q}(\Q)$ is such that $d_j^{\ver}(\ua)$ satisfies condition (2), then $\ua$ satisfies condition (2) as well.

For condition (1) the argument is analogous as the one in the proof of Lemma \ref{lem:ndegsemibisimplicial}, so we omit it.
We consider now condition (2). Let $\ua\in\Arr_{p,q}(\Q)$ for some $p,q\ge0$ and let
$0\le i\le p$. The degeneracy map $s_i^{\hor}$ acts on $\ua$ by adding an additional
inner column made of $\one$'s; in particular the entries of $\ua$ are all contained
in the set of entries of $s_i^{\hor}(\ua)$, and thus if $s_i^{\hor}(\ua)$ satisfies condition (2),
then so does also $\ua$. By the same argument, the degeneracy maps $s_j^{\ver}$ reflect condition (2).

Let now $0\le i\le p$ and suppose that $d_i^{\hor}(\ua)$ satisfies condition (2).
The columns of $d_i^{\hor}(\ua)$ are obtained from those of $\ua$ by replacing the $i$\textsuperscript{th}
and $(i+1)$\textsuperscript{st} column
by the sequence of elements $a_{i,j'}^{c(\ua_{i+1})_{j'}}a_{i+1,j'}$, for $0\le j'\le q+1$; in particular
all entries of $\ua$ occur (up to conjugation) as factors of entries of $d_i^{\hor}(\ua)$.
Since $\cJ(\Q)=\hQ\setminus\Q$ is an ideal, if all entries of $d_i^{\hor}(\ua)$ lie in $\Q$, then
also all entries of $\ua$ must lie in $\Q$. Hence condition (2) is reflected by horizontal face maps.

Let finally $0\le j\le p$ and suppose that $d_j^{\ver}(\ua)$ satisfies condition (2).
The rows of $d_j^{\ver}(\ua)$ are obtained from those of $\ua$ by replacing the $j$\textsuperscript{th}
and $(j+1)$\textsuperscript{st} row
by the sequence of elements $a_{i',j}a_{i',j+1}$, for $0\le i'\le p+1$; again we notice that all entries
of $\ua'$ occur as factors of entries of $d_j^{\ver}(\ua)$, and the same argument used before shows
that condition (2) is reflected by vertical face maps.
\end{proof}

\subsection{Configurations with monodromy}
\label{subsec:simplhur}
Since $\Arr(\Q)$ is a bisimplicial $\hQ$-crossed set,
the topological space $|\Arr(\Q)|$ is a $\hQ$-crossed space, i.e. an object in $\XTop(\hQ)$.
By Lemma \ref{lem:NAdminArr}, $\NAdm(\Q)$ is a bisimplicial $\hQ$-crossed set, hence $|\NAdm(\Q)|$ is a $\hQ$-crossed space, and $|\NAdm(\Q)|\subset|\Arr(\Q)|$ is an inclusion of $\hQ$-crossed spaces.

\begin{defn}
The \emph{simplicial Hurwitz space} with coefficients in $\Q$, denoted $\Hur^{\Delta}(\Q)$,
is the $\hQ$-crossed space $|\Arr(\Q)|\setminus|\NAdm(\Q)|$.

Recall Example \ref{ex:hQascrossedset}, and note that $\hQ$, considered as a discrete space, is also the terminal object in $\XTop(\hQ)$.
We denote by $\hat\totmon\colon \Hur^{\Delta}(\Q)\to\hQ$ the unique map of $\hQ$-crossed
spaces, and call it the \emph{total monodromy}.
The right action of $\hQ$ (and hence of $\cG(\hQ)=\cG(\Q)$) on $\Hur^{\Delta}(\Q)$ is called the \emph{action by global conjugation}.
\end{defn}

A point in $\Hur^{\Delta}(\Q)$ can be interpreted as a configuration of points in the unit square
$(0,1)^2$ with the additional information of a monodromy with values in $\Q$. We briefly describe this idea
in the following, and refer to \cite{Bianchi:Hur2} for the precise construction; see also Figure \ref{fig:Ppsi}.

\begin{figure}[ht]
 \begin{tikzpicture}[scale=4,decoration={markings,mark=at position 0.38 with {\arrow{>}}}]
  \draw[dashed] (0,1) to (1,1) to (1,0);
  \draw[dashed,->] (-.3,0) to (1.3,0);
  \draw[dashed,->] (0,-1.1) to (0,1.1);
  \node at (0,-1) {$*$};
  \node at (-.05,-.05){\tiny $0$};
  \node at (1,-.05){\tiny $1$};
  \node[anchor=east] at (-.01,-1){\tiny $-\sqrt{-1}$};
  \node[anchor=east] at (-.01,1){\tiny $\sqrt{-1}$};
  \node at (.2,-.05){\tiny $s_1$};
  \node at (.6,-.05){\tiny $s_2$};
  \node[anchor=east] at (-.01,.2){\tiny $t_1\sqrt{-1}$};
  \node[anchor=east] at (-.01,.5){\tiny $t_2\sqrt{-1}$};
  \node[anchor=east] at (-.01,.9){\tiny $t_3\sqrt{-1}$};
  \draw[dotted] (0,.2) to (1,.2);
  \draw[dotted] (0,.5) to (1,.5);
  \draw[dotted] (0,.9) to (1,.9);
  \draw[dotted] (.2,0) to (.2,1);
  \draw[dotted] (.6,0) to (.6,1);
  \node at (.2,.2){$\bullet$};
  \node at (.2,.5){$\bullet$};
  \node at (.6,.5){$\bullet$};
  \node at (.6,.9){$\bullet$};
  \draw[thin, looseness=.6, postaction={decorate}] (0,-1) to[out=70,in=-90] (.3,.1) to[out=90,in=-90] (.1,.2) to[out=90,in=90] (.34,.2)
  to[out=-90,in=60] (0,-1);
  \draw[thin, looseness=.6, postaction={decorate}] (0,-1) to[out=50,in=-90] (.42,.4) to[out=90,in=-90] (.1,.5) to[out=90,in=90] (.45,.5)
  to[out=-90,in=40] (0,-1);
  \draw[thin, looseness=.9, postaction={decorate}] (0,-1) to[out=30,in=-90] (.7,.4) to[out=90,in=-90] (.55,.5) to[out=90,in=90] (.74,.5)
  to[out=-90,in=25] (0,-1);
  \draw[thin, looseness=.9, postaction={decorate}] (0,-1) to[out=20,in=-90] (.85,.8) to[out=90,in=-90] (.55,.9) to[out=90,in=90] (.9,.9)
  to[out=-90,in=15] (0,-1);  
  \draw[thin, looseness=.9,postaction={decorate}] (0,-1) to[out=83,in=-90] (.05,.55) to[out=90,in=90] (.64,.5) to[out=-90,in=-90] (.5,.5)
  to[out=90,in=90] (.07,.55) to[out=-90,in=80] (0,-1);  
  \node at (.3,.27){\tiny $\gamma_{1,1}$};
  \node at (.3,.4){\tiny $\gamma_{1,2}$};
  \node at (.73,.57){\tiny $\gamma_{2,2}$};
  \node at (.73,.8){\tiny $\gamma_{2,3}$};
  \node at (.3,.7){\tiny $\gamma$};
 \end{tikzpicture}
 \caption{A configuration in $\Hur^{\Delta}(\Q)$ of the form $(\ua;\us,\ut)$, where $\ua\in\Arr_{2,3}(\Q)$ is an
 admissible array of size $4\times 5$ whose only entries in $\Q_+$ are $a_{1,1}$, $a_{1,2}$, $a_{2,2}$ and $a_{2,3}$.
 We have $I(\ua)=\set{(1,1),(1,2),(2,2),(2,3)}$, and $P=\set{z_{1,1},z_{1,2},z_{2,2},z_{2,3}}$.
 The monodromy $\psi$ associates with the element $[\gamma_{i,j}]\in\pi_1(\CmP;*)$ the element $a_{i,j}\in\Q_+$.
 Since $[\gamma]=[\gamma_{2,2}]^{([\gamma_{1,1}][\gamma_{1,2}])^{-1}}$ in $\pi_1(\CmP;*)$, the monodromy $\psi$ associates
 with $[\gamma]$ the element $a_{2,2}^{([a_{1,1}][a_{1,2}])^{-1}}\in\Q$, where we use the adjoint action of $\cG(\Q)$ on $\Q$. 
}
 \label{fig:Ppsi}
\end{figure}

Let $\ua$ be a non-degenerate, admissible array in $\Arr_{p,q}(\hQ)$, for some $p,q\ge0$, and let
$\us=(0=s_0<s_1<\dots<s_{p+1}=1)$ and $\ut=(0=t_0<t_1<\dots<t_{q+1}=1)$ be the coordinates of a point in the interior
of $\Delta^p\times\Delta^q$. The datum of $\ua,\us$ and $\ut$ identifies a point in $\Hur^{\Delta}(\Q)$,
which we shall denote by $(\ua;\us,\ut)$.

Let $I(\ua)\subseteq\set{0,\dots,p+1}\times\set{0,\dots,q+1}$ denote the set
of pairs $(i,j)$ with $a_{i,j}\in\Q_+=\Q\setminus\set{\one}$; in fact $I(\ua)\subseteq\set{1,\dots,p}\times\set{1,\dots,q}$ because $\ua$ is admissible. Let $P\subset(0,1)^2\subset\C$ be the finite set of points of the form $z_{i,j}:=s_i+t_j\sqrt{-1}$,
for $(i,j)\in I(\ua)$. Let $*=-\sqrt{-1}$ be the basepoint of $\CmP$.

A monodromy around points of $P$ with values in $\Q$ is to be thought of as a function $\psi$ that
associates an element of $\Q_+$ with each
element $[\gamma]\in\pi_1(\CmP,*)$ represented by a
simple loop $\gamma\subset\CmP$ spinning clockwise
around precisely one of the points $z_{i,j}$.
There are usually infinitely many isotopy classes of such small simple loops $\gamma\subset\CmP$, and
it is convenient to consider only a finite collection of them, as follows.

For all $(i,j)\in I(\ua)$ let $U_{i,j}\subset(0,1)^2$ denote a small
disc around $z_{i,j}$, disjoint from $P\setminus\set{z_{i,j}}$. Then there is a unique element of $\pi_1(\CmP,*)$ that can be represented by a simple loop
$\gamma_{i,j}$ which is contained in the region
$(U_{i,j}\setminus z_{i,j})\cup\set{z\in\C\,|\,\Im(z)\le0}\cup \set{z\in\C\,|\, s_i<\Re(z)<s_{i+1}}$ and spins clockwise around $z_{i,j}$.
We set $\psi\colon[\gamma_{i,j}]\mapsto a_{i,j}$.

Note that the elements $[\gamma_{i,j}]$ exhibit $\pi_1(\CmP,*)$ as a free group.
For a simple loop $\gamma$ spinning around $z_{i,j}$ but not isotopic to $\gamma_{i,j}$, the monodromy
$\psi([\gamma])\in\Q$ can be obtained by conjugating the element $a_{i,j}$ by a suitable element of $\cG(\Q)$. These ideas are elaborated in \cite{Bianchi:Hur2};
in the Appendix \ref{sec:racks}
we will use this informal description of configurations in $\Hur^{\Delta}(\Q)$ to justify our focus on partially multiplicative \emph{quandles} $\Q$ instead of the more general notion of \emph{partially multiplicative rack}.

We can use the total monodromy to classify connected components of $\Hur^{\Delta}(\Q)$.
\begin{thm}
 \label{thm:Hurpi0}
Let $\Q$ be an augmented PMQ with completion $\hQ$. Then the map $\hat\totmon\colon\Hur^{\Delta}(\Q)\to\hQ$ induces a bijection on connected components.
\end{thm}
\begin{proof}
In the entire proof we focus on non-degenerate, admissible arrays of $\Arr(\Q)$, and abbreviate them as ``nda arrays''.
By the definition of $\Hur^{\Delta}(\Q)$, for each nda array $\ua\in \Arr_{p,q}(\Q)$ there is a geometric, open bisimplex $\mDelta^p\times\mDelta^q\subset\Hur^{\Delta}(\Q)$, which we denote by $\mDelta_{\ua}$: it is clearly a connected subspace of $\Hur^{\Delta}(\Q)$. Moreover, if $\ua'$ is another nda array of the form $d_i^{\hor}(\ua)$ or $d_i^{\ver}(\ua)$, then the open bisimplex $\mDelta_{\ua'}$ lies in the closure of $\mDelta_{\ua}$ inside $\Hur^{\Delta}(\Q)$, in particular it lies in the same connected component. We use these principle to define some operations taking as input a nda array and giving as output another nda array, such that the open bisimplices associated with input and output lie in the same connected component of $\Hur^{\Delta}(\Q)$.
\begin{enumerate}
 \item Let $\ua\in\Arr_{p,q}(\Q)$ have a column $\ua_i$, with $1\le i\le p$, containing two or more elements in $\Q_+$, and let $1\le j\le q$ be the minimal index with $a_{i,j}\neq\one$. We can define $\ua'\in\Arr_{p+1,q}(\Q)$ by replacing $\ua_i$ with the two columns $\ua'_i,\ua'_{i+1}$, where $a'_{i,j}=a_{i,j}$ is the only non-$\one$ entry in $\ua'_i$, and where $\ua'_{i+1}$ differs from $\ua_i$ only in that $a'_{i+1,j}=\one$. Then $\ua=d_i^{\hor}(\ua')$.
 \item Let $\ua\in\Arr_{p,q}(\Q)$ have at most one non-$\one$ entry in each of its columns. We can define $\ua'\in \Arr_{p,1}$ by setting $a'_{i,1}=c(\ua_i)_{q+1}$ and letting all other entries of $\ua'$ be $\one$. Then $\ua'$ is obtained from $\ua$ by applying $q-1$ vertical face maps.
\end{enumerate}
Using several times operations (1) and (2) we can transform each nda array $\ua$ into a nda array $\ua'$ lying inside $\Arr_{p,1}(\Q)$ for some $p\ge0$,\footnote{A nda array containing only $\one$'s can be transformed to the unique nda array in $\Arr_{0,0}(\Q)$ by iterated horizontal and vertical face maps, showing that $\Hur^{\Delta}(\Q)(\one)$ is connected.} so it suffices to classify connected components of bisimplices associated with the latter type of nda arrays.

A nda array $\ua\in\Arr_{p,1}$ is uniquely described by its sequence $(a_{1,1},\dots,a_{p,1})$ of elements of $\Q$, i.e. the inner part of the unique inner row. Moreover the total monodromy of each point in $\mDelta_{\ua}$ is the element $a_{1,1}\cdot\dots\cdot a_{p,1}\in\hQ$. Given two nda arrays $\ua\in\Arr_{p,1}$ and $\ua'\in\Arr_{p',1}$ with same total monodromy, we want to prove that $\mDelta_{\ua}$ and $\mDelta_{\ua'}$ lie in the same connected component of $\Hur^{\Delta}(\Q)$. The equality $a_{1,1}\cdot\dots\cdot a_{q,1}=a'_{1,1}\cdot\dots\cdot a'_{q',1}\in\hQ$ implies that it is possible to transform the sequence $(a_{1,1},\dots,a_{q,1})$ of elements of $\Q$ into the sequence $(a'_{1,1},\dots,a'_{q',1})$ by repeatedly applying the following moves (and their inverses):
\begin{itemize}
 \item standard moves: replace the pair $a_{i,1},a_{i+1,1}$ by the pair $a_{i+1,1},a_{i,1}^{a_{i+1,1}}$;
 \item multiplications: replace the pair $a_{i,1},a_{i+1,1}$ by the element $a_{i,1}a_{i+1,1}$, provided that this product is defined in $\Q$.
\end{itemize}
If the sequences $(a_{1,1},\dots,a_{q,1})$ and $(a'_{1,1},\dots,a'_{q',1})$ are connected by a multiplication, then we have $\ua'=d_i^{\hor}(\ua)$ or, vice versa, $\ua=d_i^{\hor}(\ua')$, so that $\mDelta_{\ua}$ and $\mDelta_{\ua'}$ lie in the same component of $\Hur^{\Delta}(\Q)$. It suffices therefore to prove that if $q=q'$ and if $(a'_{1,1},\dots,a'_{q,1})=(a_{1,1},\dots,a_{i-1,1},a_{i+1,1},a_{i,1}^{a_{i+1},1},\dots,a_{q,1})$, then $\mDelta_{\ua}$ and $\mDelta_{\ua'}$ are connected in $\Hur^{\Delta}(\Q)$.
For this define nda arrays $\ua^{(1)},\ua^{(2)}\in\Arr_{q,2}(\Q)$ as follows:
\begin{itemize}
 \item $\ua^{(1)}$ is obtained from the admissible, but degenerate array $s_1(\ua)\in\Arr_{2,q}$ (which is merely $\ua$ with an extra row of $\one$'s on top) by replacing the $i$\textsuperscript{th} column $(\one,a_{i,1},\one,\one)$ by the column $(\one,\one,a_{i,1},\one)$; we have $d^{\ver}_1(\ua^{(1)})=\ua$;
 \item similarly, $\ua^{(2)}$ is obtained from $s_1(\ua')\in\Arr_{2,q}$ by replacing the $(i+1)$\textsuperscript{st} column $(\one,a_{i,1}^{a_{i+1,1}},\one,\one)$ by the column $(\one,\one,a_{i,1}^{a_{i+1,1}},\one)$; we have $d^{\ver}_1(\ua^{(2)})=\ua'$.
\end{itemize}
We then note that $d_i^{\hor}(\ua^{(1)})=d_i^{\hor}(\ua^{(2)})$ is the same nda array in $\Arr_{p-1,2}(\Q)$, having $(\one,a_{i+1,1},a_{i,1},\one)$ as $i$\textsuperscript{th} column.
\end{proof}

\subsection{Functoriality of Hurwitz spaces in the PMQ}
Let $\Q$ and $\Q'$ be augmented PMQs, and let $\psi\colon\Q\to\Q'$ be an \emph{augmented} map of PMQs,
i.e. $\psi(\Q_+)\subset\Q'_+$; let $\psi\colon\hQ\to\hQ'$ denote the corresponding map between completions.
We can consider $\Arr(\Q)$ and $\Arr(\Q')$ as plain bisimplicial sets, under the forgetful
functors $U\colon ss\XSet(\hQ)\to ss\Set$ and $U\colon ss\XSet(\hQ')\to ss\Set$. Then the assignment
\[
\ua=(a_{i,j})_{i,j}\mapsto \psi(\ua)=(\psi(a_{i,j})_{i,j})
\]
defines a map of bisimplicial sets $\psi_*\colon U\Arr(\Q)\to U\Arr(\Q')$. This map sends admissible (respectively, non-degenerate)
bisimplices to admissible (respectively, non-degenerate) bisimplices; in particular it induces a map of spaces $\psi_*\colon |U\Arr(\Q)|\to|U\Arr(\Q')|$ which restricts
to a map between Hurwitz spaces
\[
\psi_*\colon \Hur^{\Delta}(\Q)\to\Hur^{\Delta}(\Q').
\]
This map is compatible with the total monodromy, i.e. $\psi_*$ restricts to a map $\Hur^{\Delta}(\Q)(a)\to\Hur^{\Delta}(\Q')(\psi(a))$ for all $a\in\hQ$.
Note that in general it is not true that $\psi_* \colon U\Arr(\Q)\to U\Arr(\Q')$
restricts to a map $U\NAdm(\Q)\to U\NAdm(\Q')$: for instance this almost never happens when $\Q'=\hQ$ and $\psi$ is the inclusion.

\subsection{Classical Hurwitz spaces as simplicial Hurwitz spaces}
We now turn our attention to the case in which $G$ is a discrete group, and $\Q$ is the PMQ $G\sqcup\set{\one_\Q}$: the unit of $\Q$ is the extra element $\one_{\Q}$, whereas the unit $\one_G$ of $G$ plays in the following no special role. We consider the trivial product on $\Q$, and conjugation on $\Q$ extends the group conjugation of $G$.
\begin{thm}
 \label{thm:HurwitzEVW}
 Let $\Q=G\sqcup\set{\one}$ as above; then
 $\Hur^{\Delta}(\Q)$ is homeomorphic to the disjoint union $\hur(G):=\coprod_{k\ge0}\hur_k(G)$ of the classical Hurwitz spaces with monodromies in $G$.
\end{thm}
\begin{proof}
 Consider first the case in which $G=\set{\one_G}$ is the trivial group; then $\Q$ is isomorphic to the abelian PMQ $\set{0,1}$ from Example \ref{ex:Ncoconnected}, where $\one_\Q=0$ and $\one_G=1$; the completion of $\Q=\set{0,1}$ is $\N$. The realisation $|\Arr(\set{0,1})|$ can be identified with the space $\SP([0,1]^2):=\coprod_{k\ge0}\SP^k([0,1]^2)$, i.e. the disjoint union of the symmetric powers of the closed unit square $[0,1]^2$, parametrising finite configurations of points in $[0,1]^2$ carriying a multiplicity in $\N$. The cell decomposition of $|\Arr(\set{0,1})|$ is a version of the Fox-Neuwirth-Fuchs cell decomposition of $\SP([0,1]^2)$. The subspace $|\NAdm(\set{0,1})|\subset|\Arr(\set{0,1})|$ corresponds to the union of the following two subspaces of $\SP[0,1]^2$:
 \begin{itemize}
  \item the \emph{fat diagonal} of $\SP([0,1]^2)$, i.e. the subspace of finite configurations having at least one point of multiplicity at least 2;
  \item the \emph{boundary} of $\SP([0,1]^2)$, i.e. the subspace of finite configurations having at least one point (of multiplicity at least 1) lying on $\del[0,1]^2$.
 \end{itemize}
 The complement $\Hur^{\Delta}(\set{0,1})=|\Arr(\set{0,1})|\setminus|\NAdm(\set{0,1})|$ is then identified with $C((0,1)^2):=\coprod_{k\ge0}C_k((0,1)^2)$, the disjoint union of the unordered configuration space of the open unit square. This is also the same as $\coprod_{k\ge0}\hur_k(\set{\one_G})$, as the monodromy, being a group homomorphism with target the trivial group, carries no information.

 For a generic group $G$, the construction from Subsection \ref{subsec:simplhur} gives rise to a continuous, bijective map $\epsilon\colon \Hur^{\Delta}(\Q)\to\hur(G)$; if we consider the augmented map of PMQs $\psi\colon\Q\to\set{0,1}$ sending $\one_\Q\mapsto 0$ and sending each element of $G$ to 1, then we obtain a commutative diagram
 \[
  \begin{tikzcd}
   \Hur^{\Delta}(\Q)\ar[dr,"\psi_*"]\ar[r,"\epsilon"] & \hur(G)\ar[d]\\
   & C((0,1)^2),
  \end{tikzcd}
 \]
where the vertical arrow is a covering map. It thus suffices to prove that also $\psi_*$ is a covering map, in order to conclude that $\epsilon$ is a homeomorphism.

The fact that $\psi_*$ is a covering map follows from the following property of the map of bisimplicial sets $\psi_*\colon U\Arr(\Q)\to U\Arr(\set{0,1})$, which is a direct consequence of the fact that $\Q$ has trivial product. Let $p,q\ge0$, let $\ua\in \Arr_{p,q}(\set{0,1})$ be a non-degenerate, admissible array, let $\star=\hor,\ver$ and let $\ua'\in\Arr(\set{0,1})$ be another non-degenerate, admissible array, satisfying $d_i^\star(\ua')=\ua$ for some $i$. Let $\tilde\ua\in \Arr_{p,q}(\Q)$ be an array with $\psi_*(\tilde\ua)=\ua$. Then there exists a unique non-degenerate, admissible array $\tilde\ua'\in\Arr(\Q)$ satisfying $\psi_*(\tilde\ua')=\ua'$ and $d_i^\star(\tilde\ua')=\tilde\ua$. The situation is summarised in the following diagram:
\[
 \begin{tikzcd}
  \exists !\ \tilde\ua'\ar[r,dashed, mapsto,"d_i^\star"]\ar[d,dashed,mapsto,"\psi_*"] &\tilde \ua \ar[d,mapsto,"\psi_*"]\\
  \ua'\ar[r,mapsto,"d_i^\star"] &\ua.
 \end{tikzcd}
\]
\end{proof}

\begin{ex}
 \label{ex:HurwitzEVW}
Let $G$ be a group and let $c\subseteq G$ be a conjugation invariant subset. Then $\Q'=c\sqcup\set{\one_{\Q}}$ is a sub-PMQ of the PMQ $\Q=G\sqcup\set{\one_{\Q}}$ from Theorem \ref{thm:HurwitzEVW}, and $\Hur^{\Delta}(\Q')$ can be identified with a union of connected components of $\Hur^{\Delta}(\Q')$.

Correspondingly, inside $\coprod_{k\ge0}\hur_k(G)$, we obtain a subspace $\coprod_{k\ge0}\hur_{k}^c(G)$, which is a union of connected components.
For fixed $k\ge0$, the space $\hur_k^c(G)$ parametrises branched $G$-coverings of $(0,1)^2$ with $k$ branch points and local monodromies in $c$. These spaces are one of the central objects of study in \cite{EVW:homstabhur,ETW:shufflealgebras,ORW:Hurwitz}.
\end{ex}

\subsection{The relative cellular chain complex}
Let $R$ be a commutative ring. In this subsection we describe the cellular chain complex
$\CH_*(|\Arr(\Q)|,|\NAdm(\Q)|;R)$.
Note that this chain complex is obtained from the $\hQ$-crossed pair of CW complex
$(|\Arr(\Q)|,|\NAdm(\Q)|)$, and as such is a $\hQ$-crossed chain complex in $R$-modules.
\begin{lem}
 \label{lem:Hurloccpt}
Let $\Q$ be a locally finite PMQ; then $\Hur^{\Delta}(\Q)$ is a locally compact topological space.
\end{lem}
\begin{proof}
Let $\ua\in\Arr_{p,q}(\Q)$ be a non-degenerate, admissible array; since $\Q$ is locally finite, then $\ua$ is the image of finitely many non-degenerate, admissible arrays in $\Arr(\Q)$ under iterated horizontal and vertical face maps: as a consequence of this observation, each point in $\Hur^{\Delta}(\Q)$ has a neighbourhood meeting only finitely many open bisimplices $\mDelta_{\ua}\subset\Hur^{\Delta}(\Q)$. This, together with the fact that $\Hur^{\Delta}(\Q)$ arises as difference of geometric realisations of two bisimplicial sets, ensures that $\Hur^{\Delta}(\Q)$ is locally compact.
\end{proof}
In particular, if $\Q$ is locally finite, then the compactly supported dual cellular cochain complex $\CH^*_{\mathrm{cpt}}(|\Arr(\Q)|,|\NAdm(\Q)|;R)$ computes the compactly supported cohomology of $\Hur^{\Delta}(\Q)$.

We will henceforth switch back to $\CH_*(|\Arr(\Q)|,|\NAdm(\Q)|;R)$ and compare this chain complex with the double bar construction applied to the algebra
pair $(R[\Q],R,\epsilon_{\Q})\in\Alg(\XMod_R(\hQ))^{[0,1]}$, where $R[\Q]$ is regarded
as a commutative algebra in $\XMod_R(\hQ)$ as in Example \ref{ex:RQ}; 
$R$ denotes the monoidal unit of
$\XMod_R(\hQ)$, and is thus an algebra in this category; and $\epsilon_{\Q} \colon R[\Q]\to R$ is the augmentation
given by $\sca a\mapsto 0$ for all $a\in\Q_+$ and $\sca{\one_{\Q}}\mapsto 1$.
\begin{thm}
\label{thm:redchain}
 The cellular chain complex $\CH_*(|\Arr(\Q)|,|\NAdm(\Q)|;R)$ is isomorphic, as $\hQ$-crossed
 chain complex in $R$-modules, to the reduced total chain complex associated with the bisimplicial
 $\hQ$-crossed $R$-module $B_{\bullet,\bullet}(R[\Q],R,\epsilon_{\Q})$.
\end{thm}
The proof of Theorem \ref{thm:redchain} is the content of the rest of the subsection.
First we note that $\CH_*(|\Arr(\Q)|;R)$ can be obtained from the bisimplicial
$\hQ$-crossed sets $\Arr(\Q)$ in a two-step process.

The first step of the process is to replace each $\Arr_{p,q}(\Q)\in\XSet(\hQ)$ with the corresponding
$R$-linearisation $R[\Arr_{p,q}(\Q)]\in\XMod_R$:
more precisely, we consider the free $R$-module functor
$R[-]\colon\Set\to \Mod_R$, which induces first a levelwise free $R$-module functor
$R[-]\colon\XSet(\hQ)\to\XMod_R(\hQ)$, and then a free $R$-module functor $R[-]\colon ss\XSet(\hQ)\to ss\XMod_R(\hQ)$. We apply the latter functor to $\Arr(\Q)$.

In a similar fashion we can apply the functor $R[-]$ to $\NAdm(\Q)\in ss\XSet(\hQ)$.
The inclusion of bisimplicial $\hQ$-crossed sets $\NAdm(\Q)\to \Arr(\Q)$ induces an inclusion of bisimplicial $\hQ$-crossed $R$-modules
$R[\NAdm(\Q)]\hookrightarrow R[\Arr(\Q)]$; note now that $ss\XMod_R(\hQ)$ is an abelian category,
because it is a category of functors into $\Mod_R$; hence we can
consider the \emph{quotient}
\[
R[\Arr(\Q)]/R[\NAdm(\Q)]\in ss\XMod_R(\hQ),
\]
given levelwise, for $p,q\ge0$, by the $\hQ$-crossed $R$-module
\[
\pa{R[\Arr(\Q)]/R[\NAdm(\Q)]}_{p,q}=R[\Arr(\Q)]_{p,q}/R[\NAdm(\Q)]_{p,q}.
\]
We have an isomorphism of free $R$-modules $R[\Arr(\Q)]_{p,q}\cong(R[\hQ])^{\otimes (p+2)\times(q+2)}$;
in other words, $R[\Arr(\Q)]_{p,q}$ can be regarded as the free $R$-module on the set of all arrays of size $(p+2)\times(q+2)$ with entries in $\hQ$;
the submodule $R[\NAdm(\Q)]_{p,q}$ is then spanned by the non-admissible arrays of size 
$(p+2)\times(q+2)$.

Keeping the previous analysis in mind, recall
Definition \ref{defn:AlgbfA}, and note that there is a map
$(R[\hQ],R[\hQ],\Id_{R[\hQ]})\to(R[\Q],R,\epsilon_{\Q})$ of algebra pairs in $\XMod_R(\hQ)$,
given by the horizontal arrows in the commutative square
\[
 \begin{tikzcd}[column sep=3cm]
  R[\hQ]\ar[r,"-/\sca{\cJ(\Q)}"]\ar[d,"\Id_{R[\hQ]}"] & R[\Q]\ar[d,"\epsilon_{\Q}"]\\
  R[\hQ]\ar[r,"\epsilon_{\hQ}"] & R.
 \end{tikzcd}
\]
Here $\epsilon_{\hQ}\colon R[\hQ]\to R$ denotes the augmentation sending $a\mapsto 0$ for $a\in\hQ_+$ and
$\one_{\hQ}\mapsto 1$; note that $\epsilon$ is a map of algebras in $\XMod_R(\hQ)$, since $\hQ$
is augmented. Similarly we define the map of algebras $\epsilon_{\Q}\colon R[\Q]\to R$.
The map $-/\sca{\cJ(\Q)}$ is the quotient by the two-sided ideal of $R[\hQ]$ generated by the elements
$\sca a$ for $a\in\cJ(\Q)=\hQ\setminus\Q$.

Applying the double bar construction we obtain a morphism of bisimplicial $\hQ$-crossed $R$-modules
\[
(-/\sca{\cJ(\Q)},\epsilon_{\hQ})\colon R[\Arr(\Q)]=B_{\bullet,\bullet}(R[\hQ],R[\hQ],\Id_{R[\hQ]})\to
B_{\bullet,\bullet}(R[\Q],R,\epsilon_{\Q}),
\]
which for fixed $p,q\ge0$ restricts to a morphism of $\hQ$-crossed $R$-modules
\[
\begin{tikzcd}
(\, -/\sca{\cJ(\Q)}\ ,\ \epsilon_{\hQ}\,)_{p,q}\ar[r,phantom,"\colon"]& R[\Arr(\Q)]_{p,q}\ar[d,phantom,"\cong"]\ar[r] & B_{p,q}(R[\Q],R,\epsilon_{\Q})\ar[d,phantom,"\cong"] \\
 &  R[\hQ]^{\otimes (p+2)\times (q+2)} & R[\Q]^{\otimes p\times q}.
\end{tikzcd}
\]
The rightmost isomorphism comes from interpreting $B_{p,q}(R[\Q],R,\epsilon_{\Q})$ as the free $R$-module
on the set of admissible arrays of size $(p+2)\times (q+2)$.
The map $(-/\sca{\cJ(\Q)},\epsilon_{\hQ})_{p,q}$ is then precisely the quotient by the submodule
spanned by non-admissible arrays, and therefore we have a short exact sequence in $ss\XMod_R(\hQ)$
\[
 0\to R[\NAdm(\Q)] \to R[\Arr(\Q)]\to B_{\bullet,\bullet}(R[\Q],R,\epsilon_{\Q})\to 0.
\]

The second step replaces a bisimplicial $R$-module $M=M_{\bullet,\bullet}$ with its
reduced, total chain complex $\bar{\CH}_\bullet(M)$; we set
\[
\bar{\CH}_n(M)=\oplus_{i+j=n}\bar M_{i,j},
\]
where $\bar M_{i,j}$ is the quotient of $M_{i,j}$ by the sum of all images of all degeneracy
maps in $M$ with target $M_{i,j}$, both horizontal and vertical.
The differential $\del\colon \bar{\CH}_n(M)\to \bar{\CH}_{n-1}(M)$
is induced on the summand $\bar M_{i,j}$ by the formula $\sum_{i'=0}^i(-1)^{i'} d_{i'}^{\hor}+(-1)^i\sum_{j'=0}^j(-1)^{j'}d_{j'}^{\ver}$.

We call $\bar{\CH}_{\bullet}(M)$ the reduced, total chain complex associated with the
bisimplicial $\hQ$-crossed $R$-module $M$: it is a $\hQ$-crossed chain complex in $R$-modules.

In particular we have an isomorphism of chain complexes
\[
 \bar{\CH}_{\bullet}\pa{R[\Arr(\Q)]/R[\NAdm(\Q)]}\cong 
 \bar{\CH}_{\bullet}\pa{B_{\bullet,\bullet}(R[\Q],R,\epsilon_{\Q})},
\]
and the left hand side is naturally identified with $\CH_*(|\Arr(\Q)|,|\NAdm(\Q)|;R)$.

\subsection{Poincar\'e PMQs}
Let $\Q$ be an augmented PMQ. As a space, $\Hur^{\Delta}(\Q)$ is the difference of two CW complexes
$|\Arr(\Q)|\setminus|\NAdm(\Q)|$, so in particular it has no evident
CW structure; this prevents us to find, for instance, a simple, combinatorial
chain complex to study the homology of $\Hur^{\Delta}(\Q)$, as for example a cellular chain complex would be.
In a particular situation, which is described in the following definition, we can circumvent this problem.

\begin{defn}
 \label{defn:Poincare}
Let $\Q$ be a locally finite PMQ with completion $\hQ$. We say that $\Q$ is
\emph{Poincar\'e} if, for all $a\in\hQ$, the space
$\Hur^{\Delta}(\Q)(a)$ is a topological manifold of some dimension.
\end{defn}
The hypothesis that $\Q$ is locally finite in Definition
\ref{defn:Poincare} is to ensure that the space $\Hur^{\Delta}(\Q)$ is locally compact, by virtue of Lemma \ref{lem:Hurloccpt}.

\begin{ex}
 \label{ex:trivialPoincare}
 Let $\Q$ be a PMQ with trivial multiplication,
 define the norm $N\colon \Q\to\N$ by setting $N(a)=1$
 for all $a\in\Q_+$, let $\hQ$ be the completion of $\Q$ and let $N\colon\hQ\to\N$ be the extension of the norm.

 Let $a\in\hQ$ and set $k=N(a)$: then $\Hur^{\Delta}(\Q)(a)$ is a covering space of $C_k((0,1)^2)$, the $k$\textsuperscript{th}
 unordered configuration space of $(0,1)^2$, by the same argument used in Theorem \ref{thm:HurwitzEVW}.
 Since $C_k((0,1)^2)$ is an orientable manifold of dimension $2n$, the same holds
 for the space $\Hur^{\Delta}(\Q)(a)$. It follows that $\Q$ is Poincar\'e. Note also that the intrinsic pseudonorm $h\colon\Q\to\N$ is in fact a norm and coincides with $N$.
\end{ex}

\begin{ex}
 \label{ex:SP}
 Let $\Q=\N$ or $\Q=\set{0,1,\dots,n}\subset\mathbb{N}$ as in Example \ref{ex:1ton}. Then
 the completion $\hQ$ of $\Q$ is canonically identified with $\N$; for all $k\in\N$
 the space $\Hur^{\Delta}(\hQ)(k)$ is homeomorphic to the $k$-fold symmetric power $\mathrm{SP}^{k}((0,1)^2)$,
 which is homeomorphic to $\R^{2k}$. It follows that $\Q$ is Poincar\'e. Note also that the intrinsic
 pseudonorm $h\colon\Q\to\N$ is a norm and coincides with the natural inclusion of $\Q$ into $\N$. For $n\ge2$,
$\Q=\set{0,1,\dots,n}$ is an example of a Poincar\'e but not Koszul PMQ.
\end{ex}
\begin{ex}
\label{ex:segreagain}
Let $\Q$ be as in Example \ref{ex:segre}. Then $\Hur^{\Delta}(\Q)(c)$ is homeomorphic
to the union of two copies of $(0,1)^2\times(0,1)^2$ along their diagonal subspace
$(0,1)^2\subset(0,1)^2\times(0,1)^2$.
In fact, $\Hur^{\Delta}(\Q)(c)$ can be regarded as the union of
$\Hur^{\Delta}(\set{\one,a,b,c})(c)$ and
$\Hur^{\Delta}(\set{\one,a',b',c})(c)$ along $\Hur^{\Delta}(\set{\one,c})(c)$.

If $\fc\in\Hur^{\Delta}(\Q)(c)$ lies in $\Hur^{\Delta}(\set{\one,c})(c)$, then the local homology
group
\[
H_i\pa{\Hur^{\Delta}(\Q)(c)\,,\, \Hur^{\Delta}(\Q)(c)\setminus\set{\fc}\ ;\ R }
\]
is isomorphic to $R$ for $i=3$ and to $R^2$ for $i=4$. It follows that $\Q$ is not Poincar\'e, although it is Koszul.
\end{ex}

\begin{prop}
\label{prop:Poincareintrinsic}
 Let $\Q$ be a Poincar\'e PMQ. Then $\Q$ is maximally decomposable and admits an intrinsic norm
 $h\colon\Q\to\N$ in the sense of Definitions \ref{defn:irreducible} and \ref{defn:intrinsicnorm}.
 
 Moreover for all
 $b\in\hQ$ the space $\Hur^{\Delta}(\Q)(b)$ is an oriented topological manifold of dimension $2h(b)$,
 where we extend the intrinsic norm to $h\colon\hQ\to \N$.
\end{prop}
\begin{proof}
Let $b\in\Q$. Since $\Q$ is locally finite, $b$ can be decomposed in finitely many ways
as a product $b_1\dots b_r$ of elements of $\Q_+$. Fix a maximal decomposition
$b=\bar b_1\dots \bar b_r$, with all $\bar b_i\in\Q_+$ irreducible, for some $r\ge0$; then we can define an
admissible, non-degenerate array $\ua\in\Arr_{r,r}(\Q)(b)$ by setting $a_{i,i}=\bar b_i$
for all $1\le i\le r$, and all other entries $a_{i,j}=\one$.

The array $\ua$ is of maximal dimension among the admissible, non-degenerate arrays in $\Arr(\Q)(b)$,
so it corresponds to a maximal open cell of the cell stratification of $\Hur^{\Delta}(\Q)(b)$.
It follows that $\Hur^{\Delta}(\Q)(b)$ is a manifold of dimension $2r$. The same equality
holds for every other maximal decomposition of $b$ in elements of $\Q_+$: this implies
both that $\Hur^{\Delta}(\Q)(b)$ is a manifold of dimension $2h(b)$ (in the sense of Definition
\ref{defn:irreducible}), and that $h\colon\Q\to \N$ satisfies the conditions of Definition \ref{defn:intrinsicnorm}
and is thus a norm on $\Q$.

Let $h\colon\hQ\to \N$ be the extension of the intrinsic norm, and let $b\in\hQ$; we can again fix a maximal
decomposition $b=\bar b_1\dots \bar b_{h(b)}$, yielding by the same argument a maximal cell in the cell stratification
of $\Hur^{\Delta}(\Q)(b)$. This implies that $\Hur^{\Delta}(\Q)(b)$, which is 
a manifold, must be a manifold of dimension $2h(b)$.

We are left to check orientability of $\Hur^{\Delta}(\Q)(b)$ for $b\in\hQ$.
Let $\Q_{\le1}\subset\Q$ be the subset of elements of norm $\le1$, and note that $\Q_{\le1}$
has a natural structure of PMQ with trivial multiplication. The inclusion $\Q_{\le1}\subset\Q$ is a map of augmented PMQs, so it induces a map of Hurwitz spaces $\Hur^{\Delta}(\Q_{\le1})\to\Hur^{\Delta}(\Q)$.

This map is injective, since the map of bisimplicial sets $\Arr(\Q_{\le1})\to\Arr(\Q)$ restricts to an injective map between admissible arrays; note however that the completion
$\widehat{\Q_{\le1}}$ of $\Q_{\le1}$ a priori only surjects and is not in bijection with $\hQ$, so a priori we only know that $\Hur^{\Delta}(\Q_{\le1})\to\Hur^{\Delta}(\Q)$ is surjective on connected components, thanks to Theorem \ref{thm:Hurpi0}.

The complement of the image of $\Hur^{\Delta}(\Q_{\le1})\hookrightarrow|\Arr(\Q)|$ is the geometric
realisation of the sub-bisimplicial complex of $\Arr(\Q)$ spanned by all arrays with
at least one entry of norm $\ge2$. It follows that we can regard
$\Hur^{\Delta}(\Q_{\le1})$ as an open subspace of $|\Arr(\Q)|$, and hence as an open
subspace of $\Hur^{\Delta}(\Q)$.

Let $\Hur^{\Delta}(\Q_{\le1})(b)$
denote the intersection of $\Hur^{\Delta}(\Q_{\le1})$ with $\Hur^{\Delta}(\Q)(b)$. Note that this is a priori
an abuse of notation: since $b$ is an element of $\hQ$ and not of $\widehat{\Q_{\le1}}$, there may be a priori several $b'\in\widehat{\Q_{\le1}}$ mapping to $b$ under $\widehat{\Q_{\le1}}\to\hQ$.

Note now that all maximal cells of the cell stratification of $\Hur^{\Delta}(\Q)(b)$ correspond to arrays
$\ua\in\Arr_{h(b),h(b)}(\Q)(b)$ of the following special form: there exists
a maximal decomposition $b=\bar b_1\dots\bar b_{h_b}$ of $b$ into irreducible elements of $\Q_+$, and there
exists a permutation $\sigma\in\fS_{h(b)}$, such that $a_{i,h(i)}=\bar b_i$ for all $1\le i\le h(b)$,
and all other entries $a_{i,j}$ are equal to $\one$. All such arrays are in the image of
$\Arr(\Q_{\le1})\to\Arr(\Q)$, because all $\bar b_i$ belong to $\Q_{\le1}$; hence
$\Hur^{\Delta}(\Q_{\le1})(b)\subset\Hur^{\Delta}(\Q)(b)$ is a \emph{dense} open subset.

Finally, note that in fact all cells of $\Hur^{\Delta}(\Q)(b)$ of dimension $\ge 2h(b)-1$
correspond to arrays $\ua\in\Arr_{h(b),h(b)-1}(\Q)$ or $\ua\in\Arr_{h(b)-1,h(b)}$
with a similar description as above, but allowing precisely one inner row or one inner column with
two entries different from $\one$. In particular the complement of
$\Hur^{\Delta}(\Q_{\le1})(b)$ in $\Hur^{\Delta}(\Q)(b)$ has codimension $\ge2$ in the manifold
$\Hur^{\Delta}(\Q)(b)$.

This implies that $\Hur^{\Delta}(\Q_{\le1})(b)$ is in fact connected, and moreover it allows us to check only that $\Hur^{\Delta}(\Q_{\le1})(b)$ is orientable; the latter statement follows from Example \ref{ex:trivialPoincare}.
\end{proof}

If $\Q$ is a Poincar\'e PMQ and $R$ is a commutative ring,
then for all $a\in\hQ$ we can apply Poincar\'e-Lefschetz duality and obtain
an isomorphism
\[
 H_*(\Hur^{\Delta}(\Q)(a);R)\cong H^{2h(a)-*}_{\mathrm{cpt}}(\Hur^{\Delta}(\Q)(a);R);
\]
here $H^{2h(a)-*}_{\mathrm{cpt}}$ denotes cohomology with compact support.
This can be computed, in principle, using the compactly supported cellular cochain complex
of the couple $\pa{|\Arr(\Q)|,|\NAdm(\Q)|}$: the cochain complex $\CH^*_{\mathrm{cpt}}\pa{|\Arr(\Q)|,|\NAdm(\Q)|;R}$
consists of free $R$-modules, and has a
generator $\ua^*$ in degree $p+q-h(b)$ for every admissible, non-degenerate array $\ua\in\Arr_{p,q}(\Q)(b)$.

\begin{prop}
 \label{prop:Poincarecoconnected}
 Let $\Q$ be a Poincar\'e PMQ. Then $\Q$ is coconnected.
\end{prop}
\begin{proof}
 Let $\Q_{\le1}$ be as in the proof of Proposition \ref{prop:Poincareintrinsic}, let $b\in\Q$
 and consider the open, dense subspace $\Hur^{\Delta}(\Q_{\le1})(b)\subset\Hur^{\Delta}(\Q)(b)$.
 We observe the following:
 \begin{itemize}
  \item The connected components of $\Hur^{\Delta}(\Q_{\le1})(b)$ are in bijection with the equivalence
  classes of decompositions $(\bar b_1,\dots,\bar b_{h(b)})$ of $b$ with respect to $\Q_1$, where two decompositions
  are considered equivalent if they are connected by a sequence of standard moves (see Definition
  \ref{defn:standardmove}, and compare with Definition \ref{defn:coconnected}).
  \item The space $\Hur^{\Delta}(\Q)(b)$ is connected: this follows from the fact that
  there is a unique admissible array of minimal dimension in $\Arr(\Q)(b)$, namely
  the array $\ua\in\Arr_{1,1}(\Q)$ with $a_{1,1}=b$ and all other entries equal to $\one$;
  this array is the image of every other admissible array in $\Arr(\Q)(b)$
  under iterated horizontal and vertical face maps.
 \end{itemize}
 As argued in the proof of Proposition \ref{prop:Poincareintrinsic}, the open inclusion
 $\Hur^{\Delta}(\Q_{\le1})(b)\subset\Hur^{\Delta}(\Q)(b)$ has complement of codimension $\ge2$,
 hence it is a bijection on $\pi_0$.
\end{proof}
We conclude the section by considering a locally finite and coconnected PMQ $\Q$, and by defining a canonical fundamental homology class for the couple of spaces $(|\Arr(\Q)(b)|,|\NAdm(\Q)(b)|)$, for all $b\in\hQ$.
In the following we understand integral coefficients for homology.

Let $\Q_{\le1}$ be as in the proof of Proposition \ref{prop:Poincareintrinsic}, and recall that by Lemma \ref{lem:coconnected} the completions of $\Q_{\le1}$ and $\Q$ agree. Let $b\in\hQ$,
and consider again the open subspace $\Hur^{\Delta}(\Q_{\le1})(b)\subset\Hur^{\Delta}(\Q)(b)$.

By Example \ref{ex:trivialPoincare}
the space $\Hur^{\Delta}(\Q_{\le1})(b)$ is a finite covering of the unordered configuration space
$C_{h(b)}((0,1)^2)$, in particular it admits a natural structure of complex manifold
and thus a canonical orientation. Moreover, since $\Q$ is coconnected, $\Hur^{\Delta}(\Q_{\le1})(b)$
is connected, i.e. there is a canonical ground class in $H_0(\Hur^{\Delta}(\Q_{\le1})(b))$.
By Poincar\'e-Lefschetz duality we get a canonical
fundamental class in the homology group
$H_{2h(b)}\pa{|\Arr(\Q)(b)|\,,\,|\Arr(\Q)(b)|\setminus \Hur^{\Delta}(\Q_{\le1})(b)}$.

Note now that there is a canonical isomorphism of homology groups
\[
\begin{tikzcd}
H_{2h(b)}\pa{|\Arr(\Q)(b)|\,,\,|\NAdm(\Q)(b)|}\ar[d,equal]\\
H_{2h(b)}\pa{|\Arr(\Q)(b)|\,,\,|\Arr(\Q)(b)|\setminus \Hur^{\Delta}(\Q)(b)}
\ar[d,"\cong"]\\
H_{2h(b)}\pa{|\Arr(\Q)(b)|\,,\,|\Arr(\Q)(b)|\setminus \Hur^{\Delta}(\Q_{\le1})(b)}
\end{tikzcd}
\]
due to the fact that all bisimplices of $|\Arr(\Q)(b)|$ of dimension $\ge 2h(b)-1$ are contained
in $\Hur^{\Delta}(\Q_{\le1})_b\subset\Hur^{\Delta}(\Q)(b)$.

\begin{defn}
\label{defn:fundamentalclass}
Let $\Q$ be a coconnected PMQ and $b\in\hQ$. We denote by
\[
[\Arr(\Q)(b),\NAdm(\Q)(b)]\in H_{2h(b)}\pa{|\Arr(\Q)(b)|,|\NAdm(\Q)(b)|;\Z}
\]
the fundamental class constructed above. It is a generator of the homology group $H_{2h(b)}\pa{|\Arr(\Q)(b)|,|\NAdm(\Q)(b)|;\Z}$;
more generally, for any commutative ring $R$, changing the coefficients from $\Z$ to $R$ yields a canonical
fundamental class
\[
[\Arr(\Q)(b),\NAdm(\Q)(b)]\in H_{2h(b)}\pa{|\Arr(\Q)(b)|,|\NAdm(\Q)(b)|;R}\cong R,
\]
generating the top homology group as a free $R$-module of rank 1.
\end{defn}

\part{Geodesic PMQs associated with symmetric groups}
\section{Geodesic PMQs associated with symmetric groups}
\label{sec:fSgeo}
The main examples of PMQs that we study in this series of articles come from symmetric groups, considered as normed groups.
In this section, for $d\ge2$, we
consider the PMQ $\fS_d\geo$, coming from the symmetric group on $d$ letters
endowed with the word length norm with respect to all transpositions. This is the specialisation to symmetric groups of the \emph{absolute length}, a well-studied notion for arbitrary Coxeter groups; see for instance \cite[Section 2.4]{Armstrong}.

We will show that $\fS_d\geo$ is coconnected, pairwise determined and Koszul.
In fact $\fS_d\geo$ is also Poincar\'e (see Definition \ref{defn:Poincare}): the proof of this fact will be given in \cite[Theorem 4.1]{Bianchi:Hur4}.

The importance of the PMQs $\fS_d\geo$ relies on the connection between
the Hurwitz spaces $\Hur^{\Delta}(\fS_d\geo)$, for varying $d\ge2$,
and the moduli spaces $\fM_{g,n}$ of Riemann surfaces of genus $g$ with $n$ ordered and parametrised boundary curves; this connection will be given in \cite{Bianchi:Hur4}.

\subsection{The PMQ \texorpdfstring{$\fS_d\geo$}{Sdgeo} and its enveloping group}

\begin{defn}
\label{defn:symgroup}
For all $d\geq 2$ we will denote by $\fS_d$ the group
of permutations of the set $[d]=\set{1,\dots,d}$.
For $i\neq j\in[d]$ we denote $(i,j)\in\fS_d$ the transposition that exchanges $i$ and $j$.
We consider on $\fS_d$ the word length norm $N$ with respect to the generating set of all transpositions:
for $\sigma\in\fS_d$, $N(\sigma)$ is the smallest $m\geq 0$ such that there exist transpositions $\tr_1,\dots,\tr_m\in\fS_d$
with $\sigma=\tr_1\dots\tr_m$.
\end{defn}
Note that the sign of a permutation $\sigma$ (i.e. the image of $\sigma$ under
the unique surjective map of groups $\fS_d\to \set{\pm 1}$)
can be written as $(-1)^{N(\sigma)}$.
By Definition \ref{defn:Ggeo} we obtain a PMQ $\fS_d\geo$. Our first aim is
to compute the group $\cG(\fS_d\geo)$.

\begin{lem}
 \label{lem:cGfSdgeo}
 Let $d\geq 2$ and let $\tfS_d\subset \Z\times\fS_d$ be the index 2 subgroup containing pairs
 $(r,\sigma)\in\Z\times\fS_d$
 such that $r$ has the same parity as $\sigma$.
 
 Then the norm and the map $\epsilon\geo$ from Definition \ref{defn:epsilongeo} give an injection of groups
 \[
  (\cG(N),\epsilon\geo)\colon \cG(\fS_d\geo)\to \Z\times\fS_d
 \]
 with image the subgroup $\tfS_d$.
\end{lem}
\begin{proof}
The group $\cG(\fS_d\geo)$ is generated by all elements of the form $[\sigma]$, where $\sigma\in\fS_d\geo$.
We have $(\cG(N),\epsilon\geo)([\sigma])=(N(\sigma),\sigma)\in\tfS_d$, hence the image of $(N,\epsilon)$ is contained
in $\tfS_d$. Moreover, if $\tr\in\fS_d$ is any transposition, then $(\cG(N),\epsilon\geo)([\tr]^2)=(2,\one)$,
and the family of elements of the form $(N(\sigma),\sigma)$, together with $(2,\one)$, generate $\tfS_d$.

Next we show that $(\cG(N),\epsilon\geo)$ is injective. First, note that the equality $\one=\one\cdot\one$ in $\fS_d\geo$
gives an equality $[\one]=[\one]^2$ in $\cG(\fS_d\geo)$, so that $[\one]=\one$ in $\cG(\fS_d\geo)$. Note also that every
$\sigma\in\fS_d\geo$ with $\sigma\neq\one$ can be written as a product $\tr_1\cdot\dots\cdot\tr_{N(\sigma)}$
in $\fS_d\geo$, hence the generator $[\sigma]\in\cG(\fS_d\geo)$ is also equal to $[\tr_1]\cdot\dots[\tr_{N(\sigma)}]$.

This shows that the elements $[\tr]$, for $\tr$ varying in the transpositions of $\fS_d$, generate
$\cG(\fS_d\geo)$.

Second, we show that for all transpositions $\tr,\tr'\in\fS_d$, the equality
$[\tr]^2=[\tr']^2$ holds in $\cG(\fS_d\geo)$. Let $\sigma\in\fS_d$ with $\tr'=\sigma\tr\sigma^{-1}$.
Then $[\tr']^2=[\sigma][\tr]^2[\sigma]^{-1}$, using twice the relation
$[\tr']=[\sigma][\tr][\sigma]^{-1}$. On the other hand we have $[\tr][\sigma][\tr]^{-1}=[\sigma^{\tr}]$
and $[\tr][\sigma^{\tr}][\tr]^{-1}=[(\sigma^{\tr})^{\tr}]=[\sigma]$, hence $[\tr]^2$
and $[\sigma]$ commute. This shows that $[\tr]^2=[\tr']^2$.

We denote by $\mathfrak{z}$ the element $[\tr]^2\in\cG(\fS_d\geo)$, which is the same element
for any transposition $\tr\in\fS_d$ and is central in $\cG(\fS_d\geo)$. Note also that $\mathfrak{z}$ has infinite order,
since $(\cG(N),\epsilon\geo)(\mathfrak{z})=(2,\one)$. We have a commutative diagram of central extensions.
\[
 \begin{tikzcd}
  1\ar[r] & \left<\mathfrak{z}\right> \ar[r] \ar[d,"(N{,}\epsilon)"] &
  \cG(\fS_d\geo) \ar[r]\ar[d,"(\cG(N){,}\epsilon\geo)"] & \cG(\fS_d\geo)/ \left<\mathfrak{z}\right> \ar[r]\ar[d,"(\cG(N){,}\epsilon\geo)"] & 1\\
  1\ar[r] & \left<(2,\one)\right> \ar[r] & \tfS_d \ar[r] & \tfS_d / \left<(2,\one)\right>\ar[r] & 1.
 \end{tikzcd}
\]
The left vertical map is an isomorphism of infinite cyclic groups. The right vertical map is also an isomorphism:
first note that there is an isomorphism  $\tfS_d / \left<(2,\one)\right>\cong\fS_d$ induced by projection
on the second coordinate; second, recall that $\fS_d$ can be given a presentation
with the following generators and relations:

\begin{tabular}{ll}
 \textbf{Generators} & For all distinct $i,j\in [d]$ there is a generator $(i,j)=(j,i)$.\\
 \textbf{Relations} &$\bullet$ $(i,j)^2=\one$ for all distinct $i,j$.\\
  &$\bullet$ $(k,l)(i,j)(k,l)^{-1}=(i,j)$ for all distinct $i,j,k,l$.\\
  &$\bullet$ $(i,j)(j,k)(i,j)^{-1}=(j,k)(i,j)(j,k)^{-1}$ for all distinct $i,j,k$.\\
\end{tabular}

The relations of second and third type also hold between the elements $[(i,j)]$, that generate $\cG(\fS_d\geo)$; by quotienting $\cG(\fS_d\geo)$ by the subgroup $\left<\mathfrak{z}\right>$,
we impose also the relations of first type: hence also $\cG(\fS_d\geo)/ \left<\mathfrak{z}\right>$
is isomorphic to $\fS_d$, and the right vertical map is an isomorphism.
By the five lemma, the middle vertical map is an isomorphism as well.
\end{proof}
In particular the group $\cK(\fS_d\geo)$ (see Notation \ref{nota:cK}) can be identified
with the subgroup $2\Z\subset\tfS_d$.

\subsection{Some classical results about symmetric groups}

In this subsection we recollect some classical facts, most of which go back
to Clebsch \cite{Clebsch} and Hurwitz \cite{Hurwitz}, about sequences of transpositions and standard
moves. We do this for the sake of completeness, and we leave
some details to the reader.

\begin{lem}
\label{lem:formulaN}
Let $\sigma\in\fS_d$ be a permutation having a cycle decomposition with $k$ cycles $c_1,\dots,c_k$,
where fixpoints of $\sigma$ count as cycles of length 1; then $N(\sigma)=d-k$. 
\end{lem}
\begin{proof}
 Any collection of transpositions $\set{\tr_1,\dots,\tr_r}\subset\fS_d$ generates a subgroup of $\fS_d$
 of the form $\fS_{\fP_1}\times\dots\fS_{\fP_\ell}\subset\fS_d$, for some partition $(\fP_1,\dots,\fP_\ell)$
 of $[d]$. The pieces of the partition $\fP_i$ are the connected components of the graph 
 having as vertices the elements of $[d]$, and having one edge between $j$ and $j'$ for any transposition $\tr_i=(j,j')$.
 The number $\ell$ of connected components of this graph is at least $d-r$, so we get the inequality
 $r\ge d-\ell$.
 
 Let now $\sigma=\tr_1\dots\tr_r$ be a factorisation of $\sigma$ exhibiting $r=N(\sigma)$.
 Then each cycle $c_i$, considered as a subset of $[d]$,
 is contained in some piece of the partition $(\fP_1,\dots,\fP_\ell)$ of $[d]$ associated with the set of transpositions $\set{\tr_1,\dots,\tr_r}$, since the action of $\sigma$ is transitive on the set $c_i$
 and since $\sigma\in\fS_{\fP_1}\times\dots\fS_{\fP_\ell}$. This implies that $k\ge\ell$, hence
 $d-\ell\ge d-k$. Putting together the two inequalities, we obtain $N(\sigma)\ge d-k$.
 
 On the other hand, one can factor each cycle $c_i$ into $|c_i|-1$ transpositions, and 
 get in this way
 a factorisation of $\sigma$ into precisely $\sum_{i=1}^k(|c_i|-1)=d-k$ transpositions.
 See for instance the monotone decomposition described below.
\end{proof}

\begin{lem}
 \label{lem:geodesictransposition}
 Let $\sigma\in\fS_d$ be a permutation and $\tr=(i,i')\in\fS_d$ be a transposition.
 Then $N(\tr\sigma)=N(\sigma)+1$
 if $i$ and $i'$ belong to different cycles of the cycle decomposition of $\sigma$, and $N(\tr\sigma)=N(\sigma)-1$
 otherwise.
\end{lem}
\begin{proof}
 Let $c$ and $c'$ be the cycles containing $i$ and $i'$ respectively. If $c$ and $c'$ are distinct cycles,
 then the cycle decomposition of $\sigma\tr$ consists of all cycles of $\sigma$ different from $c$ and $c'$,
 plus a cycle $\tilde c$ which is a \emph{concatenation} of $c$ and $c'$:
 \[
  \tilde c=(i\,,\,\sigma(i)\,,\,\sigma^2(i),\dots,\sigma^{-1}(i)\,,\,j\,,\,\sigma(j),\dots,\sigma^{-1}(j)).
 \]
 The statement follows in this case from the formula for $N$ in terms of the number of cycles, given in Lemma
 \ref{lem:formulaN}.
 
 If instead $c=c'$, then the cycle decomposition of $\sigma\tr$ consists of all cycles of $\sigma$ different
 from $c$, plus two cycles $\hat c$ and $\hat c'$ giving a \emph{splitting} of $c$, and containing $i$ and $i'$
 respectively:
 \[
  \hat c=(i\,,\,\sigma(i),\dots,\sigma^{-1}(j));\quad\quad \hat c'=(j\,,\,\sigma(j),\dots,\sigma^{-1}(i).
 \]
 Again the statement follows from the formula for $N$ from Lemma
 \ref{lem:formulaN}.
\end{proof}

The same result holds in Lemma \ref{lem:geodesictransposition} if we consider the product $\sigma\tr$ instead of $\tr\sigma$. The following corollary is a direct generalisation of Lemma \ref{lem:geodesictransposition}.
\begin{cor}
 \label{cor:geodesicpair}
 Let $\sigma,\tau\in\fS_d$ and suppose that $N(\sigma\tau)=N(\sigma)+N(\tau)$. Let $c$ be a cycle in
 the cycle decomposition of $\sigma$, and consider $c$ as a subset of $[d]$. Then
 $c$ is contained in some cycle $c'$ of the cycle decomposition of $\sigma\tau$.
\end{cor}
\begin{proof}
 Let $\tau=\tr_1\dots\tr_r$ be a minimal decomposition of $\tau$ in transpositions, with $r=N(\tau)$.
 Then the hypothesis $N(\sigma\tau)=N(\sigma)+N(\tau)$ and the triangular inequality imply that, for all
 $0\leq j\leq r$ we also have $N(\sigma\tr_1\dots\tr_j)=N(\sigma)+j$. In particular for $1\leq j\leq r$
 we have the equality $N(\sigma\tr_1\dots\tr_{j-1})+1=N(\sigma\tr_1\dots\tr_j)$, which by Lemma
 \ref{lem:geodesictransposition} implies that each cycle of $\sigma\tr_1\dots\tr_{j-1}$ is contained
 in some cycle of $\sigma\tr_1\dots\tr_j$. In particular each cycle of $\sigma$ is contained in some
 cycle of $\sigma\tau$.
\end{proof}

\begin{defn}
 \label{defn:ht}
The \emph{height} of a permutation $\sigma\in\fS_d$, 
denoted by $\vht(\sigma)$, is the greatest index $i\in[d]$ such that $\sigma(i)\neq i$.

\end{defn}
Each permutation $\sigma\in\fS_d$ admits a unique decomposition $\sigma=\tr_1\dots,\tr_{N(\sigma)}$ into transpositions
with $\vht(\tr_1)<\dots<\vht(\tr_{N(\sigma)})$. We call this the \emph{monotone decomposition} of $\sigma$ into transpositions.

The monotone decomposition can be computed recursively
by setting $\tr_{N(\sigma)}:=(\vht(\sigma),\sigma^{-1}(\vht(\sigma)))$,
and noting that $\sigma':=\sigma\tr_{N(\sigma)}$ is a permutation of norm $N(\sigma')=N(\sigma)-1$
and height $\vht(\sigma')<\vht(\sigma)$. In fact the transposition $(\vht(\sigma),\sigma^{-1}(\vht(\sigma)))$
is the unique possible choice to ensure that $\vht(\sigma')<\vht(\sigma)$.
\begin{nota}
 \label{nota:monotonedec}
 We denote by $(\tr_1^\sigma,\dots,\tr_{N(\sigma)}^\sigma)$ the monotone decomposition of
 a permutation $\sigma\in\fS_d$. It is the empty sequence for $\sigma=\one$.
\end{nota}

\begin{lem}
\label{lem:weakClebsch}
 Let $(\tr_1,\dots,\tr_r)$ be a decomposition of $\sigma\in\fS_d$ into transpositions, witnessing $r=N(\sigma)$.
 Then there is a sequence of standard moves transforming $(\tr_1,\dots,\tr_r)$ into
 $(\tr_1^\sigma,\dots,\tr_r^\sigma)$.
\end{lem}
\begin{proof}
 It suffices to find a sequence of standard moves transforming $(\tr_1,\dots,\tr_r)$ into a decomposition
 $(\tr'_1,\dots,\tr'_r)$ of $\sigma$ with $\vht(\tr'_r)=\vht(\sigma)$ and $\vht(\tr'_i)<\vht(\sigma)$ for all $1\le i\le r-1$:
 then the permutation $\sigma'=\sigma\tr'_r=\tr'_1\dots\tr'_{r-1}$ satisfies $\vht(\sigma')<\vht(\sigma)$, whence
 $\tr'_r=\tr^\sigma_r$; we can then proceed by induction on $\sigma'$, which has both smaller height and smaller norm
 than $\sigma$.

 Note that by Lemma \ref{lem:geodesictransposition} each transposition $\tr_i$ satisfies $\vht(\tr_i)\le \vht(\sigma)$.
 Suppose that there is an index $i$ with $\vht(\tr_i)=\vht(\sigma)>\vht(\tr_{i+1})$; then we can replace $(\tr_i,\tr_{i+1})$
 by $(\tr_{i+1},\tr_i^{\tr_{i+1}})$, thus moving the transposition with maximal height to right.
 
 By repeating
 this procedure, we can assume that there is an index $1\le i\le r$ such that $\tr_1,\dots,\tr_{i-1}$ have height
 strictly less than $\vht(\sigma)$, whereas $\tr_i,\dots,\tr_r$ have height equal to $\vht(\sigma)$. Suppose $i<r$ and note that
 $\tr_i\neq\tr_{i+1}$, otherwise their product would be $\one\in\fS_d$ and thus the norm of $\sigma$ would be at most $r-2$.
 By a standard  move we can replace $(\tr_i,\tr_{i+1})$ by $(\tr_{i+1}^{\tr_i},\tr_i)$ (here we rewrite
 as $\tr_i$ the exponent, which should be $\tr_i^{-1}$). The crucial observation
 is that $\tr_{i+1}^{\tr_i}$ has height strictly less than $\vht(\tr_i)=\vht(\tr_{i+1})=\vht(\sigma)$.
 
 By repeating the last procedure, we can assume that the only transposition $\tr_i$ satisfying $\vht(\tr_i)=\vht(\sigma)$
 is $\tr_r$, as desired.
\end{proof}

\begin{lem}
\label{lem:anyjJattheend}
 Let $(\tr_1,\dots,\tr_r)$ be a sequence of transpositions in $\fS_d$ that generate $\fS_d$ as a group. Let
 $(j,j')$ be a transposition in $\fS_d$. Then there is a sequence of standard moves transforming
 $(\tr_1,\dots,\tr_r)$ into a sequence $(\tr'_1,\dots,\tr'_r)$ with $\tr'_r=(j,j')$.
\end{lem}
\begin{proof}
 The statement is obvious for $d=2$, so we assume $d\ge3$. Note that it suffices to prove the statement
 in the case $j'=d$, as the statement is invariant under conjugation of permutations in $\fS_d$. So we wish
 to achieve a sequence of transpositions ending with $(d,j)$, for a fixed $1\le j\le d-1$.
 
By using the procedures described in the proof of Lemma \ref{lem:weakClebsch}, we can operate a sequence
of standard moves and reach a sequence $(\hat\tr_1,\dots,\hat\tr_r)$ for which there exists an index
$1\le i\le r-1$ satisfying the following:
\begin{itemize}
 \item the transpositions $\hat\tr_1,\dots,\hat\tr_i$ have height $<d$;
 \item $\hat\tr_{i+1}=\dots=\hat\tr_r=(d,J)$ for some $1\le J\le d-1$.
\end{itemize}
If $J=j$ we are done; otherwise we use the inductive hypothesis on the sequence
$\hat\tr_1,\dots,\hat\tr_i$, which generates $\fS_{d-1}$: possibly after a suitable sequence
of standard moves on these $i$ transpositions, we can assume that $\hat\tr_i=(j,J)$. We can now
replace $(\hat\tr_i,\hat\tr_{i+1})=((j,J),(J,d))$ with $((J,d),(d,j))$; we can then move $(d,j)$
to the end of the sequence, replacing thus each further occurrence of $(J,d)$ with $(j,J)$.
\end{proof}

\begin{lem}
\label{lem:onlyhtsigmaJatend}
Let $\sigma\in\fS_d$ be a permutation of height $d$, let $(\tr_1,\dots,\tr_r)$ be a sequence of transpositions
generating $\fS_d$, such that $\sigma=\tr_1\dots\tr_r$. Then there is a sequence of standard
moves transforming $(\tr_1,\dots,\tr_r)$ into a sequence $(\tr'_1,\dots,\tr'_r)$ with $\tr'_1,\dots,\tr'_{r-1}$
of height $<d$, and $\tr'_r=(d,\sigma^{-1}(d))$.
\end{lem}
\begin{proof}
The statement is obvious for $d=2$, so assume $d\ge3$.
Use the procedure of Lemmas \ref{lem:weakClebsch} and \ref{lem:anyjJattheend}, and note that
the new sequence $(\hat\tr_1,\dots,\hat\tr_r)$ satisfies
$\hat\tr_{i+1}=\dots=\hat\tr_r=(d,\sigma^{-1}(d))$ for some $i<r$ with $r-i$ odd. If $r-i=1$ we are done,
so assume $r-i\ge3$.

Fix any $j<d$ with $j\neq\sigma^{-1}(d)$. By Lemma \ref{lem:anyjJattheend}
we can assume $\hat\tr_i=(j,\sigma^{-1}(d))$, possibly after applying a sequence of standard
moves on $(\hat\tr_1,\dots,\hat\tr_i)$: in fact the transpositions
$\hat\tr_1,\dots,\hat\tr_i$ generate $\fS_{d-1}$.
By a suitable sequence of 4 standard moves we can then replace
\[
(\hat\tr_i,\hat\tr_{i+1},\hat\tr_{i+2})=((j,\sigma^{-1}(d)),(d,\sigma^{-1}(d)),(d,\sigma^{-1}(d)))
\]
with $((j,\sigma^{-1}(d)),(d,j),(d,j))$. We obtain a sequence $(\tilde\tr_1,\dots,\tilde\tr_r)$
such that $\tilde\tr_1,\dots,\tilde\tr_i$ have height $<d$, whereas $\tilde\tr_{i+1},\dots,\tilde\tr_r$
have height $d$ and are not all equal.

We can now repeat the second part of the procedure of Lemma \ref{lem:weakClebsch} to the sequence
$(\tilde\tr_1,\dots,\tilde\tr_r)$, obtaining a sequence $(\tr'_1,\dots,\tr'_r)$ such that
$\tr'_1,\dots,\tr'_{i'}$ have height $<d$, and $\tr'_{i'+1},\dots,\tr'_r$ are equal and have height
$d$, for some $i'\ge i+2$. Possibly after iterating the entire procedure, we can assume $i'=r-1$.
\end{proof}

\begin{prop}
\label{prop:Clebsch}
 Let $(\tr_1,\dots,\tr_r)$ and $(\tr'_1,\dots,\tr'_r)$ be two sequences of transpositions of the same length $r$, suppose $\left<\tr_1,\dots,\tr_r\right>=\left<\tr'_1,\dots,\tr'_r\right>=\fS_d$, and also
 assume that there is $\sigma\in\fS_d$ with $\sigma=\tr_1\dots\tr_r=\tr'_1\dots\tr'_r$.
 Then there is a sequence of standard moves transforming $(\tr_1,\dots,\tr_r)$ into
 $(\tr'_1,\dots,\tr'_r)$.
\end{prop}
\begin{proof}
The statement is obvious for $d=2$, so we assume $d\ge3$. If $\vht(\sigma)=d$, by Lemma \ref{lem:onlyhtsigmaJatend}
we can assume, up to changing both sequences of transpositions by suitable sequences of standard moves,
that $\tr_r=\tr'_r=(d,\sigma^{-1}(d))$ and all other transpositions $\tr_i$ and $\tr'_i$ have height
$<d$. We then have $\tr_1\dots\tr_{r-1}=\tr'_1\dots\tr'_{r-1}$ and
$\left<\tr_1,\dots,\tr_{r-1}\right>=\left<\tr'_1,\dots,\tr'_{r-1}\right>=\fS_{d-1}$, so we can conclude by inductive hypothesis
on $r$ and $d$.

If $\vht(\sigma)<d$, by Lemma \ref{lem:anyjJattheend} we may assume $\tr_r=\tr'_r=(d,d-1)$. Let $\sigma':=\sigma\cdot(d,d-1)=\tr_1\dots\tr_{r-1}=\tr'_1\dots\tr'_{r-1}$
and note that $d-1$ and $d$ belong to the same cycle of the cycle decomposition of $\sigma'$.
We claim that $\left<\tr_1,\dots,\tr_{r-1}\right>=\fS_d$. Suppose instead that
$\left<\tr_1,\dots,\tr_{r-1}\right>=\fS_{\fP_1}\times\dots\times\fS_{\fP_\ell}$, for a proper partition $(\fP_1,\dots,\fP_\ell)$ of $[d]$.
Then the elements $d-1$ and $d$ belong to the same cycle of $\sigma'$, which is entirely contained
in one piece of the partition: it follows that $\tr_1,\dots,\tr_r$ also generate
$\fS_{\fP_1}\times\dots\times\fP_\ell$, contradicting the assumption that they generate the entire $\fS_d$.

In a similar way $\left<\tr'_1,\dots,\tr'_{r-1}\right>=\fS_d$. We conclude by induction on $r$.
\end{proof}

\subsection{Consequences for \texorpdfstring{$\fS_d\geo$}{fSdgeo}}
Using the results of the previous subsection we can prove the following properties of $\fS_d\geo$.
\begin{lem}
\label{lem:fSgeococonnectedpairwisedetermined}
The PMQ $\fS_d\geo$ is coconnected and pairwise determined.
\end{lem}
\begin{proof}
Being coconnected is the statement of Lemma \ref{lem:weakClebsch}.
To prove that $\fS_d$ is pairwise determined, we note that
a sequence of elements
$\tr_1,\dots,\tr_r\in(\fS_d\geo)_1$ does not admit a product in $\fS_d\geo$
precisely when the product permutation $\sigma:=\tr_1\dots\tr_r\in\fS_d$ satisfies $N(\sigma)<r$;
on the other hand a product of two transpositions $\tr$ and $\tr'$ is not defined in $\fS_d\geo$
if and only if $\tr=\tr'$.
The statement of the lemma is obvious for $d=2$, so we assume $d\ge3$ henceforth.

If no transposition $\tr_i$ has height $d$, the inductive hypothesis on $d-1$ ensures that there is a sequence
of standard moves transforming the sequence $(\tr_1,\dots,\tr_r)$ into a sequence ending with two equal transpositions. Otherwise assume that some $\tr_i$ has height $d$. We can
use on $(\tr_1,\dots,\tr_r)$ the procedure of Lemma \ref{lem:weakClebsch}, thus transforming
the sequence into a sequence $(\tr'_1,\dots,\tr'_r)$ with $\tr'_1,\dots,\tr'_i$ of height $<d$
and $\tr'_{i+1}=\dots=\tr'_r$ having height $d$, for some $1\le i\le r-1$. If $i\le r-2$ we have reached
a situation with two equal transpositions, so we can assume $i=r-1$.

In particular, if $i=r-1$, we have that $\vht(\sigma)=d$ and $\tr'_r=(d,\sigma^{-1}(d))=\tr_{N(\sigma)}^\sigma$.
It follows that $\sigma':=\sigma\cdot\tr'_r=\tr'_1\dots\tr'_r$ has norm $N(\sigma')=N(\sigma)-1$,
and $\sigma'$, as well as all transpositions $\tr'_1,\dots,\tr'_{r-1}$ can be regarded as permutations
in $\fS_{d-1}$. We conclude by inductive hypothesis on the sequence $\tr'_1,\dots,\tr'_{r-1}$ of elements of norm 1 in $\fS_{d-1}\geo$.
\end{proof}

We can in fact give a complete description of the completion $\widehat{\fS_d\geo}$ of $\fS_d\geo$.

\begin{prop}
 \label{prop:hatfSdgeo}
The complete PMQ $\widehat{\fS_d\geo}$ is the set of all sequences
\[
(\sigma;\fP_1,\dots,\fP_\ell;r_1,\dots,r_\ell),
\]
where $\sigma\in\fS_d$, $(\fP_1,\dots,\fP_\ell)$ is a partition of $[d]$ and $r_1,\dots,r_\ell\ge0$,
satisfying the following properties:
\begin{enumerate}
\item $\sigma\in\fS_{\fP_1}\times\dots\times\fS_{\fP_\ell}$;
\item $r_j\ge 2|\fP_j|-N(\sigma|_{\fP_j})-2$ for all $1\le j\le \ell$;
\item $r_j$ has the same parity as $N(\sigma|_{\fP_j})\in\fS_{\fP_j}$, for all $1\le j\le\ell$.
\end{enumerate}
Conjugation by $(\sigma;\fP_1,\dots,\fP_\ell;r_1,\dots,r_\ell)$ sends
\[
(\sigma';\fP'_1,\dots,\fP'_{\ell'};r'_1,\dots,r'_{\ell'})\mapsto
((\sigma')^{\sigma};\sigma^{-1}(\fP'_1),\dots,\sigma^{-1}(\fP'_{\ell'});r'_1,\dots,r'_{\ell'}).
\]

The product of $(\sigma;\fP_1,\dots,\fP_\ell;r_1,\dots,r_\ell)$ and 
$(\sigma';\fP'_1,\dots,\fP'_{\ell'};r'_1,\dots,r'_{\ell'})$ is the sequence
$(\sigma\sigma';\fP''_1,\dots,\fP''_{\ell''};r''_1,\dots,r''_{\ell''})$, where
$(\fP''_1,\dots,\fP''_{\ell''})$
is the
finest partition which is coarser than both $(\fP_1,\dots,\fP_\ell)$ and $(\fP'_1,\dots,\fP'_{\ell'})$,
and where, for all $1\le j\le \ell$,
\[
r''_i=\sum_{j\,\colon\, \fP_j\subset\fP''_i}r_j+\sum_{j'\,\colon\, \fP_{j'}\subset\fP''_i}r_{j'}.
\]
\end{prop}
In the statement of Proposition \ref{prop:hatfSdgeo}, the partition $(\fP_1,\dots,\fP_\ell)$ is \emph{unordered}, as well as the sequence of numbers $r_1,\dots,r_\ell$; still each piece $\fP_j$ of the partition
is associated with its corresponding number $r_j$.
\begin{proof}[Proof of Proposition \ref{prop:hatfSdgeo}]
By Lemma \ref{lem:fSgeococonnectedpairwisedetermined} the PMQ $\fS_d\geo$ is coconnected,
hence its completion $\widehat{\fS_d\geo}$ can be computed as the completion of $(\fS_d\geo)_1$:
thus an element of $\widehat{\fS_d\geo}$ is an equivalence class of sequences $(\tr_1,\dots,\tr_r)$
of elements in $(\fS_d\geo)_1$, where two sequences are equivalent if they can be transformed
into another by standard moves.
With a sequence $(\tr_1,\dots,\tr_r)$ we can associate the following, which are invariants under standard moves:
\begin{itemize}
 \item the product $\sigma:=\tr_1\dots\tr_r\in\fS_d$;
 \item the unordered partition $(\fP_1,\dots,\fP_\ell)$ of $[d]$, with $\left<\tr_1,\dots,\tr_r\right>=\fS_{\fP_1}\times\dots\times\fS_{\fP_\ell}$;
 \item for each partition piece $\fP_j$, the number $r_j$ of transpositions $\tr_i$ 
 belonging to the symmetric group $\fS_{\fP_j}$.
\end{itemize}
Properties (1)-(3) in the statement of the Proposition ensure that, vice versa, a sequence
$(\sigma;\fP_1,\dots,\fP_\ell;r_1,\dots,r_\ell)$ can be achieved from a suitable sequence
of transpositions $(\tr_1,\dots,\tr_r)$. Indeed, if (1)-(3) are satisfied,
one can let $r:=r_1+\dots+r_\ell$ and define $(\tr_1,\dots,\tr_r)$ as the concatenation of $(\tr_1^\sigma,\dots,\tr_{N(\sigma)}^\sigma)$ with a choice of $(r-N(\sigma))/2$ pairs of equal transpositions $(\tr,\tr)$,
chosen in such a way that, for all $1\le j\le \ell$, there are precisely $r_j-N(\sigma|_{\fP_j})/2$
pairs with $\tr\in\fS_{\fP_j}$; since $r_j-N(\sigma|_{\fP_j})/2\ge|\fP_j|-N(\sigma|_{\fP_j})-1$,
and since the last expression is 1 less than the number of cycles of $\sigma|_{\fP_j}$,
we can also ensure that $\tr_1,\dots,\tr_r$ generate precisely the subgroup $\fS_{\fP_1}\times\dots\times\fS_{\fP_r}$.

Vice versa, if two sequences of transpositions $(\tr_1,\dots,\tr_r)$ and $(\tr'_1,\dots,\tr'_{r'})$ give rise to the same sequence
$(\sigma;\fP_1,\dots,\fP_\ell;r_1,\dots,r_\ell)$, then $r=\sum_{j=1}^\ell r_j=r'$; since
transpositions lying in two different factors $\fS_{\fP_j}$ and $\fS_{\fP_{j'}}$ commute,
up to operating suitable sequences of standard moves we can assume that both 
$(\tr_1,\dots,\tr_r)$ and $(\tr'_1,\dots,\tr'_{r'})$ are concatenations of smaller sequences
in the following way:
\begin{itemize}
 \item for all $1\le j\le \ell$ there is a sequence of transpositions $(\tr_{j,1},\dots,\tr_{j,r_j})$
 in $\fS_{\fP_j}$, and $(\tr_1,\dots,\tr_r)$ is the concatenation of
 $(\tr_{1,1},\dots,\tr_{1,r_1}),\dots,(\tr_{\ell,1},\dots,\tr_{\ell,r_\ell})$;
 \item for all $1\le j\le \ell$ there is a sequence of transpositions $(\tr'_{j,1},\dots,\tr'_{j,r_j})$
 in $\fS_{\fP_j}$, and $(\tr'_1,\dots,\tr'_r)$ is the concatenation of
 $(\tr'_{1,1},\dots,\tr'_{1,r_1}),\dots,(\tr'_{\ell,1},\dots,\tr'_{\ell,r_\ell})$.
\end{itemize}
We can then apply Proposition \ref{prop:Clebsch} to each of the $\ell$ pairs of corresponding
sequences $(\tr_{j,1},\dots,\tr_{j,r_j})$ and $(\tr'_{j,1},\dots,\tr'_{j,r_j})$,
showing that these sequences of transpositions are connected by a sequence of standard moves.
Concatenating, we obtain that also $(\tr_1,\dots,\tr_r)$ and $(\tr'_1,\dots,\tr'_r)$
are connected by a sequence of standard moves.

In this way we have shown that $\fS_d\geo$ is in bijection with the set of sequences
$(\sigma;\fP_1,\dots,\fP_\ell;r_1,\dots,r_\ell)$ satisfying properties (1)-(3).
The description of conjugation and product in light of this bijection is straightforward.
\end{proof}

Let $R$ be a commutative ring.
It follows from Theorem \ref{thm:RQquadratic} that the PMQ-ring $R[\fS_d\geo]$ is isomorphic to
the free associative $R$-algebra with the following generators and relations:

\begin{tabular}{ll}
 \textbf{Generators} & For all $x<y\in[d]$ there is a generator $\sca{xy}=\sca{yx}$.\\
 \textbf{Relations} & $\sca{xy}^2=0$ for all distinct $x,y\in[d]$.\\
  & $\sca{xy}\sca{yz}=\sca{yz}\sca{zx}=\sca{zx}\sca{xy}$ for all distinct $x,y,z\in[d]$.\\
  & $\sca{xy}\sca{zw}=\sca{zw}\sca{xy}$ for all distinct $x,y,z,w\in[d]$.
\end{tabular}

The following Proposition is a reformulation of a Theorem of Visy \cite{Visy}. The proof is taken from
\cite[page 108]{BianchiPhD}.

\begin{prop}
 \label{prop:fSgeoKoszul}
Let $d\ge2$; then the normed PMQ $\fS_d\geo$ is Koszul over any commutative ring $R$.
\end{prop}
\begin{proof}

We will prove that $R[\fS_d\geo]$ admits a Poincar\'{e}-Birkhoff-Witt (PBW) basis: a result by Priddy
\cite{Priddy} then ensures that $R[\fS_d\geo]$ is Koszul; see also Theorem 3.1 in
\cite[Chapter 4]{PP:QuadraticAlgebras}.

We give a total order $\prec$ on the generators of $R[\fS_d\geo]$: let $\sca{xy}$ and $\sca{x'y'}$
be two different generators, with $x<y$ and $x'<y'$; then $\sca{xy}\prec\sca{x'y'}$ if and only if
$y<y'$, or $y=y'$ and $x<x'$.
We give the the lexicographic order, also denoted by $\prec$,
to the set $\set{\pa{\sca{xy},\sca{x'y'}}}$ of
pairs of generators.

For a transposition $\tr=(x,y)\in\fS_d$ we denote by $\sca{xy}$
the corresponding generator of $R[\fS_d\geo]$. More generally,
for $\sigma\in\fS_d$ we denote by $\sca{\sigma}\in R[\fS_d\geo]$
the generator corresponding to $\sigma$; recall that
the elements $\sca{\sigma}$, for varying
$\sigma\in\fS_d$, form a basis of $R[\fS_d\geo]$ as a free $R$-module.

Define $\mathcal{S}\subset\set{\pa{\sca{xy},\sca{x'y'}}}$ as the subset containing pairs
$\pa{\sca{\tr},\sca{\tr'}}$ such that the product $\sca{\tr}\sca{\tr'}\in R[\fS_d\geo]$
cannot be expressed as a linear combination of the form $\sum_{i=1}^m\lambda_i\sca{\tr_i}\sca{\tr'_i}$,
with $\pa{\sca{\tr_i},\sca{\tr'_i}}\prec\pa{\sca{\tr},\sca{\tr'}}$ for all $i$. A PBW-monomial
in the generators $\sca{xy}$ is then a monomial $\sca{\tr_1}\cdot\ldots\cdot\sca{\tr_p}$
with $\pa{\sca{\tr_i},\sca{\tr_{i+1}}}\in \mathcal{S}$ for all $1\leq i\leq p-1$, and our
aim is to prove that PBW-monomials form a basis of $R[\fS_d\geo]$ as a free $R$-module, called
PBW-basis.

By the relations in the presentation of $R[\fS_d\geo]$ it is straightforward to see
that $\pa{\sca{\tr},\sca{\tr'}}\in \mathcal{S}$ if and only if $\vht(\tr)<\vht(\tr')$.

Recall that every permutation $\sigma\in\fS_d$
has a unique monotone decomposition $\tr_1^\sigma \dots\tr_{N(\sigma)}^\sigma$
with $\vht(\tr_1)<\dots <\vht(\tr_{N(\sigma)})$; vice versa
every product $\tr_1\cdot\dots\tr_p$, with $\vht(\tr_1)<\dots <\vht(\tr_p)$ gives a permutation
$\sigma\in\fS_d$ of norm $p$.

This shows that PBW-monomials form precisely the standard basis of elements $\sca{\sigma}$ of
$R[\fS_d\geo]$.
\end{proof}

\appendix
\section{A brief discussion on partially multiplicative racks}
\label{sec:racks}
A \emph{rack} with unit is as a
set $\fR$ with a marked element $\one\in\fR$, called unit, and a binary operation $\fR\times\fR\to\fR$,
called conjugation and denoted $(a,b)\mapsto a^b$, such that all properties of Definition \ref{defn:quandle}
are satisfied, except possibly property (3). A morphism of racks is required to preserve
the unit and the conjugation.

A partially multiplicative rack (PMR) is then a set $\fR$ with compatible structures
of partial monoid and rack, i.e. satisfying all properties of Definition \ref{defn:PMQ}
with the word ``quandle'' replaced by the word ``rack''. PMRs form a category
$\PMR$, with morphisms those maps of sets that preserve unit, conjugation and partial
multiplication.

All definitions from Subsection \ref{subsec:basicsPMQ} extend naturally to the setting of racks.
We have a fully faithful inclusion of categories $\PMQ\subset\PMR$; the next example
shows that it is not an equivalence.
\begin{ex}
 Let $\fR=\set{\one,a,b}$ be a rack with conjugation given by $a^b=a^a=b$ and $b^b=b^a=a$.
 Then $\fR$ is not a quandle. If we consider $\fR$ as a PMR with trivial multiplication,
 we obtain an example of a PMR which is not a PMQ.
\end{ex}

In fact, all definitions and results in Sections \ref{sec:pmqbasic}, \ref{sec:normedPMQ}, \ref{sec:barconstructions} and \ref{sec:simplhur} can be generalised to the context of
PMRs. In particular:
\begin{itemize}
 \item every PMR $\fR$ can be completed to a \emph{complete PMR} (or multiplicative rack) $\hfR$;
 \item for every complete PMR $\hfR$ and every category $\bfA$ we have a category $\XA(\hfR)$ of $\hfR$-crossed object in $\bfA$, which is a functor $\hfR\borel\hfR\to\bfA$, where $\hfR\borel\hfR$ is defined precisely
as in Definition \ref{defn:hQborel}, replacing $\hQ$ by $\hfR$;
 \item if $\bfA$ is (braided) monoidal, then $\XA(\hfR)$ is also (braided) monoidal;
 \item for an augmented PMR $\fR$ we can define a simplicial Hurwitz space $\Hur^{\Delta}(\fR)$,
 arising as the difference of the geometric realisations of the bisimplicial sets
 $\Arr(\fR)$ and $\NAdm(\fR)$, which are defined precisely as in Definitions \ref{defn:ArrQ}
 and \ref{defn:NAdm};
 \item for an augmented PMR $\fR$ the double bar construction $B_{\bullet,\bullet}(R[\fR],R,\epsilon)$
 gives rise, after passing to the associated total chain complex, to the cellular chain complex
 of the pair $(|\Arr(\fR)|,|\NAdm(\fR)|)$.
\end{itemize}

The only results which do not admit a direct generalisation to PMRs are those in Section \ref{sec:freegroups}, in particular Theorem \ref{thm:FQfreePMQ}. In the following we discuss what
complications arise.
First, we note that the fully faithful inclusion of categories $\PMQ\subset\PMR$ has a right adjoint.
\begin{defn}
 \label{defn:RPMQ}
 Let $\fR$ be a PMR. An element $a\in\fR$ is \emph{quandle-like} if $a^a=a$.
 We denote by $\fR^{\mathrm{PMQ}}\subset\fR$ the subset of quandle-like elements.
\end{defn}
\begin{lem}
\label{lem:rightadjoint}
 Let $\fR$ be a PMR. Then $\fR^{\mathrm{PMQ}}$ inherits from $\fR$ a structure of PMQ, such that
 the inclusion $\fR^{\mathrm{PMQ}}\subset\fR$ is a map of PMRs. Moreover for any PMQ $\Q$,
 any map of PMRs $\psi\colon\Q\to\fR$ has image contained in $\fR^{\mathrm{PMQ}}$.
\end{lem}
\begin{proof}
 It is clear that $\one$ is quandle-like, so it belongs to $\fR^{\mathrm{PMQ}}$ and provides the unit.
 For all quandle-like $a,b$ we have to check the following:
 \begin{itemize}
  \item $a^b$ is quandle-like: indeed $(a^b)^{a^b}=(a^a)^b$ because $\fR$ is a PMR, and $(a^a)^b=a^b$ since $a$ is quandle-like;
  \item if $ab$ is defined in $\fR$, then it is quandle-like: indeed $(ab)^{ab}=a^{ab}b^{ab}=(a^a)^b(b^a)^b=(a^a)^b(b^b)^{a^b}$ because $\fR$ is a PMR, the last expression is equal
  to $a^b b^{a^b}$ because $a$ and $b$ are quandle-like, and $a^b b^{a^b}=ba^b=ab$ because $\fR$
  is a PMR.
 \end{itemize}
 This shows that $\fR^{\mathrm{PMQ}}$ is a PMQ, and inclusion $\fR^{\mathrm{PMQ}}\subset\fR$ is a map of PMRs.
 If $\Q$ is any PMQ and $\psi\colon\Q\to\fR$ is a map of PMRs, then the equality $a^a=a$
 for $a\in\Q$ implies the equality $\psi(a)^{\psi(a)}=\psi(a)$ in $\fR$, hence $\psi(a)\in\fR^{\mathrm{PMQ}}$.
\end{proof}
The previous lemma shows that the assignment $\fR\mapsto \fR^{\mathrm{PMQ}}$ gives the right
adjoint $(-)^{\mathrm{PMQ}}\colon\PMR\to\PMQ$ to the inclusion $\PMQ\subset\PMR$.

In \cite{Bianchi:Hur2} we will introduce a ``coordinate-free'' definition of Hurwitz spaces:
the definition is quite general, and it includes, for each PMQ $\Q$, the definition
of a space $\Hur((0,1)^2;\Q)$, containing configurations $P\subset(0,1)$ with the additional
datum of a monodromy $\psi$, defined on certain loops of $\CmP$ and with values in $\Q$.

More precisely, a configuration in  $\Hur((0,1)^2;\Q_+)$ takes the form of a pair $(P,\psi)$,
where $P\subset(0,1)^2$ is a finite subset, and $\psi\colon\fQ(P)\to\Q$ is a map of PMQs.
Here we define $\fQ(P)\subset\pi_1(\CmP,*)$
to be the union of $\set{\one}$ and all conjugacy classes corresponding to simple closed
curves in $\CmP$ spinning clockwise around precisely one point of $P$. The set $\fQ(P)$
is a PMQ with trivial multiplication. As in Subsection
\ref{subsec:simplhur}, we let  $*=-\sqrt{-1}\in\CmP$ be the basepoint for the fundamental group of $\CmP$.

Suppose now that $\fR$ is any PMR, and let us define
the set $\Hur((0,1)^2;\fR)$, in the most straightforward way,
as the set of pairs $(P,\psi)$, where
$P\subset(0,1)^2$ is a finite subset and
$\psi\colon\fQ(P)\to\fR$ is a morphism of PMRs; then by Lemma \ref{lem:rightadjoint}
the image of $\psi$ is contained in $\fR^{\mathrm{PMQ}}$, and in fact the set
$\Hur((0,1)^2;\fR)$ is in bijection with the set  $\Hur((0,1)^2;\fR^{\mathrm{PMQ}})$.
This means that only the quandle-like part of $\fR$ is involved in the definition of
the set $\Hur((0,1)^2;\fR)$, or, in other words, the definition of the set
$\Hur((0,1)^2;\fR)$ is only interesting when $\fR$ is a PMQ.

\bibliography{Bibliography1.bib}{}
\bibliographystyle{alpha}

\end{document}